\documentclass[11pt,reqno]{amsart}

\usepackage[latin9]{luainputenc}
\usepackage{amsfonts}
\usepackage{amsmath}
\usepackage{amstext}
\usepackage{amsthm}
\usepackage{amssymb}
\usepackage{braket}
\usepackage{bm}
\usepackage{cancel}
\usepackage{caption}
\usepackage{color}
\usepackage{enumerate}
\usepackage{enumitem}
\usepackage{esint}
\usepackage{extarrows}
\usepackage[a4paper,margin=1in,footskip=0.25in]{geometry}
\usepackage{graphicx}
\usepackage{longtable}
\usepackage{mathrsfs}
\usepackage{mathtools}
\usepackage{dsfont}
\usepackage{mdframed}
\usepackage{stmaryrd}
\usepackage{soul}
\usepackage{subcaption}
\usepackage{tikz}
\usepackage{tikz-cd}
\usepackage{xcolor}
\usepackage{ulem}
\usepackage{hyperref}
\usetikzlibrary{patterns,arrows.meta}

\newtheorem{Def}{Definition}[section]
\newtheorem{Lemma}[Def]{Lemma}
\newtheorem{Theorem}{Theorem}
\newtheorem{Cor}[Def]{Corollary}
\newtheorem{Prop}[Def]{Proposition}

\newtheorem{Rem}[Def]{Remark}

\newtheorem{maintheorem}{Theorem}

\DeclareMathOperator{\Ker}{Ker} %
\def \b {\beta}

\def \k {\kappa}
\def \e {\varepsilon}

\def \l {\lambda}

\def \t {\tau}
\def \th {\theta}
\def \o {\omega}

\def \L {\Lambda}
\def \Th {\Theta}

\def \cE {\mathcal{E}}

\def \cN {\mathcal{N}}

\newcommand{\N}{\ensuremath{\mathbb{N}}}   
\newcommand{\Z}{\ensuremath{\mathbb{Z}}}   
\newcommand{\R}{\ensuremath{\mathbb{R}}}   
\newcommand{\C}{\ensuremath{\mathbb{C}}}   

\renewcommand{\S}{\ensuremath{\mathbb{S}}} 

\def \one {{\mathds{1}}}

\def \p {\partial}

\def \dt  {\, \mbox{d}t}

\def \ddt  {\frac{\mbox{d\,\,}}{\mbox{d}t}}

\def \dd  {\mbox{d}}

\def \xmax{x_{+}}

\def \Re {\mathrm{Re}}
\def \Im {\mathrm{Im}}

\def \indices {\{1,\cdots,N\}}
\renewcommand{\mod}{\mathrm{mod} \text{ }}

\title{Graph-Induced Rotational Twisted States in Systems of Identical Oscillators}
\author[Nastassia Pouradier Duteil]{Nastassia Pouradier Duteil}
\address{Sorbonne Universit\'e, Inria, Universit\'e Paris-Diderot SPC, CNRS, Laboratoire Jacques-Louis Lions, Paris, France}
\email{nastassia.pouradier\_duteil@sorbonne-universite.fr}

\author[David Poyato]{David Poyato}
\address{Departamento de Matem\'atica Aplicada and Research Unit ``Modeling Nature'' (MNat), Facultad de Ciencias, Universidad de Granada, 18071 Granada, Spain}
\email{davidpoyato@ugr.es}

\author[David N Reynolds]{David N Reynolds}
\address{University of Warsaw, Institute of Applied Mathematics and Mechanics, Banacha 2, 02-097
Warszawa, Poland}
\email{d.reynolds@uw.edu.pl}

\date{November 2026}


\begin{document}

\maketitle

\begin{abstract}
    In this article, we study a new class of collective motion for identical coupled Stuart-Landau  oscillators on graphs.
This model was previously known to converge to the synchronized state for a certain class of initial data. Here, we show that when the interaction matrix is circulant, there exists another class of attractors, in which the particles are uniformly distributed on a circle, reminiscent of the \textit{twisted states} known for the Kuramoto model, but which rotate around the origin. However, contrary to the classical Kuramoto case, here the rotation is caused by the asymmetry of the graph structure, and not by the natural frequencies.
 We identify conditions for the existence and local stability of the \textit{rotational twisted states}, both in the Stuart-Landau model and in the Kuramoto model. We also provide sufficient conditions for the existence and stability of the synchronized state, which can co-exist with the rotational twisted state in a certain parameter region. We provide a generalization of the class of interaction matrices able to generate rotational twisted states via leader-follower interaction matrices. We show that heterogeneous rotational twisted states can also exist for a system of heterogeneous oscillators, attracted to different individual target amplitudes. This study is accompanied by numerical simulations that illustrate the possible behaviors of the system, which also include metastable dynamics and a chimera state.
    
\end{abstract}

\section{Introduction}

Interacting particle systems refer to a class of models used to study self-organization in many different application fields, as they can display a wide range of collective behavior.
For instance, 
depending on the support of the interaction function, Hegselmann-Krause-type models for opinion dynamics can exhibit convergence to consensus (agreement of all opinions), clustering or polarization \cite{HegselmannKrause02}. Depending on the coupling strength, Kuramoto-type models for coupled oscillators are known to display either synchronized or incoherent states \cite{strogatz2000}.
Second-order models are often used to study animal behavior and allows one to reproduce patterns such as flocking (i.e. alignment of the velocities) as in the Cucker-Smale model \cite{CuckerSmale07}, or the Vicsek model \cite{Vicsek95}.


In particular, fish schools are known to display two main types of collective behavior: flocking (also called \textit{schooling} for fish), in which all individuals swim at the same velocity, or milling, in which the group rotates around its center of mass. Importantly, a fish school's capacity to form -- and to transition between -- collective patterns is related to its survival capacity \cite{Ashraf16,Calovi14}.
The transition between flocking and milling has been studied in a laboratory setting, and it was shown that it can be induced by varying the ambient lighting \cite{Lafoux23}. These two phases have also been reproduced numerically \cite{Couzin02}, using second-order models of the form: 
\begin{equation}\label{eq:fishmodel}
\begin{cases}
\displaystyle\ddt x_j = v_j\\
\displaystyle \ddt v_j =  c_{\mathrm{AR}}\nabla_{x_j} \sum_{l=1}^N U(|x_j-x_l|) +  c_{\mathrm A}\sum_{l=1}^N \phi(x_j,x_l)(v_l-v_j) +c_{\mathrm{SP}}\left(1-\frac{|v_j|^2}{\beta^2}\right)v_j,
\end{cases}
\end{equation}
in which $(x_i,v_i)_{i\in\{1,\cdots,N\}}$ represent the positions and velocities of $N$ interacting individuals.
The velocity evolution in Model \eqref{eq:fishmodel} is composed of three terms representing the main forces contributing to each individual's acceleration: attraction-repulsion, alignment, and self-propulsion/ friction. Attraction (usually active at large distance), repulsion (at short distance) and alignment (at intermediate distance) are so-called ``social forces'' modeling the influence of each individual on the others, while the self-propulsion/friction force models the desire of each individual to swim at a target speed. Mathematically, conditions for convergence to flocking are well-known for the Cucker-Smale model \cite{CuckerSmale07,HaLiu09}, in which only the alignment term is considered (i.e. $c_{\mathrm{SP}}=c_{\mathrm{AR}}=0$ in \eqref{eq:fishmodel}). Interestingly, flocking and milling have been proven to co-exist as stable solutions of \eqref{eq:fishmodel} in a model containing the attraction-repulsion term and the self-propulsion term (i.e. $c_\mathrm{A}=0$) in \cite{CarrilloDOrsognaPanferov09}.
 
In the present work, we ask ourselves whether the alignment and the self-propulsion terms (i.e. $c_\mathrm{AR}=0$ in Model \eqref{eq:fishmodel}) could also produce solutions of both flocking and milling types. This system was first studied in \cite{Ha_2010}, providing flocking results under a narrow class of initial data. In \cite{LRS} the authors proved an unconditional flocking result for absolute kernels in the alignment dominant regime $( c_\mathrm{AR} \gtrsim c_\mathrm{SP})$, while for kernels satisfying the classical heavy-tail condition, flocking can only be shown for sectorial solutions which take initial velocities on one side of a hyperplane. This leaves open the possibility that in weaker communication regimes where the self-propulsion/friction force dominates $(c_\mathrm{SP}\gtrsim c_\mathrm{AR})$ that milling phenomena could be recovered.

An important note is that the interaction function $\phi$ used in \eqref{eq:fishmodel} modulates the interaction between two individuals by the distance between them: classically, $\phi$ is taken to be a positive decreasing function, to model the fact that the closer two individuals are, the more likely they are to align their velocities. With such an isotropic interaction function and longitudinal forcing term, there can be no emergent milling behavior, as the symmetry precludes a stable rotational direction from being established.

Anisotropic interaction functions have also been proposed in the literature, and have been shown to lead to different kinds of collective behavior \cite{CarrilloChoiHauray2018,CristianiFrascaPiccoli2011}. From the modeling viewpoint, anisotropic interaction functions allows for taking into account the cone of vision of individuals, for instance by introducing a blind zone in the back. 
In \cite{CristianiFrascaPiccoli2011}, anisotropic attraction-repulsion has been shown numerically to lead to interesting collective pattern such as line formation or V-shape formation. 

Here, we propose to replace the interaction function by an interaction matrix $(A_{jl})_{j,l\in\{1,\cdots,N\}}$, modeling the existence of an underlying graph of communication. This choice renders the particle system non-exchangeable: the interaction between individuals no longer depends solely on their respective positions and velocities, but also on their identity (i.e. their labels). We write the non-exchangeable model with self-propulsion and alignment as:

\begin{equation}\label{eq:fishmodel2}
\begin{cases}
\displaystyle\ddt x_j = v_j\\
\displaystyle \ddt v_j = \sum_{l=1}^N A_{jl}(v_l-v_j) +\frac{1}{\tau}\left(1-\frac{|v_j|^2}{\beta^2}\right)v_j.
\end{cases}
\end{equation}

Originally inspired by the field of statistical mechanics, the first efforts at modeling self-organization in interacting particle systems made the central assumption that particles are \textit{exchangeable} (i.e. identical). More specifically, the interaction between any two particles is classically viewed as depending only on their respective positions and velocities, and not on their \textit{identities}.
This is a reasonable assumption in many applications, and many celebrated models assuming exchangeability have been validated experimentally \cite{Aoki82, Ballerini08}. However, there is evidence that in some systems, there is a large variability within the population, both in terms of phenotypical traits and of preferred interactions with a few select individuals. For instance, a mean coefficient of size variation of 15.7\%, was reported in a survey conducted for a fish population in a Canadian lake \cite{KrauseGodinBrown1996}. Neural networks are another example of interacting particle systems in which the interactions are not based on the proximity of neurons, but rather on their connectedness in the network \cite{GuRohling19}. 

In this perspective, a wealth of new works has emerged with the aim of removing this assumption, in order to study the collective behavior of a \textit{non-exchangeable} particle system. 
In the case of first-order Hegselmann-Krause-type models, this problem has attracted a large amount of attention, and convergence to consensus was still proven to hold in the non-exchangeable setting, under various connectivity conditions on the underlying graph of interactions, see e.g. \cite{B22, Moreau2005, MT14, PROSKURNIKOV2018166}. 
More complex types of collective behavior were also exhibited in the case of mixed positive and negative interaction coefficients, such as the periodic ``dancing equilibrium''~\cite{McQuadePiccoliPouradierDuteil19}.

Under the modeling assumption $\phi(x_j,x_k)=A_{jk}$, the system decouples and we need only investigate the velocity equation, where the positions can be obtained by integrating. Further, since the phenomena we are investigating is the bistability of \textit{schooling} and \textit{milling}, it is natural to consider the evolution of particles in the complex plane
\begin{equation}\tag{SL}\label{eq:mainmodel}
\displaystyle \ddt z_j = \sum_{l=1}^N A_{jl}(z_l-z_j) +\frac{1}{\tau}\left(1-\frac{|z_j|^2}{\beta^2}\right)z_j, \qquad j\in\{1,\cdots,N\}, \quad z_j\in \mathbb{C}.
\end{equation}

The model \eqref{eq:mainmodel} is interesting in its own right as a system of coupled identical Stuart-Landau oscillators \cite{millan2025synchronizationcoupledstuartlandauoscillators}. As a model of synchronization it has been used extensively in the neuroscience community to model mesoscopic brain activity with applications to epilepsy and Alzheimer's disease as well as recent applications to machine learning and nonlinear opinion dynamics \cite{PhysRevLett.93.104101,PhysRevE.75.056206,RevModPhys.90.031001,PAD2024,reynolds2025consensuspolarizationoptimizationmean,Reynolds_2025,Sahoo_2019,10.1371/journal.pcbi.1010662,Thome_2026,NMDDSSCKLHG2023,10.1145/3774904.3792675}.

The Stuart-Landau oscillator is the normal form of the super-critical Hopf-bifurcation and is the simplest description of a limit-cycle oscillator \cite{PANTELEY2015645}. The seminal work of Yoshiki Kuramoto \cite{10.1007/BFb0013365} considered the weak-coupling limit (phase reduction) of coupled Stuart-Landau oscillators, producing what is now known as the Kuramoto model of phase-coupled oscillators:
\begin{equation}\tag{K}\label{eq:Kur}
    \displaystyle \ddt \th_j = \sum_{l=1}^N A_{jl}\sin(\th_l-\th_j), \qquad j \in \{1,\cdots,N\}, \quad \th_j\in \mathbb{S}.
\end{equation}
Formally, one obtains \eqref{eq:Kur} in the limit $\tau\to 0$ in \eqref{eq:mainmodel}.
Providing an exhaustive list of references for the Kuramoto model would be time and space prohibitive, however we refer to the following review papers for the requisite background \cite{acebron2005,rodrigues2016,strogatz2000}.

One may immediately notice the absence of natural frequencies within \eqref{eq:mainmodel} and \eqref{eq:Kur}, which usually produce the rotational motion seen in these synchronization models. In this work, we consider identical oscillators $(\Omega_j=0)$, to see how rotational (milling) states can arise from the nonexchangeability introduced in the graph structure. The Kuramoto system with identical natural frequencies has come to be known as the homogeneous Kuramoto model.

The homogeneous Kuramoto model on graphs has seen much attention, in particular in the study of existence of multiple stable equilibria of the system. It is not hard to see that the fully synchronized state $\th_j=\bar{\th}$ for all $j\in\indices$ is a stable equilibrium for any connected graph. Taylor \cite{taylor2012} first noticed that if the network connectivity was such that if each node was on average connected to at least 93.95\% of the other nodes, then this could be the only stable equilibrium of \eqref{eq:Kur}. This sparked a subsequent research program which attempted to improve this result. Ling  et  al \cite{ling2019}  improved  this  bound  to  79.29\%,  Lu and Steinberger \cite{lu2020}  to  78.89\% and  Kasabov et al. \cite{kassabov2019} to  75\%  which  they  conjecture  to  be  the  exact  threshold. Meanwhile Yoneda et al. \cite{yoneda2021} proved a 68.38\% best lower bound threshold.  The work of Townsend et al. \cite{townsend2017} also gave a previous best lower bound of 68.28\%, but more importantly studied circulant graph structures.

The circulant graph structure is critical for the following work. Besides the work of Townsend et al., Medvedev and Tang \cite{MedTang2015} investigated the connection between circulant graph structures in the Kuramoto model and Cayley graphs, providing stability for the so-called twisted states of the Kuramoto model. Twisted states occur when the oscillators are uniformly distributed in such a way that the symmetry of the graph provides a stable fixed point \cite{Berglund_2025,GMS2021,Med2}. However, a key element to the works thus far is that they all assume that the underlying graph is symmetric. We remove this assumption and study systems \eqref{eq:mainmodel} and \eqref{eq:Kur} with an underlying directed circulant graph.

The asymmetry in the graph structure leads to what we call \textit{rotational twisted states} where the oscillators, still uniformly distributed, begin to rotate at a common speed determined by the degree of asymmetry. The parameter $\tau$ in \eqref{eq:mainmodel} is known as the relaxation time for a Stuart-Landau oscillator and determines how long it takes for an uncoupled oscillator to return to its limit-cycle. The relationship between $\tau$ and the graph coupling plays a critical role in the existence of rotational twisted states. Indeed, as $\tau \to \infty$ in \eqref{eq:mainmodel} we see that the consensus mechanism dominates and the only stable asymptotic state should be full synchronization. This fact is sometimes called \textit{practical synchronization} and was shown to hold for Stuart-Landau oscillators by Panteley et al. \cite{Panteley01022020} for generic graph structures.

We begin our investigation with the $\tau \to 0$ regime yielding the homogeneous Kuramoto model \eqref{eq:Kur}. We provide a sharp criterion for the stability of rotational twisted states. For small but positive values of $\tau$, this criterion carries over to the Stuart-Landau system \eqref{eq:mainmodel} up until a finite critical value $\tau^*$ at which point a phase transition occurs and the rotational twisted states cease to exist. Beyond this threshold, dynamics tend to eventually converge to the fully synchronized state, however, we also uncover two metastable behaviors. In one case the dynamics retain their rotational structure, converging nearly to zero, before slowly converging to the synchronized state, while in the other, the rotational structure breaks down (before the critical value $\tau^*$) without radial constriction and the dynamics then provide a swarming-like behavior as they converge to the synchronized state. These types of metastable behaviors are seen in the Transformer model \cite{geshkovski2024dynamic,geshkovski2025mathematical} and have recently been taken advantage of in training neural networks in order to mitigate the oversmoothing problem \cite{10.1145/3774904.3792675}.

We also investigate including heterogeneity into the $\beta$ parameter, indicating a different desired amplitude intrinsic to each oscillator. Through a perturbative analysis we can provide a more diverse structure for rotational twisted states. Lastly, we explore the question of whether the circulant structure of the interaction matrix is necessary for the existence of rotational states. As a first step towards answering this question, we investigate a larger class of interaction matrices modeling leader-follower dynamics, which provides another phase transition between the stable rotational twisted state and a kind of chimera state where the leaders rotate and the followers fail to adopt the rotational behavior of the leaders.

The outline of the paper is as follows. Section \ref{s:2} provides the relevant preliminary material, notation, and statement of main results. Section \ref{s:3} proves existence and the stability criteria for the Kuramoto rotational twisted states. Section \ref{s:4} establishes existence of the rotational twisted states for the Stuart-Landau system, identifying the phase transition at $\tau^*>0$ for which the state ceases to exist. Stability for a potentially smaller but positive $\tau^{**}$ is also provided. Section \ref {s:5} proves an unconditional synchronization result for $\tau$ large enough. Section \ref{s:6} introduces the leader-follower dynamics providing critical thresholds for the transition between the leader-follower rotational twisted state and the chimera state. Section \ref{s:7} is dedicated to a numerical investigation of each of the regimes with particular emphasis given to metastable rotational twisted states.

\section{Preliminaries and main results}\label{s:2}

\subsection{Notations and definitions}

We first note that as we will be working in the complex plane we let $i=\sqrt{-1}$ be the imaginary unit, and avoid utilizing $i$ as an index. Further, we let $\S=\R/2\pi\Z$ denote the one-dimensional periodic domain. We denote by $\mathbb{R}_+$ the set of non-negative real numbers (including $0$). We denote by $\N$ the set of positive integers (excluding $0$). For all integers $k\in\Z$ and $N\in\N$, we let $k[N]:=k\, \mod N\in \{0,\cdots,N-1\}$ denote the remainder of the Euclidean division of $k$ by $N$, and $k \langle N \rangle :=(k-1) [N]+1 \in \{1,\cdots,N\}$. Throughout this article, $\phi_j:=\frac{2\pi j}{N}$, $j\in\indices$ will denote the angles corresponding to the $N$-th roots of unity. For any  matrix $A\in M_{N_1\times N_2}(\mathbb{R})$ with real entries its transpose is denoted by $A^T\in M_{N_2\times N_1}(\mathbb{R})$. Similarly, for any matrix $A\in M_{N_1\times N_2}(\mathbb{C})$ with complex entries its transpose-conjugate is denoted $A^*=\bar A^T\in M_{N_2\times N_1}(\mathbb{C})$. If the matrix is square we further denote this space $M_N(\R)$ or $M_N(\C)$ for simplicity.

The stability of asymptotic states of \eqref{eq:mainmodel} and \eqref{eq:Kur} is the predominant focus of this work. We first define the two states to which the dynamics of each system converges.

\begin{Def}\label{sync}
    A solution of \eqref{eq:Kur} is at a \emph{synchronized state} at time $t$ if there exists $\theta^\infty\in\S$ such that for all $j\in\indices$
    \begin{align*}
        \th_j(t)&=\th^\infty.
    \end{align*}
    Analogously, a solution of \eqref{eq:mainmodel} is at a \emph{synchronized state} at time $t$ if there exists $z^\infty\in\C$ such that for all $j\in\indices$, 
    \begin{align*}
        z_j(t)&=z^\infty.
    \end{align*}
\end{Def}

\begin{Def}\label{RTS}
    For $q\in\Z$, a solution of \eqref{eq:Kur} is at a \emph{$q$-rotational twisted state} if there exists $\omega\in\R$ such that for all $j\in\indices$,
    \begin{align*}
        \th_j(t)&=\frac{2\pi jq}{N}+\th(t) 
    \end{align*}
    where $\theta$ satisfies 
$\displaystyle\ddt \th(t)=\omega$.
The case $q=0$ returns the synchronized state while in the case $q=1$, we will simply use the term \emph{rotational twisted state}.
    
    Analogously, a solution of \eqref{eq:mainmodel} is at a \emph{rotational twisted state} if there exists $\omega\in\R$ and $R\in\R_+$ such that for all $t\geq 0$,  for every $j\in\indices$, $z_j(t)=r_j(t)e^{i\th_j(t)}$ with
    \begin{align*}
        \th_j(t)=\phi_j+\th(t), \qquad
        r_j(t)=R, 
    \end{align*}
    in which $\theta$ satisfies 
$\displaystyle\ddt \th(t)=\omega$, and $\displaystyle \phi_j:=\frac{2\pi j}{N}$.

    Moreover, a solution of \eqref{eq:mainmodel} is at a \emph{heterogeneous rotational twisted state} if there exists $\omega\in\R$, $(R_j)_{j\in\indices}\in(\R_+)^N$ and $(\theta_j^\infty)_{j\in\indices}\in\S^N$ such that for all $t\geq 0$,  for every $j\in\indices$, $z_j(t)=r_j(t)e^{i\th_j(t)}$ with
    \begin{align*}
        \th_j(t)=\theta^\infty_j+\th(t), \qquad
        r_j(t)=R_j,
    \end{align*}
    in which $\theta$ satisfies 
$\displaystyle\ddt \th(t)=\omega$.
\end{Def}

It turns out that the class of connectivity matrices known as circulant matrices lends itself to producing the \textit{rotational twisted state}.
 
\begin{Def}\label{circulant}
    A matrix $A\in M_N(\R)$ is said to be \textit{circulant} if there exists $a=(a_0,\cdots,a_{N-1})\in\R^N$ such that for each $j,l\in\indices$, 
    \[A_{jl}=a_{(l-j) [N]}.\]
The whole matrix is then generated by the elements in the first row $(a_0,a_1,\cdots,a_{N-1})$, and the remaining rows are obtained by shifting, that is,
\[
    A=\begin{pmatrix}
        a_0 & a_1 & \ldots & a_{N-2} & a_{N-1} \\
        a_{N-1} & a_0 & a_1 & \ldots & a_{N-2} \\
        \vdots & \ddots & \ddots & \ddots & \vdots\\
         \vdots & & \ddots & \ddots & \vdots\\
         & & & & a_1\\
         a_1 & \ldots &\ldots & a_{N-1} & a_0
    \end{pmatrix}.
\]
\end{Def}
An important note about the circulant matrix here is that we do not require symmetry, in fact the asymmetry is precisely what produces the rotational dynamics. Despite the lack of symmetry an important fact is that the \textit{degree} of each node is identical.
\begin{align*}
    \mathrm{deg}(j)=d_j=\sum_{l\neq j} A_{jl}=\sum_{l+1\neq j+1} A_{j+1,l+1}=d_{j+1}\equiv d.
\end{align*}
Further, circulant matrices are easily diagonalizable. The transition matrix is provided by the Fourier modes so that for $\phi_j=\frac{2\pi j}{N}$, the normalized eigenvectors
are orthogonal, and provided by
\begin{align}\label{circ-eigenv}
    w_k=\frac{1}{\sqrt{N}}(1, e^{ik\phi_1},e^{ik\phi_2},\cdots, e^{ik\phi_{N-1}})^T, \ \ \ k \in\{0,\cdots,N-1\},
\end{align}
with the eigenvalues provided by
\begin{align}\label{circ-eigenvalues}
    \l_k=a_0+a_1e^{ik\phi_1}+a_2e^{ik\phi_{2}}+...+a_{N-1}e^{ik\phi_{N-1}}, \ \ \ k\in\{0,\cdots,N-1\}.
\end{align}

We also recall that a matrix 
$A\in M_N(\R_+)$ characterizes a directed weighted graph $G_N = (V,E,W)$ in the following way: 
\begin{itemize}
    \item The $N$ vertices $V$ of the graph are given by the set of indices $\{1,\cdots,N\}$;
    \item The directed edges $E$ are defined by the non-zero coefficients of $A$, with $(i, j)\in E$ if and only if $A_{ij}>0$;
    \item The edge weight of each edge $(i,j)\in E$ is given by the coefficients $A_{ij}>0$.
\end{itemize}
A \textit{path} of length $q\leq N$ is a sequence of $q$ nodes $(j_k)_{k\in\{1,\cdots,q\}}$ satisfying $A_{j_k,j_{k+1}}>0$ for all $k\in\{1,\cdots,q-1\}$. If $j_1=j_q$, then the path is called a \textit{loop}.

The graph is said to have a \textit{spanning tree} if there exists a node $j_0\in\indices$ such that for every $k\in \indices\setminus j_0$, there exists a path from $j_0$ to $k$. The node $j_0$ is referred to as the \textit{root} node.

If for every $(j,k)\in\indices$, there exists a path $(n_m)_{m\in\{1,\cdots,l\}}$ satisfying $n_1=j$ and $n_l=k$, then the graph is said to be \textit{strongly connected}.

An example of a graph associated with a circulant matrix is given in Figure  \ref{Fig:Network}.

\begin{Rem}\label{Rem:spanningtree}
Note that if $A\in M_N(\mathbb{R})$ is a circulant matrix, its associated graph $G(A)$ is strongly connected if and only if it has a spanning tree. Indeed, due to the circulant structure of the matrix, every node can play the role of the root node, which means that there is a path from every node to every other node.
\end{Rem}

We also present here a cone condition for circulant matrices which will provide the sharp cutoff for stability of the \textit{twisted rotational states} associated to a circulant matrix.

\begin{Def}[Cone Condition]\label{def:cone}
     Let $A\in M_N(\mathbb{R})$ be a circulant matrix generated by the vector $a=(a_0,\cdots,a_{N-1})\in\R^N$. Let $\phi_j=\frac{2\pi j}{N}$ for $j\in\indices$ be the angles of the $N$-th roots of unity. 
     $A$ is said to satisfy the cone condition if the off-diagonal elements given by $\bar{a}=(a_1, \ldots,a_{N-1}) \in \mathbb{R}^{N-1}$ are in the open cone
    \begin{align}\tag{C}\label{cone-condition}
        \mathcal{C}_N:=\{\bar{a}\in \mathbb{R}^{N-1} : (W\bar{a})_l>0 \ \text{ for all } l\in\{1,\cdots,N-1\}\},
    \end{align}
    where $W\in M_{N-1}(\R)$ is given by
    \begin{align*}
        W_{kj}=\cos(\phi_{j})(1-\cos(k\phi_{j})), \ \ \ \text{ for all } j,k\in\{1,\cdots, N-1\}.
    \end{align*}
\end{Def}

\begin{Rem}
The fact that a circulant matrix satisfies the cone condition of Definition \ref{def:cone} implies that the corresponding directed graph is strongly connected. Suppose that the circulant graph does not have a spanning tree. Then there exists at least one sub-cycle of length $n$ where $nj=N$ for an integer $j$. Then $\cos(n\phi_j)=1$ implying $(W\bar{a})_n=0$ and the cone condition fails. Thus by the contrapositive a circulant graph satisfying the cone condition has a spanning tree, which in turn implies that it is strongly connected (see Remark \ref{Rem:spanningtree}).
\end{Rem}

\subsection{Main results}
With the definitions in hand, we briefly present the connection between the relevant models discussed here. The Kuramoto model \eqref{eq:Kur} is formally obtained in the weak coupling regime by letting $\tau\to 0$ in \eqref{eq:mainmodel} as in \cite{Kuramoto75}.  Indeed, considering the amplitude and phase dynamics of \eqref{eq:mainmodel}
\begin{align}
    \ddt r_j &=\sum_{l=1}^NA_{jl}(\cos(\th_j-\th_l)r_l-r_j)+\frac{1}{\tau}\left(1-\frac{r_j^2}{\beta^2}\right)r_j,\label{SL1}\\
    \ddt \th_j &=\sum_{l=1}^N A_{jl}\frac{r_l}{r_j}\sin(\th_l-\th_j)\label{SL2},
\end{align}
multiplying \eqref{SL1} by $\tau$ and sending $\tau\to 0$ formally yields $r_j=\beta$ for each $j$, and thus \eqref{SL2} reduces to \eqref{eq:Kur}. This phase reduction can also be made rigorous, see \cite{Hoppensteadt1997}. In particular, the interesting regime for analyzing synchronous behavior in the phase-reduced setting is when the natural frequencies $\Omega_i=O(\varepsilon)$, for $\varepsilon>0$ being the scale for the weak-coupling. In our case, all natural frequencies are identically zero and hence this condition holds automatically.

On the other hand, sending $\tau\to \infty$ in \eqref{eq:mainmodel} produces the Abelson (or Hegselmann-Krause) \cite{Abelson1964,Abelson1967,HegselmannKrause02} consensus model
\begin{align}\label{Abelson}
    \ddt z_j=\sum_{l=1}^N A_{jl}(z_l-z_j).
\end{align}
 Convergence to consensus for such class of linear models of opinion dynamics has classically been established through two different metrics. On the one hand, Olfati-Saber and Murray \cite{OlfatiSaberFaxMurray2007,OlfatiSaberMurray2004} proved that, for symmetric or, more generally, balanced non-negative connectivity matrices, the $L^2$ disagreement (or variance) decays exponentially at a rate given by the \textit{Fiedler number} $\lambda_2$ (also called \textit{algebraic connectivity}) of the associated Laplacian matrix \cite{Fiedler1973}, so that consensus is reached as soon as $\lambda_2>0$. On the other hand, for general (non-symmetric, non-balanced), the $L^\infty$ disagreement (or variance) decays exponentially at a rate given by the \textit{scrambling coefficient} $\eta$ of the connectivity matrix, so that consensus is reached as soon as $\eta>0$. The scrambling coefficient originated in the ergodic theory of inhomogeneous Markov chains \cite{Hajnal1958,Seneta2006} and it was later exploited for consensus and flocking models in \cite{MT14}. Recently, the first author \cite{B22} recovered both ideas for time-varying weights in the framework of graphon models. Specifically, she obtained exponential $L^\infty$ consensus under a persistence condition on a time-dependent scrambling coefficient, and $L^2$ consensus for symmetric, balanced or strongly connected topologies in terms of a generalised algebraic connectivity.\\

Our first result contributes to the Kuramoto theory by establishing the existence and stability of \textit{rotational twisted states} where the rotation speed is induced by the asymmetries in the circulant graph structure rather than the natural frequency.
\begin{maintheorem}\label{t:Kurrotpre}
    Let $N\geq 2$. Let $A\in M_N(\mathbb{R})$ be a circulant matrix satisfying the cone condition of Definition \ref{def:cone}. Then for any $c\in\S$, the \textit{rotational twisted state} defined by
    \begin{equation}\label{eq:RotationalStateKuramoto}
        \th_j(t)=\phi_j+\omega t +c, \quad j\in\indices,
    \end{equation}
    where $\displaystyle \omega = \sum_{l=1}^{N-1}a_l\sin(\phi_l)$ and $\displaystyle  \phi_j=\frac{2\pi j}{N}$,
    is linearly stable for the Kuramoto model \eqref{eq:Kur}.
\end{maintheorem}
Note that assuming symmetry in $A$ yields $\omega=0$, producing the stationary twisted states of the Kuramoto model \cite{townsend2017}. Further, the constant $c\in \S$ arises from the rotational invariance of the system, which provides a 1-parameter family of solutions.

Our second contribution shows that the foundational Stuart-Landau system \eqref{eq:mainmodel} also has stable rotational twisted states for small, but positive values of $\tau$.

\begin{maintheorem}\label{t:SLcirc}
    Let $N\geq 2$. Let $A\in M_N(\mathbb{R})$ be a circulant matrix satisfying the cone condition of Definition \ref{def:cone}. 
    Let $d= \sum_{l=1}^{N-1} a_l$,  $\rho=\sum_{l=1}^{N-1}a_l\cos(\phi_l)$, and $\tau^*:=(d-\rho)^{-1}$.
    Then there exists $\tau^{**}\leq\tau^*$ such that for all $\tau\leq \tau^{**}$, for all $c\in\S$, the \textit{rotational twisted state} defined by
    \begin{align*}
        \th_j(t)&=\phi_j+\omega t +c, \quad j\in\indices ,\\
        r_j(t)&= R , \quad j\in\indices,
    \end{align*}
    in which 
    \[R:=\beta\sqrt{\tau\left(\rho-d+\frac{1}{\tau}\right)},\quad \omega = \sum_{l=1}^{N-1}a_l\sin(\phi_l) ,\quad \displaystyle  \phi_j=\frac{2\pi j}{N},\]
    is linearly stable for  the Stuart-Landau model \eqref{eq:mainmodel}. 
\end{maintheorem}
Here,  $\tau^*=(d-\rho)^{-1}$ provides the phase transition for the existence of the rotational twisted state and $\tau^{**}$ provides a quantitative bound which guarantees stability, to be defined later.\\

On the other hand, it is well known that for nonnegative connected matrices, the \textit{synchronized state} of the Kuramoto model is stable, while it was recently shown for the Stuart-Landau model in the complete graph case \cite{millan2025synchronizationcoupledstuartlandauoscillators}. Here we present stability of the \textit{synchronized state} for any nonnegative connected graph, and further provide unconditional convergence to the synchronized state for nonnegative connectivity matrices if the inverse of the relaxation time $\frac{1}{\tau}$ is smaller than either the Fiedler number (for symmetric or balanced connectivity) or the scrambling coefficient (for general topologies). Our result generalizes the above-mentioned classical works for the Abelson model, in particular \cite{B22,MT14, OlfatiSaberMurray2004,OlfatiSaberFaxMurray2007}, and it highlights the competition between the consensus mechanism and the nonlinear forcing term.

\begin{maintheorem}
Let $A\in M_N(\mathbb{R}_+)$ be a general matrix with non-negative entries.
\begin{itemize}
    \item[i)] When $A$ is balanced (i.e., $\sum_k A_{jk}=\sum_k A_{kj}$), consider the second smallest eigenvalue $\lambda_2(L_S)$ counted without multiplicity, where $L\in M_N(\mathbb{R})$ is the Laplacian of $A$, and $L_S=\frac{1}{2}(L+L^T)\in M_N(\mathbb{R})$. Suppose that $\lambda_2(L_S)>0$, and also $\tau> \frac{1}{\lambda_2(L_s)}$. Then, the system \eqref{3rddeq} converges exponentially to a fully synchronized state. Specifically, we have
    $$\sum_{j=1}^N|z_j(t)-\bar z(t)|^2\leq C e^{-2(\lambda_2(L_S)-\frac{1}{\tau})t},$$
    where we define the center of mass $\bar z(t):=\frac{1}{N}\sum_{j=1}^Nz_j(t)$.
    \item[ii)] When $A$ is not necessarily balanced, consider its associated scrambling coefficient
    $$\eta(A):=\min_{1\leq j,k\leq N}\sum_{l=1}^N\min\{A_{jl},A_{kl}\}.$$
    Suppose that $\eta(A)>0$, and also $\tau>\frac{1}{\eta(A)}$. Then, the system \eqref{3rddeq} converges exponentially to a fully synchronized state. Specifically, we have
    $$\max_{1\leq j,l\leq N}|z_j(t)-z_l(t)|\leq C e^{-(\eta(A)-\frac{1}{\tau})t}.$$
\end{itemize}
\end{maintheorem}

This series of results can be encapsulated in the following graphic:

\begin{figure}[h!]
\begin{tikzpicture}[x=1cm,y=1cm,font=\small]
 
\def\xmax{10}            
\def\tss{2.5}            
\def\ts{3.5}             
\def\lam{5.0}            
\def\xeta{7.0}            
\def\h{0.32}             
 
\newcommand{\fillbar}[4]{\fill[#4] (#2,#1-\h) rectangle (#3,#1+\h);}
\newcommand{\outline}[3]{\draw[thick] (#2,#1-\h) rectangle (#3,#1+\h);}
 
\draw[-{Latex[length=2.5mm]},thick] (-0.1,0) -- (\xmax,0) node[below=2pt] {$\infty$};
\foreach \x/\lab in {0/{$0$},\tss/{$\tau^{**}$},\ts/{$\tau^{*}$},\lam/{$\frac{1}{\lambda_2(L_S)}$},\xeta/{$\frac{1}{\eta(A)}$}}{
  \draw[thick] (\x,0.1) -- (\x,-0.1) node[below] {\lab};
}
\node[below=18pt] at (0,0) {(Kuramoto)};
\node[below=18pt] at (\xmax,0) {(Abelson)};
\foreach \x in {\tss,\ts,\lam,\xeta}{\draw[gray!60,densely dotted] (\x,0) -- (\x,4.6);}
 
\outline{1.0}{0}{\ts}
\fillbar{1.0}{0}{\tss}{pattern=north east lines}
\node[right=6pt] at (\xmax,1.0) {Rotational twisted state};
\outline{2.0}{0}{\xmax}
\fillbar{2.0}{0}{\lam}{pattern=north east lines}
\fillbar{2.0}{\lam}{\xmax}{gray!55}
\node[right=6pt] at (\xmax,2.0) {Synchronized state};
\node[left=6pt,align=right,font=\small] at (-0.5,1.5) {Circulant\\matrices};
 
\draw[thick] (-2.3,2.8) -- (\xmax+4.0,2.8);
\draw[thick] (-0.4,0.4) -- (-0.4,4.4);
 
\outline{3.6}{0}{\xmax}
\fillbar{3.6}{0}{\xeta}{pattern=north east lines}
\fillbar{3.6}{\xeta}{\xmax}{gray!55}
\node[right=6pt] at (\xmax,3.6) {Synchronized state};
\node[left=6pt,align=right,font=\small] at (-0.5,3.6) {General\\matrices};
 
\begin{scope}[shift={(0.0,-2.3)}]
  \draw[thick] (-0.3,-0.55) rectangle (\xmax+1.2,0.55);
  \draw[thick] (0.2,-0.2) rectangle (1.2,0.2);
  \node[right] at (1.3,0) {existence};
  \draw[thick,pattern=north east lines] (3.7,-0.2) rectangle (4.7,0.2);
  \node[right] at (4.8,0) {local stability};
  \draw[thick,fill=gray!55] (7.4,-0.2) rectangle (8.4,0.2);
  \node[right] at (8.5,0) {global stability};
\end{scope}
\end{tikzpicture}
\caption{Conditions on $\tau\in[0,\infty]$ for existence and stability of the rotational and synchronized states for nonnegative matrices.}
\end{figure}

Note that the local and global asymptotic stability carries over to the case of generic connected matrices, while the circulant structure is critical for the existence of the \textit{rotational twisted states} as defined in Definition \ref{RTS}.

Finally we study both the Kuramoto and Stuart-Landau systems in the context of leader-follower dynamics, where the leaders are coupled by a circulant structure, and the followers try to adopt the rotation dynamics given by the leaders. So as not to introduce too much notation here, we will hold off on stating the main theorems explicitly here, but note that in both cases we prove the existence of a phase transition for which above the threshold, the followers adopt the rotation speed, while below threshold, they fail to do so and the dynamics are given by a so-called chimera state.

\section{Kuramoto Rotational Twisted States}\label{s:3}
Consider the Kuramoto model with identical natural frequencies, $\Omega_j\equiv 0$ for all $j=1,...,N$, coupled via a directed graph so that the adjacency matrix $A$ is asymmetric. 
\begin{align}\tag{K}\label{Kuramoto}
    \ddt \th_j=\sum_{l=1}^NA_{jl}\sin(\th_l-\th_j), \quad j\in\indices.
\end{align}
It is well known that the synchronized state such that all $\th_j\equiv \th(t)$ with $\ddt \th=0$ is an exponentially stable fixed point for any connected graph $A$, indeed, by linearizing around the fixed point $\th_j=\th_k$ for all $j,k$, the dynamics are actually provided by the Abelson model \eqref{Abelson} and hence for any adjacency matrix containing a directed spanning tree, the synchronized state is stable. However, there have been a wealth of studies to find other types of fixed points for such a system, however most studies still retain the simplifying assumption that $A$ be a symmetric matrix. We present here a new class of stable rotational twisted states which rotate at a new common speed induced by the graph structure.\\

\subsection{Existence of the rotational twisted states}
We begin by showing the existence of the rotational twisted state defined in Theorem \ref{t:Kurrotpre}.

\begin{Prop}
    Let $A\in M_N(\mathbb{R})$ be a circulant matrix generated by a vector $a=(a_0,\cdots,a_{N-1})\in\R^N$. Then, for any $c\in\S$ and any $q\in\Z$, the 
    state defined by  \begin{align*}
        \th_j(t)=\frac{2\pi j q}{N}+\omega_q t +c, \quad j\in\indices,
    \end{align*}
    where $\displaystyle \omega_q := \sum_{l=1}^{N-1}a_l\sin(q\phi_l)$,
    is a solution to the Kuramoto model \eqref{Kuramoto}.\\
    Note that the case $q=0$ corresponds to a synchronized state (with angular velocity $\omega_q=0$), while the case $q\neq 0$ defines a $q$-rotational twisted state.
\end{Prop}

\begin{proof}
Let $A$ be a circulant matrix. Then we can rewrite \eqref{Kuramoto} as 
\begin{align*}\label{Kurcirc}
    \ddt \th_j=\sum_{l=1}^{N}a_{(l-j)[N]}\sin(\th_l-\th_j).
\end{align*}
Now suppose that there is a rotational fixed point such that for all $j\in\indices$, $\ddt \th_j=\omega$ for some $\omega$ to be determined. We define by 
\begin{equation}\label{eq:Th_j}
\Th_j:= 
\left\{\begin{array}{ll}
(\th_{j+1}-\th_j) \ \mod 2\pi, & \text{ if } j\in\{1,\cdots,N-1\},\\
(\th_1-\th_N) \ \mod 2\pi, & \text{ if } j = N,
\end{array}\right.
\end{equation}
the angle differences between two subsequent oscillators, and suppose that they satisfy
\begin{align*}
   \Th_j(t)=\frac{2\pi q}{N} \ \mod 2\pi,
\end{align*} 
for some $q\in\N$ and all $j\in\indices$.
Then,
\begin{align*}
    \ddt \th_j=\sum_{l=1}^{N}a_{(l-j)[N]}\sin\left(\frac{2\pi(l-j)q}{N}\right).
\end{align*}
A change of variables leads to noticing that the right-hand side does not depend on $j$, so that we obtain a common angular velocity for all agents $j\in\indices$ given by:

\begin{align*}
    \omega_q=\sum_{l=1}^{N-1}a_l\sin\left(\frac{2\pi l q }{N}\right)=\sum_{l=1}^{N-1} a_l\sin(q \phi_l),
\end{align*}
where $(\phi_l)_{l\in\{0,\ldots, N-1\}}$ are the angles corresponding to the $N$-th roots of unity, with $\phi_l=\frac{2\pi l}{N}$.
\end{proof}

Now in order to prove linear stability of such states, we need a fixed point map. Indeed, as $\omega_q$ is determined precisely by the known graph structure, we can uniformly shift the natural frequencies $\Omega_j=0$ by the emergent $\omega_q$ to put us in the rotational reference frame for the rotational twisted state. Thus the equations for which we analyze the linear stability are given by
\begin{align}\label{RRF}
    \ddt \th_j=-\omega_q+\sum_{l=1}^NA_{jl}\sin(\th_l-\th_j).
\end{align}

Then the solution $\th_j=\frac{2\pi jq}{N}+c$ is a fixed point of \eqref{RRF}, while the synchronized state would have rotations of speed $\omega_q$.

\subsection{Linear Stability}\label{sec:KurStability}
As seen in the previous section, the solution $\th_j=\frac{2\pi jq}{N}+c$ for $q \in \mathbb{Z}$ is a fixed point of the above system \eqref{RRF}. The case $q=0$ provides the synchronized state, while we focus here on the case $q=1$ and note what happens for $q\in\Z\setminus\{0,1\}$ in Section \ref{ss:3}.
The aim of this section is to prove Theorem \ref{t:Kurrotpre}. 

In order to study the stability, we compute the Jacobian of the fixed-point map. Let $F:\S^N\to\R^N$ be defined by 
\begin{align*}\label{F-kurmap}
    F(\th)_j&=-\omega+\sum_{l=1}^NA_{jl}\sin(\th_l-\th_j)=-\o+\sum_{l=1}^{N-1}a_{(l-j)[N]}\sin(\th_l-\th_j),
\end{align*}
so that 
\[
\ddt\th_j = F(\th)_j.
\]

Now computing the Jacobian gives

\begin{align*}
    \partial_{\th_j}F(\th)_j&=-\sum_{l=1}^Na_{(l-j)[N]}\cos(\th_l-\th_j).
\end{align*}
Plugging in the fixed point $\th_j\equiv \frac{2\pi}{N}+c$ gives
\begin{align}\label{eq:JacKura1}
    \partial_{\th_j}F(\th)_j&=-\sum_{l=1}^{N}a_l\cos\left(\phi_l\right):=-\rho,
\end{align}
 The off diagonal elements are then given by
\begin{align}\label{eq:JacKura2}
    \partial_{\th_{(j+n)\langle N \rangle}}F(\th)_j&=a_{(n)[N]}\cos(\th_{j+n}-\th_j)=a_{(n)[N]}\cos(\phi_n).
\end{align}

Therefore the Jacobian at this fixed point is also represented by a circulant matrix generated by the vector $(d_0,...,d_{N-1})$:
$$
d_j=\left\{\begin{array}{ll}
-\sum_{l=1}^{N-1}a_l\cos\left(\phi_l\right):=-\rho, & \text{if }j=0,\\
a_{j}\cos\left(\phi_j\right),& \text{if }j\in\{1,\ldots, N-1\}.\end{array}\right.
$$
We can then take advantage of the diagonalization formula for circulant matrices yielding the following eigenvalues,
\begin{align*}\label{1kur:eigenvalues}
    \l_k=d_0+d_1e^{ik\phi_1}+d_2e^{ik\phi_{2}}+...+d_{N-1}e^{ik\phi_{N-1}}, \ \ \ k\in\{0,...,N-1\}.
\end{align*}

We see that $\lambda_0=0$, as to be expected from the rotational invariance of the system and for all other eigenvalues we have to verify that
\begin{align*}
    \Re\lambda_k=\Re\left(d_0+d_1e^{ik\phi_1}+d_2e^{ik\phi_{2}}+...+d_{N-1}e^{ik\phi_{N-1}}\right)<0,
\end{align*}
to obtain the linear stability of the rotational state. We immediately get a lemma consistent with the theory for stationary twisted states.

\begin{Lemma}
Let $A\in M_N(\mathbb{R}_+)$ be a nonnegative circulant matrix. 
If $N\leq 4$, then the \textit{rotational twisted state} defined by \eqref{eq:RotationalStateKuramoto}
are not linearly stable.
\end{Lemma}
\begin{proof}
    For $N=2$ we see that $d_0=a_1=-d_1$ so that $\l_1=0,\l_2=2a_1>0$.
    
    For $N=3$ we see that $d_0=-a_1\cos(2\pi/3)-a_2\cos(4\pi/3)>0$ and hence $D_{\theta}F(\theta)$ cannot have all eigenvalues  (other than $\lambda_0=0$) with strictly negative real part.
    
    For $N=4$, we see that $d_0=a_2=-d_2$ so that $\mathrm{Re}(\l_k)=a_2(1-(-1)^k)$ so that $\l_1=0=\mathrm{Re}(\l_2)$, and a study of the nonlinear stability as done in \cite{Sclosa_2024} is necessary.
\end{proof}
Note that the nonnegativity is critical here as if the circulant matrix can take negative values then the rotational twisted state can become stable even for $N\geq 2$.

Indeed, the cone condition Definition \ref{def:cone} provides us with the last ingredient for concluding a necessary and sufficient condition for the \textit{rotational twisted state} to be stable. We conclude with the proof of Theorem \ref{t:Kurrotpre}, which we state again for clarity of presentation.

\begin{Theorem}
    Let $N\geq 2$. Let $A\in M_N(\mathbb{R})$ be a circulant matrix satisfying the cone condition of Definition \ref{def:cone}. Then for any $c\in\S$, the \textit{rotational twisted state} defined by
    \begin{align*}
        \th_j(t)=\phi_j+\omega t +c, \quad j\in\indices,
    \end{align*}
    where $\displaystyle \omega = \sum_{l=1}^{N-1}a_l\sin(\phi_l)$ and $\displaystyle  \phi_j=\frac{2\pi j}{N}$,
    is linearly stable for the Kuramoto model \eqref{eq:Kur}.
\end{Theorem}

\begin{proof}
     Expanding $\mathrm{Re} \lambda_k$ for each $k=1,\ldots, N-1$ yields
    \begin{align*}
        \mathrm{Re} \lambda_k=-\sum_{l=1}^{N-1} a_l\cos(\phi_l)(1-\cos(k\phi_l)),
    \end{align*}
    The negativity of the above is equivalent to $\bar{a} \in \mathcal{C}_N$.
\end{proof}
Of particular interest in the above theorem is that there is no explicit sign condition on the values of the adjacency matrix. Therefore this criterion covers the case of all circulant matrices, even those with repulsive interconnections. In particular, there are convenient subsets of the cone $\mathcal{C}_N$ which can be explicitly stated.
\begin{Cor}\label{Cor:Example1}
Define the sets
    \begin{align*}
        \mathcal{P}:=\{l : \cos(\phi_l)>0\}, \ \ \ \mathcal{N}:=\{l : \cos(\phi_l)\leq 0 \}.
    \end{align*}
    Let $A\in M_N(\mathbb{R})$ be a circulant matrix generated by the vector $a=(a_0,\cdots,a_{N-1})\in\R^N$, such that if $l \in \mathcal{P}$, $a_l\geq 0$ and if $l \in \mathcal{N}$, $a_l\leq 0$, and such that there exists $j> 0$ such that the graph is connected. Then the rotational twisted state is a linearly stable solution to the Kuramoto model.
\end{Cor}
\begin{proof}
    It is trivial to check that such a vector satisfies $\bar{a}=(a_1,\cdots a_{N-1}) \in \mathcal{C}_N$.
\end{proof}
The above condition provides that each oscillator communicates positively with the other oscillators that asymptotically should be in the same half-plane as themselves, and communicates repulsively with those on the other side of the half-plane. This prevents cross communication from breaking down the stability of the rotational state. However, we can allow cooperative cross communication as long as the communication with closer agents is emphasized.
\begin{Cor}\label{Exabs}
    Let $A\in M_N(\mathbb{R})$ be a circulant matrix generated by the vector $a=(a_0,\cdots,a_{N-1})\in\R^N$, with the sets $\mathcal{P}$ and $\mathcal{N}$ defined as in  Corollary \ref{Cor:Example1}.
    Then, if the following condition holds
    \begin{align*}
        \min_{1\leq k\leq N-1} \sum_{l \in \mathcal{P}}a_l\cos(\phi_l)(1-\cos(k\phi_l))>\max_{1\leq k\leq N-1} \sum_{l \in \mathcal{N}}a_l|\cos(\phi_l)|(1-\cos(k\phi_l)),
    \end{align*}
    the rotational twisted state is linearly stable.
\end{Cor}
\begin{proof}
    It is trivial to check that such a vector satisfies $\bar{a} =(a_1,\ldots a_{N-1})\in \mathcal{C}_N$.
\end{proof}
This allows for even a graph structure with $A_{jk}>0$, for all $j,k\in\indices$ to have a stable attractor which is not the fully synchronized state, as long as the near neighbor communication is emphasized.

\begin{Rem}
The condition $\kappa\bar{a} \in \mathcal{C}_N$ for $\kappa>0$ holds if and only if $\bar{a} \in \mathcal{C}_N$ so if the weighted homogeneous Kuramoto model is endowed with a uniform coupling strength,
\begin{align*}
    \ddt \th_j=\k\sum_{l=1}^NA_{jl}\sin(\theta_l-\theta_j),
\end{align*}
then there is no phase transition in $\kappa$ with regards to the stability of the rotational twisted state.
\end{Rem}

\subsection{Existence of multiple rotational states}\label{ss:3}
It is important to note that the identification of the rotational twisted state creates a link between the indices of the oscillators and the order in which they are distributed around the circle. This raises the question as to whether or not a reordering of oscillators could provide multiple different rotational twisted states for the same adjacency matrix. Indeed, this amounts to letting $\th_j=\frac{2\pi qj}{N}+c+\o_qt$ for $q\neq 0,1$, which provides what are known as the q-twisted rotational states of the Kuramoto model \cite{WileyStrogatzGirvan06}. However, not all values of $q$ produce an appropriate rearrangement of indices, as if $\gcd(q,N)\neq 1$ then clustering can occur. We choose not to discuss this here and leave the phenomenon of rotational clusters to a future study.

Rather, we see that the number of distinct rotational twisted states for a given circulant adjacency matrix is given exactly by half of Euler's totient function.
\begin{Prop}\label{multi}
    Let $A\in M_N(\mathbb{R})$ be a circulant matrix generated by vector $a=(a_0,\cdots,a_{N-1})\in\R^N$.
    The permutations of $\indices$ that provide distinct twisted rotational states take the form
    \begin{align*}
    \sigma:\indices&\to \indices\\
        k &\mapsto  (qk)\, \langle N\rangle, 
    \end{align*}
    for some integer $q$ with $\gcd(q,N)=1$, in which for any $k\in\Z$, $k\langle N \rangle := (k-1)[ N ] +1$. In particular, there are exactly $\varphi(N)$ such permutations, leading to $\frac{1}{2}\varphi(N)$ such states up to the reflection $\th\mapsto -\th$,  where $\varphi$ denotes Euler's totient function.
    
    Moreover, denoting by $P$ the transition matrix associated with a permutation $\sigma$, the corresponding rotational twisted state is stable if $PAP^{-1}$ satisfies the cone condition of Definition \ref{def:cone}, and the resulting speed of the rotational state is given by \begin{align*}
    \omega=\sum_{l=1}^{N-1} a_{\sigma(l)}\sin(\phi_l).
\end{align*}

\end{Prop}
\begin{proof}
    Since $A$ is a circulant matrix
    \begin{align*}
        A_{jl}=a_{(l-j) [N]},
    \end{align*}
    for each $j,l\in\indices$.
    Let $P$ be a permutation matrix corresponding to a permutation $\sigma$ of $\{1,...,N\}$.  Applying the permutation gives
    \begin{align*}
    (PAP^{-1})_{j,l}= A_{\sigma(j),\sigma(l)}= a_{(\sigma(l)-\sigma(j)) [N]}.
    \end{align*}
For the resulting matrix $PAP^{-1}$ to be circulant, the quantity
$(\sigma(l)-\sigma(j))[N]$ must depend only on $(l-j)[N]$.
This occurs precisely when

\[
\sigma: k \mapsto (q k + r) \langle N \rangle, 
\]
for some integers $q\in\Z$ and $r\in\Z$, which implies
\begin{align*}
(\sigma(l)-\sigma(j))[N]
= (q(l-j)) [N].
\end{align*}
Moreover, since the shift by an integer $r$
does not change the ordering, we dismiss it by taking $r=0$.
Then, such a map $\sigma$ is bijective from $\indices$ to $\indices$ if and only if
$\gcd(q,N)=1$. 
Denoting by $\varphi$ Euler's totient function, this gives $\frac{1}{2}\varphi(N)$ distinct rotational twisted states as the reverse ordering of each permutation produces the same state, but with the opposite rotational speed $-\omega$.

The stability of the rotational twisted state and the rotation speed are then obtained by reasoning with the circulant matrix $PAP^{-1}$.
\end{proof}

As shown in the previous proposition, the homogeneous Kuramoto model with circulant adjacency matrix can exhibit multistability of distinct rotational twisted states along with stability of the fully synchronized state.
Note that this coprimality condition for circulant graphs is equivalent to the uniform $q$-twisted states \cite{WileyStrogatzGirvan06}.

\section{Stuart-Landau Rotational States}\label{s:4}
We continue our analysis considering the Stuart-Landau system which allows oscillators $z_j\in \mathbb{C}$ to have both phases $\th_j$ and amplitudes $r_j$, so that $z_j=r_je^{i\theta_j}$. We rewrite the model here with a circulant matrix $A$ and identical Hopf-parameters $\beta_j\equiv \beta>0$.

\begin{align}\tag{SL}\label{2ndeq}
    \ddt z_j=\sum_{l=1}^NA_{jl}(z_l-z_j)+\frac{1}{\tau}\left(1-\frac{|z_j|^2}{\beta^2}\right)z_j, \ \ \ z_j\in \mathbb{C}, \ \ j\in\indices.
\end{align}

\subsection{General properties}

We begin our study of the model \eqref{2ndeq} by proving that the system stays bounded at all time, regardless of the circulant structure, and of whether the interaction matrix has positive values (i.e. the interactions are attractive) or negative values (the interactions are repulsive).

\begin{Prop}
Let $A\in M_N(\R)$ be a general matrix, and let \[\alpha:= \max_{j\in\{1,\cdots,N\}} \left( \sum_{k=1}^N |A_{jk}| \one_{A_{jk}<0}\right).\]
Let $(z_j)_{j\in\{1,\cdots,N\}}$ be a solution to the Stuart-Landau system \eqref{2ndeq} with interaction matrix $A$ and initial data $(z_j^0)_{j\in\{1,\cdots,N\}}$.
Then for all $t\geq 0$ and all $j\in\{1,\cdots,N\}$, it holds
\[
|z_j(t)| \leq \max\left( \max_{k\in\{1,\cdots,N\}} |z_k^0|, \beta\sqrt{2\tau\alpha +1}\right).
\]
In particular, if $A$ has no negative coefficient, $|z_j(t)| \leq \max\left( \max_{k\in\{1,\cdots,N\}} |z_k^0|, \beta\right).$
\end{Prop}

\begin{proof}
Let $m:t\mapsto \max_{j\in\{1,\cdots,N\}} |z_j(t)|$.
By Danskin's theorem, $m$ is differentiable for almost every $t\geq 0$, and 
\[
\ddt m^2(t) = \max_{j\in\Omega(t)} \frac{\dd}{\dt} |z_j(t)|^2,
\]
where $\Omega(t):=\arg\max_{j\in\{1,\cdots,N\}}\{|z_j(t)|\}$.

Trivially, $m(0)\leq M:=\max\left(m(0),\beta\sqrt{2\tau\alpha+1}\right)$.
Suppose that there exists $t>0$ such that $m(t)>M$, and let $t_0:=\inf\{t\geq 0 \,|\, m(t)>M\}$. By continuity, $m(t_0) = M$ and there exists $\varepsilon>0$ such that for almost all $t\in(t_0, t_0+\varepsilon)$, $\frac{\dd}{\dt}m(t)>0$ and $m(t)> M$.

Then, for almost all $t\in (t_0, t_0+\varepsilon)$, for all $j\in\Omega(t)$,
\[
\begin{split}
\frac{\dd}{\dt} |z_j(t)|^2  =&  2 \sum_{k=1}^N A_{jk} (z_k(t)-z_j(t))\cdot z_j(t) +\frac{2}{\tau}\left(1-\frac{|z_j(t)|^2}{\beta^2}\right)|z_j(t)|^2\\
= & 2\sum_{k:\, A_{jk}>0} A_{jk} (z_k(t)-z_j(t))\cdot z_j(t) + 2\sum_{k:\, A_{jk}<0} |A_{jk}| (z_k(t)-z_j(t))\cdot z_j(t) \\
&+\frac{2}{\tau\beta^2}\left(\beta^2-|z_j(t)|^2\right)|z_j(t)|^2 \\
\leq & 2\sum_{k:\, A_{jk}>0} A_{jk} (|z_k(t)||z_j(t)|-|z_j(t)|^2) + 2\sum_{k:\, A_{jk}<0} |A_{jk}| |z_j(t)-z_k(t)| |z_j(t)| \\
&+\frac{2}{\tau\beta^2}\left(\beta^2-|z_j(t)|^2\right)|z_j(t)|^2 .
\end{split}
\]
Then, 
\[
\begin{split}
\ddt m^2(t) = \max_{j\in\Omega(t)} \frac{\dd}{\dt} |z_j(t)|^2  \leq &\, 0  + 4\sum_{k:\, A_{jk}<0} |A_{jk}|m^2(t) 
+\frac{2}{\tau\beta^2}\left(\beta^2-m^2(t)\right)m^2(t) \\
\leq & \frac{2}{\tau\beta^2}\left( 2\tau\beta^2\sum_{k:\, A_{jk}<0} |A_{jk}| + \beta^2-m^2(t)\right)m^2(t)\\
\leq &  \frac{2}{\tau\beta^2}\left( M^2-m^2(t)\right)m^2(t)\leq 0
\end{split}
\]
by definition of $M$ and $m$.
By Danskin's theorem, this implies that for almost all $t\in [t_0, t_0+\varepsilon]$, $\frac{\dd m(t)}{\dt}<0$, which contradicts the existence of $t\geq 0$ such that  $m(t)>M$.
\end{proof}

\subsection{Existence of the rotational state}
Inspired by our investigation of the Kuramoto model, we begin with identifying the rotational twisted state. For a circulant matrix $A\in M_N(\mathbb{R})$ generated by the vector $a=(a_0,\cdots,a_{N-1})\in\R^N$, let us again define
\begin{align}\label{eq:omega}
    d=\sum_{l=1}^{N-1} a_l, \quad \rho=\sum_{l=1}^{N-1}a_l\cos(\phi_l), \quad \omega=\sum_{l=1}^{N-1}a_l\sin(\phi_l),
\end{align}
where $\phi_l=\frac{2\pi l}{N}$ are the roots of unity and $d$ is the degree of each node, which is uniform due to the circulant structure. 

\begin{Prop}\label{Prop:ExistenceMillingSL}
Let $A\in M_N(\mathbb{R})$ be a circulant matrix satisfying $d-\rho > 0$, where $d$ and $\rho$ are given by \eqref{eq:omega}. Let $\tau^*=(d-\rho)^{-1}$. Then, for all $\tau<\tau^*$, for all $\theta\in\S$, the \textit{rotational twisted state} defined by
    \begin{align}\label{ms1}
       z_j(t)=Re^{i\left(\omega t+\phi_j+\theta\right)},
    \end{align}
    with 
    \begin{align}\label{ms2}
       R:=\beta\sqrt{\tau\left(\rho-d+\frac{1}{\tau }\right)}=\beta\sqrt{\tau\left(\frac{1}{\tau }-\frac{1}{\tau^*}\right)}<\beta
    \end{align} is a solution to \eqref{2ndeq}.\\
    On the other hand, if $d-\rho \leq 0$, then the rotational twisted state given by \eqref{ms1} exists without any restriction on $\tau$, with $R:=\beta\sqrt{\tau\left(\rho-d+\frac{1}{\tau }\right)}\geq\beta$.
\end{Prop}
\begin{proof}
    Let us begin with the characterization of \eqref{ms1}-\eqref{ms2}. Taking as an ansatz $z_j(t)=re^{i\left(\Omega t+\phi_j+\theta\right)}$ where $r>0$ and $\Omega \in \mathbb{R}$ are yet to be defined. Differentiating $z_j$ yields
    \begin{align*}
        \ddt z_j= i\Omega z_j&=\sum_{l=1}^NA_{jl}(z_l-z_j)+\frac{1}{\tau}\left(1-\frac{|z_j|^2}{\beta^2}\right)z_j,
        \end{align*}
        i.e.
        \begin{align*}
       i\Omega re^{i\left(\Omega t+\phi_j+\theta\right)}  &=\sum_{l=1}^NA_{jl}re^{i\left(\Omega t+\phi_j+\theta\right)}\left(e^{i(\phi_l-\phi_j)}-1\right)+\frac{1}{\tau}re^{i\left(\Omega t+\phi_j+\theta\right)}\left(1-\frac{r^2}{\beta^2}\right).
    \end{align*}
    Dividing out by $z_j=re^{i\left(\Omega t+\phi_j+\theta\right)}$ and taking real and imaginary parts yields
    \begin{align}
       0&=\sum_{l=1}^NA_{jl}\left(\cos\left(\phi_l-\phi_j\right)-1\right)+\frac{1}{\tau}\left(1-\frac{r^2}{\beta^2}\right)=\sum_{l=1}^Na_{l}\left(\cos\left(\phi_l\right)-1\right)+\frac{1}{\tau}\left(1-\frac{r^2}{\beta^2}\right),\label{Req}
        \end{align}
        and
       \begin{align*} \Omega&=\sum_{l=1}^NA_{jl}\sin\left(\phi_l-\phi_j\right)=\sum_{l=1}^{N-1}a_l\sin(\phi_l)=\omega.
    \end{align*}
    Solving \eqref{Req} for $r$ leads to $r=R$, with $R$ given in Equation \eqref{ms2}. Further, for the rotational twisted state to exist it is necessary that $R>0$. 
    
    If $ d-\rho >0$, this imposes the inequality
    \begin{align}\label{eq:taustar}
        \tau<\left(d-\rho\right)^{-1}=\tau^*.
    \end{align}

    Hence if $\tau<\tau^*$, then \eqref{ms1}-\eqref{ms2} defines a rotational twisted state.
    On the other hand, in the case $d-\rho\leq 0$,  
    the rotational state exists for all $\tau>0$. In particular, $\tau^*=\infty$ holds for any nonpositive circulant matrix.
    
\end{proof}

In the case $d-\rho>0$, the value $\tau^*$ represents a phase transition, since the rotational state cannot exist for $\tau\geq \tau^*$, as it would imply $R^2\leq 0$.

\subsection{Explicit solution for radially symmetric initial data}\label{ss:Invariant}
 Supposing we begin on a common radius $R(0)>0$ with $\Th_j(0)= \theta_{(j+1)\langle N\rangle}(0)-\theta^0_j(0)=\frac{2\pi}{N} \ \mod 2\pi$ we can compute the explicit solution for the evolution of $R(t)$ and $\th_j(t)$.
We assume again that $A\in M_N(\mathbb{R})$ is a circulant matrix generated by the vector $a=(a_0,\cdots,a_{N-1})\in\R^N$ so that $A_{jl} = a_{(l-j)[N]}$, and we look for solutions to the Stuart-Landau system \eqref{2ndeq} that are invariant by a rotation of $\frac{2\pi}{N}$, i.e. solutions that satisfy 
$z_j(t) = R(t) e^{i(\frac{2\pi j}{N}+\theta(t))}$ 
for some common radius $R:\R_+\rightarrow\R_+$ and angle $\theta:\R_+\rightarrow \mathbb{S}$.
Plugging this ansatz into the system of ODEs \eqref{2ndeq}, we obtain
\[
\begin{split}
    \ddt z_j(t) & = \left(\ddt R(t) + i R(t)  \ddt\theta(t)\right) e^{i(\theta(t)+\frac{2\pi j}{N})} \\
    & = \sum_{l=1}^N A_{jl} R(t) e^{i(\frac{2\pi j}{N}+\theta(t))} \left(e^{i(\frac{2\pi l}{N}-\frac{2\pi j}{N})}-1\right) +\frac{1}{\tau} \left(1-\frac{R^2(t)}{\beta^2}\right) R(t)e^{i(\frac{2\pi j}{N}+\th(t))}.
\end{split}
\]
Separating real and imaginary parts, we get the following closed equation for the evolution of the radius
\[
\begin{split}
\ddt R(t) & = R(t) \sum_{l=1}^N A_{jl} \left( \cos\left(\frac{2\pi(l-j)}{N}\right) -1\right) +\frac{1}{\tau} \left(1-\frac{R^2(t)}{\beta^2}\right) R(t)\\
& = R(t) \sum_{l=1}^N a_{(l-j)[N]} \left( \cos\left(\frac{2\pi(l-j)}{N}\right) -1\right) +\frac{1}{\tau} \left(1-\frac{R^2(t)}{\beta^2}\right) R(t)\\
 & = R(t) \left(\left(\frac{1}{\tau}-\frac{1}{\tau^*}\right) -\frac{1}{\tau} \frac{R^2(t)}{\beta^2}\right), 
\end{split}
\]
in which $\tau^*$ is given in \eqref{eq:taustar}.
Assuming that $R(t)>0$, the angle $\theta(t)$ can simply be found by integrating the constant angular velocity: 
\[
\ddt \theta(t)  = \sum_{l=1}^{N-1} a_{l}  \sin\left(\frac{2\pi l}{N}\right). 
\]
Denoting $R_0:=R(0)$, we focus on solving 
\begin{equation}\label{eq:Rdot}
    \ddt R(t) = R(t) \left(\left(\frac{1}{\tau}-\frac{1}{\tau^*}\right) -\frac{1}{\tau} \frac{R^2(t)}{\beta^2}\right), 
\end{equation}
distinguishing between three cases.
This gives us the following result.
\begin{Prop} Let $A\in M_N(\mathbb{R})$ be a circulant matrix generated by the vector $a=(a_0,\cdots,a_{N-1})\in\R^N$.
Let $z_j\in C(\R_+;\C)$ be the solution to the Stuard-Landau system \eqref{2ndeq} with initial data given by $z_j^0 = R_0 e^{i(\frac{2\pi j}{N}+\theta_0)}$. Let $\omega$ be defined as in \eqref{eq:omega}. Then, $z_j(t)= R(t) e^{i(\frac{2\pi j}{N}+\theta(t))}$, in which for all $t\in\R_+$, 
\[
\theta(t) = \theta(0)+\omega t,
\]
and $R$ depends on the sign of $\tau-\tau^*$.

In the case $\tau>\tau^*$,
\[
R(t) = R_0\sqrt{\frac{ \frac{1}{\tau^*}-\frac{1}{\tau}}{\left(\frac{R_0^2}{\tau\beta^2}+(\frac{1}{\tau^*}-\frac{1}{\tau})\right)e^{2(\frac{1}{\tau^*}-\frac{1}{\tau})t}-\frac{R_0^2}{\tau\beta^2}} },
\]
so that $R$ converges to zero exponentially, with rate $\frac{1}{\tau^*}-\frac{1}{\tau}$.

In the case $\tau=\tau^*$, 
\[
R(t) = R_0\beta \sqrt{\frac{\tau }{\tau\beta^2 + 2R_0^2 t}},
\]
so that again the radius converges to zero, but only with a polynomial rate $t^{-1/2}$.

In the case $\tau<\tau^*$,
\[
R(t) = R_0\sqrt{\frac{ \frac{1}{\tau}-\frac{1}{\tau^*}}{\left((\frac{1}{\tau}-\frac{1}{\tau^*})-\frac{R_0^2}{\tau\beta^2}\right)e^{-2(\frac{1}{\tau}-\frac{1}{\tau^*})t}+\frac{R_0^2}{\tau\beta^2}} },
\]
so that for any positive initial data $R_0>0$, $R$ converges exponentially to the rotational twisted state characterized by radius $ \beta\sqrt{1-\frac{\tau}{\tau^*}}$.
We recover the existence of the rotational twisted state $\eqref{ms1}-\eqref{ms2}$ for any $\tau<\tau^*$,
and all initial data on the manifold ($R_j(0)=R_0$, $\theta_j(0)=\frac{2\pi j}{N}+\theta_0$) converges to it. In the case $d-\rho\leq 0$, we have $\tau^*=\infty$ and the limiting radius is characterized by $\beta\sqrt{1-\tau(d-\rho)}\geq \beta$.

\end{Prop}

\subsection{Stability of the rotational state}
The stability of the rotational twisted state off of the above submanifold is complicated by the amplitude dependence of each oscillator. However, we can still show the existence of a quantitative upper bound $\tau^{**}$ for which the rotational twisted state is linearly stable if $\tau<\tau^{**}$, and if the circulant matrix also satisfies the same cone condition provided in Definition \ref{def:cone}.

\begin{Theorem}\label{t:SLstab}
    Let $N\geq 2$. Let $A\in M_N(\mathbb{R})$ be a circulant matrix satisfying the cone condition of Definition \ref{def:cone}, and 
    $d,\rho$ defined in \eqref{eq:omega} such that $\frac{1}{\tau}>d-\rho$. .
    Then there exists $\tau^{**}>0$
    such that for all $\tau\leq \tau^{**}$, for all $c\in\S$, the \textit{rotational twisted state} defined by
    \begin{align*}
        \th_j(t)&=\phi_j+\omega t +c, \quad j\in\indices ,\\
        r_j(t)&= R , \quad j\in\indices,
    \end{align*}
    in which 
    \[R=\beta\sqrt{\tau\left(\rho-d+\frac{1}{\tau}\right)},\quad \omega = \sum_{l=1}^{N-1}a_l\sin(\phi_l) ,\quad \displaystyle  \phi_j=\frac{2\pi j}{N},\]
    is linearly stable for  the Stuart-Landau model \eqref{eq:mainmodel}. 
    Moreover, if $d-\rho>0$, it holds $\tau^{**}\leq \tau^* = (d-\rho)^{-1}$.
\end{Theorem}

\begin{proof}

As in the Kuramoto case (cf. Section \ref{s:3}), as we know the limiting rotational speed of the system we shift to the rotational reference frame
\begin{align}\label{RRFSL}
    \ddt z_j=\sum_{l=1}^NA_{jl}(z_l-z_j)+\frac{1}{\tau}\left(1-\frac{|z_j|^2}{\beta^2}\right)z_j-i\omega z_j.
\end{align}
Breaking down \eqref{RRFSL} into its amplitude and phase equations yields
\begin{align*}
    \ddt r_j & =\sum_{l=1}^N A_{jl}(\cos(\theta_{l}-\theta_j)r_l-r_j)+\frac{1}{\tau}\left(1-\frac{r_j^2}{\beta^2}\right)r_j, \\
    \ddt \th_j&=-\omega+\sum_{l=1}^NA_{jl}\frac{r_l}{r_j}\sin(\theta_l-\theta_j). 
\end{align*}
     Introducing the functions $F_r:(\R_+)^N\times \R^N\to \R^N$ and $F_\theta:(\R_+)^N\times \R^N\to \R^N$, the equations for the radii and the phases can then be written as:
\begin{align*}
    \ddt r_j=&(F_r)_j:=\sum_{l=1}^N a_{(l-j)[N]}(\cos(\theta_{l}-\theta_j)r_l-r_j)+\frac{1}{\tau}\left(1-\frac{r_j^2}{\beta^2}\right)r_j,\\
    \ddt \theta_j&=(F_{\theta})_j:=-\omega+\sum_{l=1}^N a_{(l-j)[N]}\frac{r_l}{r_j}\sin(\th_l-\th_j).
\end{align*}
    Since the Stuart-Landau oscillators are amplitude-dependent, the fixed-point analysis must take into account all phases $(\th_j)_{j\in\indices}$ as well as the amplitudes $(r_j)_{j\in\indices}$. Thus, the fixed-point map is given by
    \begin{align*}
        F_{r,\theta}=(F_r,F_{\theta}),
    \end{align*}
    and the corresponding Jacobian will be a $2N \times 2N$ matrix. More specifically, the Jacobian is given by 4 $N\times N$ blocks that we will now compute piece by piece.
    \[
    J=\left(\begin{array}{c|c}
       D_{r,r}  & D_{\theta,r}  \\
         \hline
       D_{r,\theta}  & D_{\theta, \theta}
    \end{array}\right),
    \]
    in which for each $j,k\in\indices$ and $a,b\in\{r,\theta\}$, $(D_{a,b})_{j,k}:=\partial_{a_k}(F_{b})_j(R,\cdots,R,\phi_1+c,\cdots,\phi_N+c)$ are the partial derivatives of $F_{r,\theta}$ evaluated at the rotational twisted state value found in Proposition~\ref{Prop:ExistenceMillingSL}. 
    We begin by computing $D_{r,r}$. For its diagonal terms, it holds for all $j\in\indices$:
    \begin{align*}
        \partial_{r_j}(F_r)_j=\frac{1}{\tau}\left(1-\frac{3r_j^2}{\beta^2}\right)-d=\frac{1}{\tau}\left(1-\frac{3R^2}{\beta^2}\right)-d.
    \end{align*}
    Using $(F_r)_j(R,\cdots,,R,\phi_1+c,\cdots,\phi_N+c)=0$, we have $\frac{1}{\tau}\left(1-\frac{R^2}{\beta^2}\right)-d=-\rho$ and hence
    \begin{align*}
        (D_{r,r})_{j,j}=\partial_{r_j}(F_r)_j(R,\cdots,R,\phi_1+c,\cdots,\phi_N+c)=-\frac{2R^2}{\tau\beta^2}-\rho.
    \end{align*}
    This provides a negative dominance on the diagonal of the upper-left block of the Jacobian, which will help facilitate stability when $\tau$ is small.
    
    Computing the off-diagonal elements gives
    \begin{align*}
        &(D_{r,r})_{j,(j+n)\langle N\rangle}=\partial_{r_{{j+n}\langle N\rangle}}(F_r)_j(R,\cdots,R,\phi_1+c,\cdots,\phi_N+c)= a_{n[N]}\cos\left(\phi_{n}\right),
    \end{align*}
    in which we recall that $\phi_n=\frac{2\pi n}{N}$.
    Thus the upper-left block of the Jacobian is given by a $N\times N$ circulant submatrix generated by the row \[\overline{d_{r,r}}=\left(-\frac{2R^2}{\tau\beta^2}-\rho, a_1\cos(\phi_1), \ldots ,a_{N-1}\cos(\phi_{N-1})\right).\] Note that this submatrix is a signed, directed Laplacian diagonally dominated by $-\frac{2R^2}{\tau\beta^2}$ since the sum of the off-diagonal elements gives exactly $\rho$.\\

    Moving to the upper-right block, we must compute $(D_{\theta,r})_{j,k} = \partial_{\theta_k}(F_r)_j(R,\cdots,R,\phi_1+c,\cdots,\phi_N+c)$ for each $j,k\in\indices$. Then, for all $j\in\indices$,
    \begin{align*}
        (D_{\theta,r})_{j,j}=\sum_{l=1}^{N}a_{(l-j)[N]}\sin(\th_l-\th_j)r_l=R\omega.
    \end{align*}
    The off-diagonal elements give
    \begin{align*}
        (D_{\theta,r})_{j,(j+n)\langle N\rangle}&=\partial_{\theta_{(j+n)\langle N\rangle}}(F_r)_j(R,\cdots,R, \phi_1+c,\cdots,\phi_N+c)\\
        &=-a_{n[N]}r_{(j+n)\langle N\rangle}\sin(\phi_n)=-a_{n[N]}R\sin(\phi_n).
    \end{align*}
    
    Thus the upper-right block is given by another circulant matrix generated by the row $$\overline{d_{\theta,r}}=R\left(\omega, -a_1\sin(\phi_1),\cdots -a_{N-1}\sin(\phi_{N-1}) \right),$$
    which is also a signed, directed Laplacian.
    The bottom-left block involves computing each $(D_{r,\theta})_{j,(j+n)\langle N\rangle}=\partial_{r_k}(F_{\theta})_j(R,\cdots,R,\phi_1+c,\cdots,\phi_N+c)$. The main diagonal is given by
    \begin{align*}
        (D_{r,\theta})_{j,j}=&\partial_{r_j}(F_{\theta})_j(R,\cdots,R,\phi_1+c,\cdots,\phi_N+c)=-\sum_{l=1}^Na_{(l-j)[N]}\frac{r_l}{r_j^2}\sin(\th_l-\th_j)=-\frac{1}{R}\omega,
    \end{align*}
        and the off-diagonal terms
    \begin{align*}
        (D_{r,\theta})_{j,(j+n)\langle N\rangle}&=\partial_{r_{(j+n)\langle N \rangle}}(F_{\theta})_j(R,\cdots,R,\phi_1+c,\cdots,\phi_N+c)\\
        &=\frac{1}{r_j}a_{n[N]}\sin(\th_{j+n}-\th_j)=\frac{1}{R}a_{n[N]}\sin(\phi_n).
    \end{align*}

    Thus the lower-left block is also circulant and generated by the row
    $$\overline{d_{r,\theta}}=-\frac{1}{R}\left(\omega, -a_1\sin(\phi_1),\cdots -a_{N-1}\sin(\phi_{N-1}) \right)=-\frac{1}{R^2}\overline{d_{\theta,r}},$$
    which is the same signed, directed Laplacian as the upper right block, with a prefactor of $-\frac{1}{R^2}$.

    Finally, the lower-right block will be reminiscent of the Kuramoto Jacobian computed in Equations \eqref{eq:JacKura1}-\eqref{eq:JacKura2}. We start again computing the main diagonal:
    \begin{align*}
        (D_{\theta,\theta})_{j,j}&=\partial_{\th_j}(F_{\theta})_j(R,\cdots,R,\phi_1+c,\cdots,\phi_N+c)=-\sum_{l=1}^N a_{(l-j)[N]}\frac{r_l}{r_j}\cos(\th_l-\th_j)=-\rho,
    \end{align*}
    and the off-diagonal elements:
    \begin{align*}
        (D_{\theta,\theta})_{j,(j+n)\langle N\rangle}&=\partial_{\th_{(j+n)\langle N\rangle}}(F_{\theta})_j(R,\cdots,R,\phi_1+c,\cdots,\phi_N+c)\\
        &=a_{n[N]}\frac{r_{n\langle N\rangle}}{r_j}\cos(\th_{j+n}-\th_j)=a_{n[N]}\cos(\phi_n).
    \end{align*}
    
    Therefore, the bottom-right matrix is also circulant with generating row
    \[\overline{d_{\theta,\theta}}=\left(-\rho, a_1\cos(\phi_1),\ldots,a_{N-1}\cos(\phi_{N-1})\right),\]
    which is also a signed, directed Laplacian and the exact matrix found in the Kuramoto case.\\

    Now, as each block is itself a circulant matrix, we can compute the eigenvalues and eigenvectors of each block and thus perform a block-by-block diagonalization. 
    
    First, let us define $F_N\in M_N(\mathbb{C})$ as the transition matrix for the eigenvectors of a circulant matrix, which is defined column-wise by:
\[
F_N=(w_0 \ \dots \ w_{N-1}),
\]
where $(w_k)_{k\in\indices}$ represent the normalized eigenvectors of a circulant matrix, defined in \eqref{circ-eigenv}. Note that $F_N$ is a unitary matrix (that is, $F_N^*F_N=I_N=F_NF_N^*$) since the eigenvectors form an orthonormal basis of $\C^N$.
Then, it holds
\[
J = \begin{pmatrix}
F_N \L_{rr}F_N^* & F_N \L_{\theta r}F_N^*\\
F_N \L_{r\theta}F_N^* & F_N \L_{\theta\theta}F_N^*
\end{pmatrix},
\]
where each $\L_{jl}$ for $j,l\in\{r,\theta\}$ is the diagonal matrix of eigenvalues of each of the four blocks, and where $F_N^*$ denotes the transpose-conjugate of the matrix $F_N$.

We define $P\in M_{2N}(\C)$ as the two-block diagonal matrix $P=I_2 \otimes F_N$, for $\otimes$ the Kronecker product, and  $I_2$ the $2\times 2$ identity matrix. As $P$ is constructed from two unitary blocks, it is also unitary and therefore the following matrix has the same eigenvalues as $J$:
\[
P^*JP = \begin{pmatrix}
 \L_{rr} &  \L_{\theta r}\\
 \L_{r\theta} & \L_{\theta\theta}
\end{pmatrix}.
\]
Each block is now a $N \times N$ diagonal matrix, where each block's eigenvalues can be computed directly using \eqref{circ-eigenv} and \eqref{circ-eigenvalues}. For each $k\in\{0,\cdots,N-1\}$,
\begin{align*}
    (\Lambda_{rr})_{k}&=\lambda_k^{rr}=\frac{-2R^2}{\tau \beta^2}+\sum_{l=1}^{N-1}a_l\cos(\phi_l)\left(e^{ik\phi_l}-1\right),\\
   \frac{1}{R}(\Lambda_{\theta r})_{k}&=\frac{1}{R}\lambda_k^{\theta r}=-R\lambda_k^{r\th}=-R(\L_{r\th})_k=\sum_{l=1}^{N-1}a_l\sin(\phi_l)\left(1-e^{ik\phi_l}\right),\\
    (\Lambda_{\theta\theta})_{k}=&\lambda_{k}^{\theta\theta}=\sum_{l=1}^{N-1}a_l\cos(\phi_l)\left(e^{ik\phi_l}-1\right).
\end{align*}
Thus we can relabel indices to acquire a block diagonal matrix of $N$ different $2 \times 2$ matrices:
\[
M_k=\begin{pmatrix}
    \l^{rr}_{k} & \l^{\theta r}_{k} \\
    \l^{r\theta}_k & \l^{\theta\theta}_{k}
\end{pmatrix}, \ \ \ k\in\{0,\cdots,N-1\}.
\]
The first submatrix $M_0$ is given by
\[
M_0=\begin{pmatrix}
   \frac{-2R^2}{\tau\beta^2} & 0\\
   0 & 0
\end{pmatrix},
\]
which has eigenvalues $\l_1(M_0)=0$ and $\l_2(M_0)=\frac{-2R^2}{\tau\beta^2}<0$. The first eigenvalue is zero due to the rotational twisted state producing a one parameter family of solutions parameterized by $c\in\mathbb{S}$. The second eigenvalue of $M_0$ is always negative and in particular its distance from zero is modulated by the size of $\tau$.\\

The rest of the eigenvalues of $J$ are computed by solving the characteristic equations originating from the matrices
\[
    M_k=\begin{pmatrix}
        \lambda_k^{rr} & \lambda_k^{\theta r}\\
        \lambda_k^{r\theta} & \lambda_k^{\theta \theta}
    \end{pmatrix}, \ \ \ k\in\{1,\cdots, N-1\}.
\]
Thus they are given by
\begin{align}
    \lambda_{1,2}(M_k)=\frac{1}{2}(\lambda_k^{rr}+\lambda_k^{\theta\theta})\pm\frac{1}{2}\sqrt{(\lambda_k^{rr}+\lambda_k^{\theta\theta})^2-4(\lambda_k^{rr}\lambda_k^{\theta\theta}-\lambda_k^{r\theta}\lambda_k^{\theta r})}, \ \ k\in\{1,\cdots, N-1\}.\nonumber
\end{align}
If $\Re(\lambda_{1,2}(M_k))<0$ for all $k\in\{1,\cdots,N-1\}$, then the circulant structure is linearly stable. Note now that $\rho, \omega, \phi_j$ are independent of $\tau$, but $R$ depends on $\tau$ as given in \eqref{ms2}. Hence the dependence on $\tau$ is found only in the term $\lambda_k^{rr}$ which can be rewritten as
\begin{align*}
    \lambda_{k}^{rr}=\frac{-2R^2}{\tau\beta^2}+\lambda_{k}^{\theta\theta}.
\end{align*}
Note that as in the Kuramoto case, the cone condition $\bar{a}\in\mathcal{C}_N$ exactly guarantees that 
\begin{align*}
    \Re(\lambda_{k}^{\theta\theta})<0, \ \ k\in \{1,\cdots, N-1\}.
\end{align*}
Further, there is an implicit dependence on $\tau$ found in each of $\lambda_{k}^{r\theta}$ and $\lambda_{k}^{\theta r}$ as they each depend on $R$. However, since they only appear multiplied together, the term $\lambda_{k}^{r\theta}\lambda_{k}^{\theta r}$ cancels the dependence in $R$, and hence is a bounded complex value independent of $\tau$. Also note that according to the expression \eqref{ms2} for $R(\tau)=\beta\sqrt{\tau(\frac{1}{\tau}-d+\rho)}$, it holds 
\begin{align*}
    \lambda_{1,2}(M_k)=d-\rho-\frac{1}{\tau}+\lambda_{k}^{\theta \theta}\pm \sqrt{\left(d-\rho-\frac{1}{\tau}\right)^2+\lambda_k^{r\theta}\lambda_k^{\theta r}}.
\end{align*}
Thus, achieving stability is guaranteed by the following geometric lemma:
\begin{Lemma}\label{Lem:technical}
    Let $z_1=-x+iy$ for $x>0$, and $y\in \mathbb{R}$, and let $z=z_1 + \sqrt{(z_1-z_2)^2+z_3}$ with $\mathrm{Re}(z_2)<0$, and $z_2, z_3 \in \mathbb{C}$ fixed. If 
    \begin{align*}
        x>\frac{1}{-2\mathrm{Re}(z_2)}(\mathrm{Re}(z_2)^2+(y-\mathrm{Im}(z_2))^2+|z_3|),
    \end{align*}
    then $\mathrm{Re}(z)<0$.
\end{Lemma}
\begin{proof}[Proof of Lemma \ref{Lem:technical}].
    Letting $w=(z_1-z_2)^2+z_3$ we need to show $|w|<x^2$. Using the triangle inequality,
    \begin{align*}
        |w|\leq |z_1-z_2|^2+|z_3|=(x+\mathrm{Re}(z_2))^2+(y-\mathrm{Im}(z_2))^2+|z_3|.
    \end{align*}
    Setting the above strictly less than $x^2$ yields the inequality.
\end{proof}
Since the cone condition guarantees $\Re(\lambda_{k}^{\theta\theta})<0$ and we supposed $\frac{1}{\tau}>d-\rho$, setting $z_1^k=d-\rho-\frac{1}{\tau}+\lambda_{k}^{\theta \theta}$ implies $\mathrm{Re}(z_1^k)<0$.
Using that $\lambda_{k}^{\theta\theta},\lambda_{k}^{r\theta}\lambda_{k}^{\theta r}$ are independent of $\tau$, and setting $z_2^k=\lambda_k^{\theta\theta}$ (so then $\Im(z_2^k)=y$), and $z_3^k=\lambda_{k}^{r\theta}\lambda_{k}^{\theta r}$, we get a sufficient bound: if

\begin{align*}
    \frac{1}{\tau}>\max_{k=1,...N-1}\left(-\frac{|\l_k^{r\th}\l_k^{\th r}|}{2\Re(\l_k^{\th\th})}+\frac12\Re(\l_k^{\th\th})+d-\rho\right),
\end{align*}
then the rotational twisted state is stable. Note that if the maximum is nonpositive, then this holds for all $\tau$ up to the critical relaxation time $\tau^*$. Let $S$ be the set of indices such that $\left(-\frac{|\l_k^{r\th}\l_k^{\th r}|}{2\Re(\l_k^{\th\th})}+\frac12\Re(\l_k^{\th\th})+d-\rho\right)>0$. Then 
since the above condition is only sufficient, this implies the existence of a threshold 
\[\tau^{**}\geq \min_{k\in S}\left(-\frac{|\l_k^{r\th}\l_k^{\th r}|}{2\Re(\l_k^{\th\th})}+\frac12\Re(\l_k^{\th\th})+d-\rho\right)^{-1},\] such that for all $\tau<\tau^{**}$, the rotational twisted state is linearly stable. Moreover, note that in the case $d-\rho>0$, it should necessarily hold $\tau^{**}<\tau^*=(d-\rho)^{-1}$ in order to have existence of the rotational twisted state.
\end{proof}

\subsection{Perturbation of the Hopf-parameter}
 The Hopf-parameter $\beta$ serves as a target amplitude for an uncoupled Stuart-Landau oscillator, while in the context of flocking dynamics~\eqref{eq:fishmodel} it represents a target velocity for an agent. It is known that when this parameter is replaced by a family of heterogeneous parameters $(\beta_j)_{j\in\indices}$ depending on the labels, 
the model~\eqref{eq:fishmodel} with $c_{\mathrm{AR}}=0$ fails to flock as the alignment mechanism is not powerful enough to overcome the mismatch in velocities \cite{LRS}. However, within the context of first-order synchronization models, this type of heterogeneity has been shown to facilitate synchronous behavior~\cite{millan2025synchronizationcoupledstuartlandauoscillators}. In order to investigate whether the \textit{rotational twisted state} can survive heterogeneity in $(\beta_j)_{j\in\indices}$, we introduce a linear perturbation of $\beta$ by taking $\varepsilon>0$ and $(\eta_j)_{j\in\{1,\cdots,N\}}\in\R^N$ such that $\beta_j(\varepsilon) = \beta+\varepsilon\eta_j$ for all $j\in\{1,\cdots,N\}$, and study the existence and stability of a \textit{heterogeneous} rotational twisted state for small values of $\varepsilon$, solution to:
\begin{align}\tag{SL-H}\label{eq:STheterogeneous}
    \ddt z_j=\sum_{l=1}^NA_{jl}(z_l-z_j)+\frac{1}{\tau}\left(1-\frac{|z_j|^2}{\beta_j^2}\right)z_j, \ \ \ z_j\in \mathbb{C}, \ \ j\in\indices.
\end{align}

To this end, we reintroduce the difference variables $\Th_j$ defined in \eqref{eq:Th_j}. Note that the stability of the fixed point $(r_1,\cdots,r_N,\th_1,\cdots,\th_N)=(R,\cdots,R,\phi_1+c,\cdots,\phi_N+c)$ is equivalent to the stability of the fixed point for the variables $(r_1,\cdots,r_N,\Th_1,\cdots,\Th_N)=(R,\cdots,R,\frac{2\pi}{N},\cdots,\frac{2\pi}{N})$. We take advantage of this fact to prove the persistence of the rotational twisted states in the one-neighbor circulant matrix case.

\begin{maintheorem}
Let $\alpha>0$ and let $A\in M_N(\mathbb{R})$ be the one-neighbor circulant matrix generated by the vector $a=(0,\alpha,0\cdots,0)\in \R^N$.
Let $(\eta_j)_{j\in\{1,\cdots,N\}}\in\R^N$, $0<|\varepsilon|\ll1$ and for all $j\in\{1,\cdots,N\}$, $\beta_j = \beta+\varepsilon\eta_j$.
Then there exists $\delta>0$ such that for all $\varepsilon\in (-\delta,\delta)$, there exists a heterogeneous rotational twisted state solution to \eqref{eq:STheterogeneous}, defined by 
\begin{align*}
        z_j(t)=R_j e^{i\left(\tilde\omega t+\tilde\phi_j+\tilde\theta\right)},
    \end{align*}
    for some $\tilde{\omega}\in\R$, $(R_j)_{j\in\{1,\cdots,N\}}\in(\R_+)^N$, $(\tilde{\phi}_j)_{j\in\{1,\cdots,N\}}\in \S^N$, and $\tilde\theta\in \S$.
Moreover, there exists $\tilde\delta\leq \delta$ such that if $\varepsilon\in (-\tilde\delta,\tilde\delta)$, then this solution is stable.
\end{maintheorem}

\begin{proof}
 Let 
 \begin{align*}
     F :\R^{2N+2} & \to \R^{2N+1}\\
       (r_1,\cdots,r_N,\Theta_1,\cdots,\Theta_N,\omega,\varepsilon) & \mapsto F(r_1,\cdots,r_N,\Theta_1,\cdots,\Theta_N,\omega,\varepsilon),
       \end{align*}
 defined for each $j\in\{1,\cdots,N\}$ by 
 \begin{align*}
 F_j(r,\Theta,\omega,\varepsilon) = & \frac{1}{\tau}(1-\frac{r_j^2}{\beta_j^2(\varepsilon)})+ \sum_{l=1}^N a_{(l-j)[N]}\left(\frac{r_l}{r_j} \cos(\sum_{k=0}^{l-j-1}\Theta_{(j+k)\langle N\rangle })-1\right) = \frac{(F_r)_j}{r_j},\\
 F_{N+j}(r,\Theta,\omega,\varepsilon) =& \sum_{l=1}^N a_{(l-j)[N]} \frac{r_l}{r_j} \sin(\sum_{k=0}^{l-j-1}\Theta_{(j+k)\langle N\rangle })-\omega,\\
F_{2N+1}(r,\Theta,\omega,\varepsilon) =& \sum_{k=1}^N \Theta_k - 2\pi. 
\end{align*}
 From Proposition \ref{Prop:ExistenceMillingSL}, $F(R,\cdots,R,\frac{2\pi}{N},\cdots,\frac{2\pi}{N},\bar\omega,0) = 0$, where $R$ is defined in \eqref{ms2} and $\bar\omega = \sum_{l=1}^{N-1} a_l \sin(\frac{2\pi l}{N})$ is defined in \eqref{eq:omega}.
 
 From the Implicit Function Theorem, if $J:=\nabla_{(r,\Theta,\omega)}F(R,\cdots,R,\frac{2\pi}{N},\cdots,\frac{2\pi}{N},\bar\omega,0)$ is invertible, then there exists an interval $(-\delta,\delta)$ and a function $g:(-\delta,\delta)\to \R^{2N+1}$ such that for all $\varepsilon \in (-\delta,\delta)$,
 \[ F(g_1(\varepsilon),\cdots,g_{2N+1}(\varepsilon),\varepsilon) = 0. 
 \]
 Moreover, 
 \[
 \left( \frac{\dd g_i}{\dd\varepsilon}\right)_{i\in\{1,\cdots,2N+1\}} = - \left(\nabla_{(r,\Theta,\omega)}F(g_1(\varepsilon),\cdots,g_{2N+1}(\varepsilon),\varepsilon)\right)^{-1}\nabla_\varepsilon F(g_1(\varepsilon),\cdots,g_{2N+1}(\varepsilon),\varepsilon),
 \]
 which allows to compute $g$ numerically on a small interval around $0$.
 
 In the case of just one neighbor, $a_1 = \alpha $ for some $\alpha\in\R$, and $a_j=0$ for $j\neq 1$.
 We compute $\nabla_{(r,\Theta,\omega)}F$ and get, for all $j\in\{1,\cdots,N\}$: 
 \[
 \begin{aligned}
\partial_{r_j} F_j & = -\frac{2r_j}{\tau\beta_j^2(\varepsilon)} -\frac{\alpha  r_{j+1}}{r_j^2}\cos(\Th_j), \\
\partial_{r_{j+1}} F_j & = \frac{\alpha }{r_j} \cos(\Theta_j), \\
\partial_{\Th_j} F_{j} & = -\alpha \frac{r_{j+1}}{r_{j}}\sin(\Th_j), \\
\partial_{r_j} F_{N+j} & = -\alpha \frac{r_{j+1}}{r_j^2}\sin(\Th_j),\\
\partial_{r_{j+1}} F_{N+j} & =  \frac{\alpha }{r_j} \sin(\Theta_j),\\
\partial_{\Th_j} F_{N+j} & =  \alpha  \frac{r_{j+1}}{r_j}\cos(\Theta_j),\\
\partial_{\omega} F_{N+j} & = -1,\\
\partial_{\Th_j} F_{2N+1} & =  1.
\end{aligned}
\]
Evaluated at $(R,\cdots,R,\frac{2\pi}{N},\cdots,\frac{2\pi}{N},\bar\omega,0)$, we get the following $(2N+1) \times (2N+1)$ block matrix:

    
    \[
    J=\left(\begin{array}{ccc|ccc|c}
    & & & & & & 0 \\
      & D_{r,r} & & & D_{\Theta,r} & & \vdots \\
    & & & & & & 0 \\  
         \hline
    & & & & & & -1\\
     &  D_{r,\Theta} & & & D_{\Theta, \Theta} & & \vdots \\
     & & & & & & -1 \\
       \hline
     0 & \cdots & 0 & 1 & \cdots & 1 & 0
    \end{array}\right):=\left(\begin{array}{c|c}
       A  & B \\
         \hline
       C  & D \\
    \end{array}\right),
    \]
where for each $m,n\in\{r,\Theta\}$, $D_{m,n}\in M_N(\mathbb{R})$, and the final column is given by $(0, \cdots,0,-1,\cdots,-1,0)$.  Analogously, the final row is given by $(0,\cdots,0,1,\cdots,1,0)$. To show invertibility of this matrix, we resort to the Schur complement which yields for the four block matrix
\begin{align*}
    \det(J)=\det(D)\det(A-BD^{-1}C).
\end{align*}
In this case we take $A=D_{r,r}\in M_N(\mathbb{R})$ is a circulant matrix generated by the vector 
\[
\begin{pmatrix}
    -\frac{2R}{\tau\beta^2}-\frac{\alpha }{R}\cos\left(\frac{2\pi}{N}\right), & \frac{\alpha }{R}\cos\left(\frac{2\pi}{N}\right), & 0 & \cdots & 0
\end{pmatrix}.
\]

The submatrix $B\in M_{N\times (N+1)}(\mathbb{R})$ has its first $N$ columns given by $D_{\Th,r}=-\alpha \sin\left(\frac{2\pi}{N}\right)I_N$ and the final column being all zeros. Subsequently, $C\in M_{(N+1)\times N}(\mathbb{R})$ is the matrix with the final row being all zeros and the first $N$ rows provided by $D_{r,\Th}$, which is also a circulant matrix generated by the vector
\[
\begin{pmatrix}
    -\frac{\alpha}{R}\sin\left(\frac{2\pi}{N}\right), & \frac{\alpha}{R}\sin\left(\frac{2\pi}{N}\right), & 0 & \cdots & 0
\end{pmatrix}.
\]
Finally, $D\in M_{N+1}(\mathbb{R})$ is the matrix with $D_{\Th,\Th}=\alpha\cos\left(\frac{2\pi}{N}\right)I_N$ giving the upper left $N\times N$ block, the upper right block given by the length $N$ column vector of $-1$'s, and the lower left the length $N$ row vector of $1$'s, and the final element being $0$.

In light of the Schur decomposition, if $D$ and $A-BD^{-1}C$ are invertible, then  $J$ is invertible. Now, proving the invertibility of $A-BD^{-1}C$ implies the need to explicitly compute this matrix, and hence the need to invert the $(N+1)\times (N+1)$ matrix $D$. Luckily, with patience, one can compute this via Gaussian elimination. After performing the requisite linear algebra operations, the resulting matrix $A-BD^{-1}C$ is another $N \times N$ circulant matrix generated by the row
\[
\begin{pmatrix}
-\frac{2R}{\tau \beta^2}-\frac{\alpha}{R}\left(\cos\left(\frac{2\pi}{N}\right)+\frac{\sin^2\left(\frac{2\pi}{N}\right)}{\cos\left(\frac{2\pi}{N}\right)}\right), & \frac{\alpha}{R}\left(\cos\left(\frac{2\pi}{N}\right)+\frac{\sin^2\left(\frac{2\pi}{N}\right)}{\cos\left(\frac{2\pi}{N}\right)}\right), & 0, & \dots & 0 
\end{pmatrix},
\]
which is once again a diagonally dominated negative Laplacian matrix for which the Perron-Frobenius theory provides strictly negative eigenvalues. Hence the Jacobian is invertible and the Implicit Function Theorem provides a unique perturbed \textit{rotational twisted state} for all $\e \in (-\delta, \delta)$.

The stability of such perturbed states (for a potentially smaller interval $(-\tilde\delta,\tilde\delta)$) is inherited from the stability of the homogenous $\beta$ case due to the fact that the eigenvalues of the linearization computed in Section \ref{s:4} evolve continuously with respect to the perturbation.
\end{proof}

The resulting perturbed \textit{rotational states} leads to each oscillator rotating at a different radius, however with a common rotation speed $\tilde\omega(\varepsilon)$. This result is of particular interest to the future application to the second-order model \eqref{eq:fishmodel}. If these results carry over to this setting, then it would imply that with heterogeneous desired speeds within the group, schooling behavior would remain impossible, but that the diameter of the group may not actually diverge, and milling phenomena could emerge instead. We provide numerical simulations for the $\beta$-perturbed first-order model in Section \ref{s:7}.

\section{Full Synchronization}\label{s:5}
Results thus far have been related to the \textit{rotational twisted states}, however for the Stuart-Landau system we saw that even for circulant matrices, when $d-\rho>0$, such states cease to exist when $\tau>\tau^*=(d-\rho)^{-1}$. 
\subsection{Unconditional Synchronization} 
In \cite{millan2025synchronizationcoupledstuartlandauoscillators}, analogous to the Kuramoto theory, full synchronziation results were provided in the complete-graph case as long as initial data is confined to a half plane. However, we can provide a stronger result here. In this section we provide an unconditional synchronization result for general nonnegative matrices, with identical Hopf-parameters $\beta_j\equiv \beta>0$, according to the following equation:
\begin{align}\tag{SL}\label{3rddeq}
    \ddt z_j=\sum_{l=1}^NA_{jl}(z_l-z_j)+\frac{1}{\tau}\left(1-\frac{|z_j|^2}{\beta^2}\right)z_j, \ \ \ z_j\in \mathbb{C}, \ \ j\in\indices.
\end{align}

\begin{Theorem}
Let $A\in M_N(\mathbb{R}_+)$ be a general matrix with non-negative entries.
\begin{itemize}
    \item[i)] When $A$ is balanced (i.e., $\sum_k A_{jk}=\sum_k A_{kj}$), consider the second smallest eigenvalue $\lambda_2(L_s)$ counted without multiplicity, where $L\in M_N(\mathbb{R})$ is the Laplacian of $A$, and $L_s=\frac{1}{2}(L+L^T)\in M_N(\mathbb{R})$. Suppose that $\lambda_2(L_s)>0$, and also $\tau> \frac{1}{\lambda_2(L_s)}$. Then, the system \eqref{3rddeq} converges exponentially to a fully synchronized state. Specifically, we have
    $$\sum_{j=1}^N|z_j(t)-\bar z(t)|^2\leq C e^{-2(\lambda_2(L_s)-\frac{1}{\tau})t},$$
    where we define the center of mass $\bar z(t):=\frac{1}{N}\sum_{j=1}^Nz_j(t)$.
    \item[ii)] When $A$ is not necessarily balanced, consider its associated scrambling coefficient
    $$\eta(A):=\min_{1\leq j,k\leq N}\sum_{l=1}^N\min\{A_{jl},A_{kl}\}.$$
    Suppose that $\eta(A)>0$, and also $\tau>\frac{1}{\eta(A)}$. Then, the system \eqref{3rddeq} converges exponentially to a fully synchronized state. Specifically, we have
    $$\max_{1\leq j,l\leq N}|z_j(t)-z_l(t)|\leq C e^{-(\eta(A)-\frac{1}{\tau})t}.$$
\end{itemize}
\end{Theorem}

\begin{proof}
For the first part, write $z_j=x_j+{\rm i} y_j$, and set the vectors $x=(x_1,\ldots,x_N)\in \mathbb{R}^N$, $y=(y_1,\ldots,y_N)\in\mathbb{R}^N$ and $z=(z_1,\ldots,z_N)\in \mathbb{R}^{2N}$. Then, \eqref{3rddeq} can be written as
$$\ddt x=-Lx+H^x(z),\quad \ddt y=-Ly+H^y(z),$$
where we define the functions
$$H^x_j(z_j)=\frac{1}{\tau}(1-\frac{|z_j|^2}{\beta^2})x_j,\quad H^y_j(z_j)=\frac{1}{\tau}(1-\frac{|z_j|^2}{\beta^2})y_j,\quad j=1,\ldots,N$$
together with the vector fields
\begin{align*}
H^x(z_1,\ldots,z_N)&:=(H^x_1(z_1),\ldots,H^x_N(z_N)),\\ H^y(z_1,\ldots,z_N)&:=(H^y_1(z_1),\ldots,H^y_N(z_N)).
\end{align*}
Let $\bar z:=\frac{1}{N}\sum_{j=1}^Nz_j$, which we can write by its real and imaginary part as $\bar{z}=\bar{x}+{\rm i} \bar{y}$.
Letting $\mathcal{A}(t)=\sum_{j=1}^N|z_j(t)-\bar z(t)|^2$, note that we can equivalently write
\begin{align*}
\mathcal{A}(t)&=\sum_{j=1}^N|x_j(t)-\bar x(t)|^2+\sum_{j=1}^N|y_j(t)-\bar y(t)|^2\\
&=|x(t)-\bar x(t)\mathds{1}|^2+|y(t)-\bar y(t)\mathds{1}|^2.
\end{align*}

On the one hand, using that $(x-\bar x\mathds{1})\cdot \mathds{1}=0$, together with the fact that $L\mathds{1}=0$, and also $H^x(\bar z\mathds{1})=\frac{1}{\tau}(1-\frac{|\bar z|^2}{\beta^2})\bar x\mathds{1}$, we have
\begin{align*}
&\frac{1}{2}\ddt |x-\bar x\mathds{1}|^2\\
&\quad =-(x-\bar x\mathds{1})\cdot\left(Lx-\frac{1}{N}((Lx)\cdot \mathds{1})\mathds{1}\right)+(x-\bar x\mathds{1})\cdot \left(H^x(z)-\frac{1}{N}(H^x(z)\cdot \mathds{1})\mathds{1}\right)\\
&\quad =-(x-\bar x\mathds{1})^T L(x-\bar x\mathds{1})+(x-\bar x\mathds{1})\cdot (H^x(z)-H^x(\bar z\mathds{1}))\\
&\quad =-(x-\bar x\mathds{1})^T L_s(x-\bar x\mathds{1})+(x-\bar x\mathds{1})\cdot (H^x(z)-H^x(\bar z\mathds{1}))\\
&\quad \leq -\lambda_2(L_s)|x-\bar x\mathds{1}|^2+(x-\bar x\mathds{1})\cdot (H^x(z)-H^x(\bar z\mathds{1})),
\end{align*}
where in the next to last line we have symmetrized the product $(x-\bar x\mathds{1})L(x-\bar x\mathds{1})$. In the last line we have used that since $A$ is balanced, then $L_s$ is positive semidefinite and $\mathds{1}\in \Ker(L_s)$, and therefore $\lambda_1(L_s)=0$ and by the Courant-Fisher min-max theorem the second-smallest eigenvalue $\lambda_2(L_s)$ admits the variational characterization
$$\lambda_2(L_s)=\min_{w\in \mathds{1}^\perp}\frac{w^T L_s w}{|w|^2}.$$
On the other hand, a similar argument as above implies
\begin{align*}
&\frac{1}{2}\ddt |y-\bar y\mathds{1}|^2\\
&\quad\leq  -\lambda_2(L_s)|y-\bar y\mathds{1}|^2+(y-\bar y\mathds{1})\cdot (H(z)-H(\bar z\mathds{1})).
\end{align*}
Since we have
\begin{align*}
(x-\bar x\mathds{1})\cdot (H^x(z)-H^x(\bar z\mathds{1}))&=\frac{1}{\tau}\sum_{j=1}^N(x_j-\bar x)\left((1-\frac{|z_j|^2}{\beta^2})x_j-(1-\frac{|\bar z|^2}{\beta^2})\bar x\right),\\
(y-\bar y\mathds{1})\cdot (H^y(z)-H^y(\bar z\mathds{1}))&=\frac{1}{\tau}\sum_{j=1}^N(y_j-\bar y)\left((1-\frac{|z_j|^2}{\beta^2})y_j-(1-\frac{|\bar z|^2}{\beta^2})\bar y\right),
\end{align*}
then summing the above implies
$$\frac{1}{2}\ddt\mathcal{A}(t)\leq -\lambda_2(L_s)\mathcal{A}(t)+\frac{1}{\tau}\sum_{j=1}^N (z_j-\bar z)\cdot (G(z_j)-\bar G(\bar z)),$$
where we define
$$G(w):=\nabla U(w),\quad U(w):=\frac{1}{2}|w|^2-\frac{1}{4\beta^2}|w|^4,\quad w\in \mathbb{R}^d.$$
By the mean value theorem in integral form we can write
\begin{align*}
&\frac{1}{\tau}\sum_{j=1}^N (z_j-\bar z)\cdot (G(z_j)-\bar G(\bar z))\\
&\quad=\frac{1}{\tau}\sum_{j=1}^N \int_0^1 (z_j-\bar z)^T D^2U(\bar z+\theta (z_j-\bar z))(z_j-\bar z)\,d\theta\leq \frac{1}{\tau}\sum_{j=1}^N|z_j-\bar z|^2,
\end{align*}
where we have noted that $D^2 U(w)=(1-\frac{|w|^2}{\beta^2})I_2-\frac{2}{\beta^2}w\otimes w$, and therefore
$$v^T D^2U(w)v=\left(1-\frac{|w|^2}{\beta^2}\right)|v|^2-\frac{2}{\beta^2}|v\cdot w|^2\leq |v|^2,$$
for all $v\in \mathbb{R}^2$. This implies,
$$\frac{1}{2}\ddt \mathcal{A}(t)\leq -\lambda_2(L_s)\mathcal{A}(t)+\frac{1}{\tau}\mathcal{A}(t),$$
which ends the proof of the first part by Gr{\"o}nwall's lemma.

For the second part, we let $\mathcal{A}(t):=\max_{1\leq j,k\leq N}|z_j(t)-z_k(t)|$. Then computing via Rademacher's lemma and taking $(j,k)$ to be a maximizing pair, we get
    \begin{align*}
        \frac{1}{2} \ddt \mathcal{A}(t)^2 &=\sum_{l=1}^N(z_j-z_k)\cdot\left(A_{jl}(z_l-z_j)-A_{kl}(z_l-z_k)\right)+\frac{1}{\tau}(z_j-z_k)\cdot\left(G(z_j)-G(z_k)\right),
    \end{align*}
where $G$ is defined as in the previous step. On the one hand, we have
\begin{align*}
&\sum_{l=1}^N(z_j-z_k)\cdot\left(A_{jl}(z_l-z_j)-A_{kl}(z_l-z_k)\right)\\
&= -\sum_{l=1}^N\left(A_{jl}(|z_j-z_k|^2-(z_j-z_k)\cdot (z_l-z_k))+A_{kl}(z_j-z_k)\cdot (z_l-z_k)\right)\\
&\leq -\sum_{l=1}^N\min\{A_{jl},A_{kl}\}\left(|z_j-z_k|^2-(z_j-z_k)\cdot (z_l-z_k))+(z_j-z_k)\cdot (z_l-z_k))\right)\\
&\leq -\eta(A)|z_j-z_k|^2,
\end{align*}
where in the second line we have added and substracted $A_{jl}(z_l-z_k)\cdot (z_l-z_k)$ to the first term, and in the third line we have used the non-negativity of $A_{jl}$ and $A_{kl}$, together with the maximality of the pair $(j,l)$, which implies $|z_j-z_k|^2-(z_j-z_k)\cdot (z_l-z_k)\geq 0$ and $(z_j-z_k)\cdot (z_l-z_k)\geq 0$. For the second term we argue as in the previous step and we then obtain
$$\frac{1}{2}\ddt\mathcal{A}(t)^2\leq -\eta(A)\mathcal{A}(t)+\frac{1}{\tau}\mathcal{A}(t).$$
Again, we conclude by Gr{\"o}nwall's lemma.
\end{proof}

\begin{Rem}
This fact is rather astounding in that if  $\frac{1}{\tau}$, for $\tau$ the relaxation time of the Hopf-bifurcation, is less than the generalized algebraic connectivity of a balanced connectivity matrix, then asymptotic states other than full synchronization cannot exist, and the convergence rate is given exactly by the gap between the inverse of the relaxation time and the generalized algebraic connectivity. This applies directly to circulant matrices since they are balanced. On the other hand, for general matrices we cannot appeal to the generalized algebraic connectivity, but we still get a condition provided by the gap between the scrambling coefficient and the inverse of the relaxation time.
\end{Rem}

It should also be stated that the final value of $|z_j|$ was not given by the above proof. In principal, it is possible for all $z_j \to 0$ instead of settling to the chosen amplitude $|z_j^{\infty}|=\beta$.

\subsection{Linear stability of fixed points}
The unconditional synchronization result implies that all $|z_j-z_k| \to 0$ exponentially fast. However, it does not show whether $|z_j|\to \beta$ or $|z_j|\to 0$ which can both be fixed points of the system \eqref{3rddeq}. Here we provide a linear stability analysis of each of these fixed points.

\begin{maintheorem}\label{t:stab}
    Let $A \in M_N(\mathbb{R}_+)$ be any nonnegative connectivity matrix with a directed spanning tree. Then the fixed point $z_j=0$ for all $j=1,\ldots,N$ is linearly unstable while the class of fixed points such that $z_j=z_k$ for all $j,k=1,\ldots, N$ with $|z_j|=\beta$, is linearly stable.
\end{maintheorem}
\begin{proof}
    In order to determine the stability of the zero fixed point, we can no longer use the amplitude-phase representation since phases are not well-defined at zero. Therefore we take advantage of the Wirtinger derivatives in order to compute the linearization. Let
    \begin{align*}
        F_{z_j}=\sum_{l=1}^NA_{jl}(z_l-z_j)+\frac{1}{\tau}\left(1-\frac{|z_j|^2}{\beta^2}\right)z_j, \quad F_{\bar{z}_j}=\sum_{l=1}^NA_{jl}(\bar{z}_l-\bar{z}_j)+\frac{1}{\tau}\left(1-\frac{|z_j|^2}{\beta^2}\right)\bar{z}_j.
    \end{align*}
    Then the Jacobian matrix $J\in M_{2N}(\mathbb{R})$  is given once again by 4 blocks of size $N\times N$:
    \[
    J=\left(\begin{array}{c|c}
       D_{z,z}  & D_{\bar{z},z}  \\
         \hline
       D_{z,\bar{z}}  & D_{\bar{z}, \bar{z}}
    \end{array}\right).
    \]
    The Wirtinger derivatives take each $z_j$ and the complex conjugate $\bar{z}_j$ to be independent variables. This also allows one to avoid the complications that arise from the fact that $|z_j|^2$ is not holomorphic. Now the necessary derivatives are given by
    \begin{align*}
        &\partial_{z_j}F_{z_j}=-d_j+\frac{1}{\tau}\left(1-2\frac{|z_j|^2}{\beta^2}\right),\\
        &\partial_{z_l}F_{z_j}=A_{jl},\\
        &\partial_{\bar{z}_j}F_{z_j}=-\frac{z_j^2}{\tau\beta^2},\\
        &\partial_{\bar{z}_l}F_{z_j}=0
    \end{align*}
    and by symmetry we get the same values for
    \begin{align*}
        &\partial_{\bar{z}_j}F_{\bar{z}_j}=-d_j+\frac{1}{\tau}\left(1-2\frac{|z_j|^2}{\beta^2}\right),\\
        &\partial_{\bar{z}_l}F_{\bar{z}_j}=A_{jl},\\
        &\partial_{z_j}F_{\bar{z}_j}=-\frac{\bar{z}_j^2}{\tau\beta^2},\\
        &\partial_{z_l}F_{\bar{z}_j}=0,
    \end{align*}
    Thus the upper left and bottom right blocks are identical. And are given by the negative Laplacian plus a diagonal term. Meanwhile the bottom left and upper right blocks are diagonal.

    First let us check the zero fixed point. The corresponding Jacobian is given by
    \[
    J_0=\left(\begin{array}{c|c}
       D_{z,z}(0)  & D_{\bar{z},z}(0)  \\
         \hline
       D_{z,\bar{z}}(0)  & D_{\bar{z}, \bar{z}}(0)
    \end{array}\right),
    \]
    with 
    $$D_{z,z}(0)=-L(A)+\frac{1}{\tau}\mathrm{I}_N=D_{\bar{z},\bar{z}}(0),$$
    and
    $$D_{\bar{z},z}=D_{z,\bar{z}}=0.$$
    Thus the entire matrix $J$ is a negative Laplacian matrix plus a constant positive diagonal. Hence $\lambda_1=\frac{1}{\tau}>0$  since the first eigenvalue of the Laplacian is zero. Thus the zero fixed point is unstable.\\

    Next to check at the fully synchronized state such that $|z_j|=\beta$ with $z_j=z_k$ for all $j,k=1,\ldots,N$. Then $z_j=\beta e^{i\theta}$ for some $\theta \in \mathbb{S}$.

    \[
    J_{\beta}=\left(\begin{array}{c|c}
       D_{z,z}(\beta)  & D_{\bar{z},z}(\beta)  \\
         \hline
       D_{z,\bar{z}}(\beta)  & D_{\bar{z}, \bar{z}}(\beta)
    \end{array}\right),
    \]
    with 
    $$D_{z,z}(\beta)=-L(A)-\frac{1}{\tau}\mathrm{I}_N=D_{\bar{z},\bar{z}}(\beta),$$
    and
    $$D_{\bar{z},z}(\beta)=-\frac{1}{\tau}e^{2i\theta}\mathrm{I}_N=\overline{D_{z,\bar{z}}(\beta)}.$$
    Since the upper left and lower right blocks are identical and are diagonally dominated Laplacians of the original connectivity matrix, we can diagonalize them with $\lambda_1,\ldots,\lambda_N$ such that $-\frac{1}{\tau}=\lambda_1>\Re(\lambda_j)$ for each $j=2,\ldots, N$. Therefore, just as before, we can compute the eigenvalues of the Jacobian by solving the characteristic equations given by
    \begin{align*}
        (\lambda_j-\mu_j(J_{\beta}))^2-\frac{1}{\tau^2}=0
    \end{align*}
    Thus $$\mu_j(J_{\beta})=\lambda_j\pm\frac{1}{\tau}$$
    which gives $\mu_1(J_{\beta})=0$ and $\Re(\mu_j)(J_{\beta})<0$ for the remaining $2N-1$ eigenvalues. Once again, the zero eigenvalue is due to the fact that the final angle $\theta$ is an emergent property and the class of synchronized solutions with $|z_j|=\beta$ is linearly stable.
\end{proof}
\begin{Rem}
We make note here that if the original matrix $A$ were circulant, then since $\mu_j(J_{\beta})=\lambda_j\pm\frac{1}{\tau}$ where $\lambda_j$ are the eigenvalues of the Laplacian matrix generated by $A$, then the imaginary parts of the eigenvalues $\mu_j$ are identical to those of the original circulant matrix which induced rotations. Therefore via the linearization, we expect rotational dynamcis near the fixed point to remain stable, which we investigate numerically in section \ref{s:7}.
\end{Rem}
\begin{Rem}
    We further note that the stability for the synchronized state could have also been deduced from the fact that the basin of attraction for such a state is at least as large as one half-plane. This was shown for the complete graph case in \cite{millan2025synchronizationcoupledstuartlandauoscillators}, and we note that the relative adjustments to the connected graph case can be made to guarantee convergence to the synchronized state at an exponential rate, regardless of the relationship between the minimal connectivity and the relaxation time. 
\end{Rem}

The instability of the zero fixed point for all $\tau>0$ implies that  so long as the initial data is not on the stable manifold of the origin (a set of Lebesgue measure zero), then amplitude death cannot occur. Thus for $\tau\geq \frac{1}{\l_2(L_s)}$ for nonnegative balanced graphs and $\tau\geq \frac{1}{\eta}$ for general nonnegative adjacency matrices, convergence to the active synchronized state occurs for almost all initial data so that $|z_j| \to \beta$ and $|z_j-z_k|\to 0$ for all $j,k=1,..,N$.\\

\section{Leader-follower augmented rotating states}\label{s:6}
In Section \ref{s:4}, we saw that the circulant structure could be perturbed and still yield stable rotational states. In this section we continue to address the question of whether the circulant nature of the interaction matrix is necessary for the existence of a rotational twisted state. 
As a first step towards generalizing the class of matrices able to generate rotational twisted states, we explore the dynamics of incorporating leaders and followers into the graph structure. We assume that the $N$ agents are split into  $N_1$ leaders and $N_2$ followers, such that $N=N_1+N_2$. We suppose that the $N_1$ leaders are interconnected via a circulant adjacency matrix satisfying the cone condition, and hence gives rise to a stable rotational state for $\tau$ small enough. The remaining $N_2$ followers are allowed to receive input from the  $N_1$ leaders, but we assume that they do not influence any other agent. The adjacency matrix would then take the following form:
\[
A=\left(\begin{array}{c|c}
       A_1  & A_2  \\
         \hline
       A_3  & A_4
    \end{array}\right)=\left(\begin{array}{c|c}
       C & 0  \\
         \hline
       A_3  & 0
    \end{array}\right),
\]
where $C\in M_{N_1}(\mathbb{R})$ is the circulant graph structure satisfying the cone condition. The submatrix $A_3\in M_{N_2\times N_1}(\mathbb{R})$ dictates the behavior of the $N_2$ followers. As $A_2=0$, the behavior of the first $N_1$ agents is determined only by the circulant matrix $C$, so that depending on the initial data they can converge to the fully synchronized state or to the twisted rotational state. Naturally, if the leaders converge to the fully synchronized state, then as long as the whole adjacency matrix has a directed spanning tree, then the remainder of the $N_2$ agents also synchronize. We are concerned here with the case of the leaders producing a twisted rotational state and seeing whether or not the followers will adopt the same rotational velocity. We begin the investigation once more with the Kuramoto model.

\subsection{Kuramoto Uniform following}
We start with the Kuramoto case in which $\mathcal{N}_1:=\{1,\cdots,N_1\}$ denotes the group of leaders and $\mathcal{N}_2=\{N_1+1,\cdots,N\}$ the group of followers. We further suppose that followers only follow leaders so that $A_2=A_4=0$ and that each follower $k$ will be led by exactly $n_k\in\N$ consecutive leaders, with a coupling strength given by $\frac{\alpha_k}{n_k}\in\R$. In other words, if agent $k$ is a follower and $n_k=3$, then there exists $j \in \mathcal{N}_1$ such that $A_{kj}=A_{k,(j+1)\langle N_1\rangle}=A_{k,(j+2)\langle N_1\rangle}=\frac{\alpha}{3}$. We denote such $j$ by $j=m_k$. The equations are then given by
\begin{equation}\label{eq:KuramotoLeadersFollowers}
    \begin{aligned}
    \ddt \th_j&=\sum_{l=1}^{N_1}(A_1)_{jl}\sin(\th_l-\th_j), \quad j \in \mathcal{N}_1=\{1,\cdots,N_1\},\\
    \ddt \th_k&=\frac{\alpha_k}{n_k}\sum_{l \in \cE(k)}\sin(\th_l-\th_k), \quad k \in \mathcal{N}_2=\{N_1+1,\cdots,N\},
\end{aligned}
\end{equation}
where $A_1$ is a circulant matrix satisfying the cone condition, and $\cE(k):=\{m_k,\cdots,(m_k+n_k-1)\,\langle N_1\rangle \}\subset \mathcal{N}_1$ denotes the $n_k$ leaders to which oscillator $k$ is coupled.
Note that we impose no sign condition on $\alpha_k$, so that even though we use the term ``follower'', if $\alpha_k<0$, agent $k$ is in reality repelled from its ``leaders''.

Supposing further that the leaders are in the basin of attraction of the rotational twisted state with rotation speed $\omega$, then we know from Theorem \ref{t:Kurrotpre} that each leader tends to the following state:
\begin{align}
    \th_j(t)&=\omega t+\frac{2\pi j}{N_1}+\th, \quad j \in \cN_1,\label{kurmilllead1}\\ 
    \omega&=\sum_{l=1}^{N-1}a_l\sin(\phi_l),\label{kurmilllead2}
\end{align}
for some fixed phase shift $\th \in \S$ common to all agents.

The following theorem provides the phase transition in terms of the coupling strength $\alpha_k$, the number of followed leaders $n_k$ and the rotation speed $\omega$.

\begin{maintheorem}
    Consider the solution to the leader-follower system \eqref{eq:KuramotoLeadersFollowers}, with $N_1\geq 2$ and $n_k<N_1$ for all $k\in \mathcal{N}_2=\{N_1+1,\cdots,N\}$. Let $A_1\in M_{N_1}(\mathbb{R})$ be a circulant matrix satisfying the cone condition of Definition \ref{def:cone}, which guarantees the existence and stability of the rotational twisted state \eqref{kurmilllead1} for the agents in $\{1,\cdots,N_1\}$ at angular velocity $\omega:=\sum_{l=1}^{N_1-1}a_l\sin(\frac{2\pi l}{N_1})$.
    Then, the followers $k\in\mathcal{N}_2$ can join the rotational twisted state if and only if for all $k\in\mathcal{N}_2$,
    \begin{align*}
        |\alpha_k|\geq \alpha^*_k:=\frac{|\omega| n_k\sin(\frac{\pi}{N_1})}{\sin\left(\frac{n_k\pi}{N_1}\right)}.
    \end{align*}
    In that case, the rotational twisted state is given by 
    \begin{equation}\label{eq:KurLeadFollEquil}
    \begin{aligned}
    \th_j(t)&=\omega t+\frac{2\pi j}{N_1}+\th, && j \in \{1,\cdots,N_1\},\\ 
    \th_k(t)& = \omega t + \theta_k^* +\theta,  &&  k \in \{N_1+1,\cdots,N\},
    \end{aligned}
    \end{equation}
in which for each $k\in\mathcal{N}_2$, the phase shift $\theta_k^*$ can take one of two values: 
\begin{align}
\theta_k^*=(\theta_k^*)_+:=&\frac{2\pi}{N_1}(m_k + \frac{n_k-1}{2})- \arcsin(\mathrm{sgn}(\omega)\frac{\alpha_k^*}{\alpha_k})\label{eq:thetak+},\\
\text{or} \quad \theta_k^*=(\theta_k^*)_-:=&\frac{2\pi}{N_1}(m_k + \frac{n_k-1}{2})- \pi+\arcsin(\mathrm{sgn}(\omega)\frac{\alpha_k^*}{\alpha_k}).\label{eq:thetak-}
\end{align}
Moreover, the rotational twisted state satisfying for all $k\in\{N_1+1,\cdots,N\}$
\begin{equation*}
\theta_k^* = 
    \left\{\begin{array}{ll}
    (\theta_k^*)_+, & \text{ if } \mathrm{sgn}(\alpha_k)=\mathrm{sgn}(\omega),\\
    (\theta_k^*)_-, & \text{ if } \mathrm{sgn}(\alpha_k)=-\mathrm{sgn}(\omega),
    \end{array}\right.
\end{equation*}
is locally stable.
\end{maintheorem}

\begin{proof}
    We seek the solution
    \begin{align}\label{kurmillunif}
        \ddt \th_k=\frac{\alpha_k}{n_k}\sum_{l\in \cE(k)}\sin(\th_l-\th_k)=\omega.
    \end{align}
    If $\ddt\th_k=\omega$, then integrating gives
    \begin{align}\label{th:foll}
        \th_k(t)=\omega t+\th_k^*+\th,
    \end{align}
    where $\th_k^*$ is the phase offset that will be determined below.
Plugging in \eqref{th:foll} and \eqref{kurmilllead1} into \eqref{kurmillunif} yields the self-consistency equation
\[
    \frac{\alpha_k}{n_k}\sum_{l \in \cE(k)}\sin\left(\frac{2\pi l}{N_1}-\th_k^*\right)=\omega.
\]
Letting $\cE(k)=\{m_k,\cdots,(m_k+n_k-1)\,\langle N_1\rangle\}$, the sum is given by
\[
    \frac{\alpha_k}{n_k}\sum_{l=m_k}^{m_k+n_k-1}\sin\left(\frac{2\pi l}{N_1}-\th_k^*\right)=\omega.
\]
Using the finite sum formula for sines, we get
\begin{align}\label{thk*}
    \frac{\alpha_k\sin\left(\frac{n_1\pi}{N_1}\right)\sin\left(\frac{(n_k-1+2 m_k)\pi}{N_1}-\th_k^*\right)}{n_k\sin\left(\frac{\pi}{N_1}\right)}=\omega.
\end{align}
    Since $|\sin\left(\frac{(n_1+1)\pi}{N_1}-\th_k^*\right)|\leq1$, solving for $\alpha_k$ gives the exact phase transition associated with agent $k$. Finally, we can solve for $\th_k^*$ by taking the inverse sine, and consequently obtain two possible values \eqref{eq:thetak+} and \eqref{eq:thetak-}.

We can further prove the stability of such a state.
    We study the Jacobian of the fixed point map. However, we need new difference variables to discuss the phase behavior of the followers in relation to the leaders. Since each follower $k \in \cN_2$ follows $n_k$ adjacent leaders $\cE(k)=\{m_k,\cdots,(m_k+n_k-1)\langle N_1\rangle\}$, we need only observe the difference between the follower and one of the leaders, for instance the first one. We define
    \begin{align*}
        \Th_{m_kk}:=\th_{m_k}-\th_k \ \mod 2\pi,
    \end{align*}
    where $m_k$ denotes the first leader that oscillator $k$ is coupled to. In this sense, we can write the remainder of the leaders to which oscillator $k$ is coupled via adding and subtracting the oscillator $\th_{m_k}$, that is, we have
    \begin{align*}
        \th_{j}-\th_k=\th_j-\th_{m_k}+\Th_{m_k k},\quad j\in \cE(k)
    \end{align*}

    Therefore, appending the followers to the circulant leader group incorporates exactly $N_2$ more variables to the original Kuramoto circulant system: 
    \begin{align*}
\ddt \Th_{m_k k} =&
\sum_{j\in \cN_1} (A_1)_{m_k j} \sin (\theta_j-\theta_{m_k}) - \frac{\alpha_k}{n_k} \sum_{j\in \cE(k)}\sin(\theta_{j}-\theta_k), \quad k\in \cN_2,
    \end{align*}
    for the last $N_2$.
The Jacobian then takes the form
\[
J=\left(\begin{array}{c|c}
       J_1  & J_2  \\
         \hline
       J_3  & J_4
    \end{array}\right),
\]
where $J_1\in M_{N_1}(\mathbb{R})$ is the original Jacobian of the circulant structure given by $A_1$. The submatrix $J_2\in M_{N_1\times N_2}(\mathbb{R})$ is computed via the partial derivatives of $\Theta_{m_k k}$ acting on the fixed-point maps of the leader agents. As none of the followers influence the leaders, $J_2=0$. Hence the eigenvalues of the full matrix $J$ are given by the union of the eigenvalues of $J_1$ and $J_4$. As seen in Section \ref{sec:KurStability}, $\l_1(J_1)=0$ and every other eigenvalue of $J_1$ has negative real part. Hence, we need only compute the eigenvalues of $J_4$.

Further, since there are no interfollower interactions, $J_4$ is a diagonal matrix given by
\begin{align*}
    J_4=\mathrm{diag}\left\{-\frac{\alpha_k}{n_k}\sum_{l \in \cE(k)}\cos\left(\frac{2\pi l}{N_1}-\th_k^*\right)\right\}_{k\in \{1,\cdots,N_2\}}.
\end{align*}
Using the finite sum formula for cosines, we get
\begin{align*}
    \lambda_k(J_4)=-\frac{\alpha_k}{n_k}\sum_{l \in \cE(k)}\cos\left(\frac{2\pi l}{N_1}-\th_k^*\right)=-\frac{\alpha_k}{n_k}\frac{\sin\left(\frac{n_k\pi}{N_1}\right)}{\sin\left(\frac{\pi}{N_1}\right)}\cos\left(\frac{(n_k-1+2 m_k)\pi}{N_1}-\th_k^*\right).
\end{align*}
We now study the sign of the eigenvalues associated with the two possible values of $\theta_k^*$ given in \eqref{eq:thetak+}-\eqref{eq:thetak-}.
It holds 
\[
\lambda_k(J_4) = 
\left\{\begin{array}{ll}
-\frac{\alpha_k}{n_k} \frac{\sin\left(\frac{n_k\pi}{N_1}\right)}{\sin\left(\frac{\pi}{N_1}\right)}
\cos(\arcsin(\mathrm{sgn}(\omega) \frac{\alpha_k^*}{\alpha_k})), & \text{ if } \theta_k^*=(\theta_k^*)_+,\\
-\frac{\alpha_k}{n_k} \frac{\sin\left(\frac{n_k\pi}{N_1}\right)}{\sin\left(\frac{\pi}{N_1}\right)}
\cos(\pi-\arcsin(\mathrm{sgn}(\omega) \frac{\alpha_k^*}{\alpha_k})), & \text{ if } \theta_k^*=(\theta_k^*)_-.
\end{array}\right.
\]
Since $\arcsin:[-1,1]\to [-\frac{\pi}{2}, \frac{\pi}{2}]$, this implies: 
\begin{align*}
\lambda_k(J_4)<0 & \quad \text{ if } \theta_k^*=(\theta_k^*)_+ \text{ and } \mathrm{sgn}(\alpha_k)=\mathrm{sgn}(\omega), \quad \text{ or if } \theta_k^*=(\theta_k^*)_- \text{ and } \mathrm{sgn}(\alpha_k)=-\mathrm{sgn}(\omega)\\
\lambda_k(J_4)>0 & \quad \text{ if } \theta_k^*=(\theta_k^*)_+ \text{ and } \mathrm{sgn}(\alpha_k)=-\mathrm{sgn}(\omega)
\quad \text{ or if } \theta_k^*=(\theta_k^*)_- \text{ and } \mathrm{sgn}(\alpha_k)=\mathrm{sgn}(\omega).
\end{align*}

\end{proof}

\begin{Rem}
Note that if $n_k=1$, i.e. agent $k$ follows only agent $m_k\in\cN_1$, it holds $\theta_k^*=\frac{2\pi m_k}{N_1} - \arcsin(\mathrm{sgn}(\omega)\frac{\alpha_k^*}{\alpha_k})$, which implies that the necessary coupling strength in order for agents $k$ and $m_k$ to be exactly superimposed (i.e. $\theta_k^*=\frac{2\pi i_k}{N_1}$) is infinitly large.\\
Further, note that $n_k<N_1$ guarantees that $\sin\left(\frac{n_1\pi}{N_1}\right)>0$. This further implies that if $n_k=N_1$, so that the follower $k$ follows every leader equally, then the followers cannot join the rotational state for any finite coupling strength $\alpha_k$.\\
Lastly, note that in the phase shifts $(\theta_k^*)_+$ and $(\theta_k^*)_-$, the angle $\frac{2\pi}{N_1} (m_k+\frac{n_k-1}{2})$ represents the phase shift of the average phase of the leaders followed by $k$.
\end{Rem}

Briefly, we note the special case where $\omega=0$. This occurs if the circulant adjacency matrix is also symmetric. In this case, the twisted state is a fixed point of the system. Thus the equation for a follower must also relax to a fixed point,
\begin{align*}
    \ddt \th_k=\frac{\alpha_k}{n_k}\sum_{l \in \cE(k)}\sin(\th_l-\th_k)=0,
\end{align*}
for any $\alpha_k>0$. By symmetry we see that $$\th_k^{\infty}=\frac{1}{n_k}\sum_{l\in \cE(k)}\th_l^{\infty},$$
where the $\th_l^{\infty}$ have settled to some rotation of the roots of unity. This provides a substantial difference between the follower dynamics of a twisted state with graph-induced rotations and those of a stationary twisted state.

\subsection{Stuart-Landau: Unique leader-follower dynamics}
In the Stuart-Landau case, the amplitude dependence adds further complications. Therefore we restrict ourselves to the case of each follower following a unique leader, and to the case of nonnegative connections, $A_{jk}\geq 0$. Again we assume $A_4=0$, and now for each $k \in \mathcal{N}_2$, there is exactly one $m_k \in \mathcal{N}_1$ for which the agent follows with connection strength $a_{m_k k}$. Let us write the equations here:
\begin{align*}
    \ddt z_j&=\sum_{l=1}^{N_1}(A_1)_{jl}(z_l-z_j)+\frac{1}{\tau}\left(1-\frac{|z_j|^2}{\beta^2}\right)z_j, \quad j \in \mathcal{N}_1,\\
    \ddt z_k&=a_{m_k k}(z_{m_k}-z_k)+\frac{1}{\tau}\left(1-\frac{|z_k|^2}{\beta^2}\right)z_k, \quad k \in \mathcal{N}_2,
\end{align*}
where $A_1$ is a nonnegative ciruclant matrix satisfying the cone condition.

As agent $k$ follows a unique agent $j=m_k\in\mathcal{N}_1$, we seek the solution where agent $k$ adopts the rotational velocity, $\omega$, however, the limiting amplitude $r_k$ and the corresponding phase difference $\th_j-\th_k=\Th_{jk}$ may be different than that of the leaders. 

We seek to find a regime for which we can prove that the follower can join the group of leaders. Consider the fixed-point map of the phase difference $\Th_{jk}$ and the amplitude $r_k$:
\begin{align}
    F(\Th_{jk})&=a_{jk}\frac{R}{r_k}\sin(\Th_{jk})-\omega=0\label{fpm:lf1},\\
    F(r_k)&=\frac{1}{\t}\left(1-\frac{r_k^2}{\beta^2}\right)r_k+a_{jk}(\cos(\Th_{jk})R-r_k)=0.\label{fpm:lf2}
\end{align}
Taking inspiration from the $n_k=1$ case for Kuramoto, we seek conditions such that \eqref{fpm:lf1} is satisfied with $\sin(\Th_{jk})=1$, and hence $\cos(\Th_{jk})=0$. Plugging this ansatz into \eqref{fpm:lf2} yields
\begin{align*}
    \frac{1}{\t}\left(1-\frac{r_k^2}{\beta^2}\right)r_k-a_{jk}r_k=0,
\end{align*}
solving for $a_{jk}$ yields
\begin{align*}
    a_{jk}=\frac{1}{\t}\left(1-\frac{r_k^2}{\beta^2}\right).
\end{align*}
Then utilizing $\sin(\Th_{jk})=0$ we have $a_{jk}=\frac{r_k}{R}\omega$. Equating these two gives the following quadratic equation in $r_k$
\begin{align*}
    \frac{R}{\t\o\b^2}r_k^2+r_k-\frac{R}{\o\t}=0,
\end{align*}
which provides
\begin{align*}
    r_k=\frac{-\b^2\pm \b\sqrt{\b^2+\frac{4R^2}{\o^2\t^2}}}{\frac{2R}{\t\o}},
\end{align*}
and, subsequently, a sufficient threshold for the existence of the leader-follower rotational twisted state
\begin{align}\label{critSL}
    a^*=\frac{\t\o^2\b^2}{2R^2}\left(\sqrt{1+\frac{4R^2}{\o^2\t^2\b^2}}-1\right).
\end{align}
Note that $\lim_{\tau\to 0}a^*(\t)=|\omega|$, which is exactly the condition provided for $n_1=1$ in the previous Kuramoto leader-follower coupling. Further note that $\tau,\o,\b,R$ are all known. We can now prove stability of the leader-follower twisted rotational state for $a_{jk}>a^*$. First, let us prove the following lemma.
\begin{Lemma}
    Let $k\in\mathcal{N}_2$, and suppose $a_{m_k k}>a^*$. Then there exists a rotational solution to \eqref{fpm:lf1}-\eqref{fpm:lf2} such that $\cos(\Th_{i_k k}),\sin(\Th_{i_k k})>0$.
\end{Lemma}
\begin{proof}
    Let $k\in\mathcal{N}_2$.
    For conciseness of notation, we rename $m_k=j$.
    Without loss of generality suppose $\omega>0$. Let $a_{jk}>a^*$, the first equation yields the bounds for $r_k$:
    \begin{align*}
        0\leq r_k\leq \frac{a_{jk}R}{\omega}.
    \end{align*}
    Now using the first equation to substitute in for cosine in the second equation we define
    \begin{align*}
        G(r_k,a_{jk}):=\frac{1}{\t}\left(1-\frac{r_k^2}{\beta^2}\right)r_k+a_{jk}\left(R\sqrt{1-\frac{\omega^2r_{k}^2}{a_{jk}^2R^2}}-r_k\right)=0,
    \end{align*}
    where we chose the positive branch of cosine, and finding a solution to $G(r_k,a_{jk})=0$ for each $a_{jk}>a^*$ will yield the result.
    
    Now $G(0,a_{jk})=a_{jk}R>0$, and we compute at the upper endpoint:
    \begin{align}
    G\left(\frac{a_{jk}R}{\omega},a_{jk}\right)=\frac{a_{jk}R}{\omega}\left(\frac{1}{\tau}\left(1-\frac{a_{jk}^2R^2}{\omega^2\beta^2}\right)-a_{jk}\right).
    \end{align}
    
    Now, letting $$H(a_{jk}):=\left(\frac{1}{\tau}\left(1-\frac{a_{jk}^2R^2}{\omega^2\beta^2}\right)-a_{jk}\right),$$
    we have $\mathrm{sgn}\left(H(a_{jk})\right)=\mathrm{sgn}\left(G\left(\frac{a_{jk}R}{\omega},a_{jk}\right)\right)$ and $H(a^*)=0$. Differentiating yields
    \begin{align*}
        H'(a_{jk})=-\frac{2a_{jk}R^2}{\tau\omega^2\beta^2}-1<0.
    \end{align*}
    Hence $H$ is strictly decreasing in $a_{jk}$, so $H(a_{jk})<0$ for all $a_{jk}>a^*$. Thus $G\left(\frac{a_{jk}R}{\omega},a_{jk}\right)<0$ and by the Intermediate Value Theorem there exists $r_k(a_{jk}) \in \left(0,\frac{a_{jk}R}{\omega}\right)$ and $\Th_{jk}(a_{jk})=\arcsin\left(\frac{\omega r_k(a_{jk})}{a_{jk}R}\right)$.
\end{proof}
\begin{Rem}
Note that $a^*$ is only a sufficient bound for the leader-follower rotational state to exist in the regime where the cosine of the leader-follower interaction is positive. Indeed, it appears that values below $a^*$ can produce a stable leader-follower rotational state such that the cosine becomes negative, however we leave a further investigation of this study to the future.
\end{Rem}
\begin{maintheorem}
    Let $A\in M_N(\mathbb{R})$ be a nonnegative adjacency matrix such that $A_1\in M_{N_1}(\mathbb{R})$ is circulant, satisfying the cone condition of Definition \ref{def:cone}, $A_2=0=A_4$, and for each $k \in \mathcal{N}_2$, there exists a unique $m_k \in \mathcal{N}_1$ such that $(A_3)_{m_k k}=a_{m_k k}>0$, and every other element of $A_3$ is zero.
    Suppose further that $\tau<\tau^{**}$ and  $a_{m_k k}>a^*$ for all $k\in\cN_2$. Then there exists a linearly stable rotational state such that for $z_j=r_je^{i\th_j}$
    \begin{align*}
        z_j&=Re^{i(\omega t+\frac{2\pi j}{N_1}+\th)}, \quad j \in \mathcal{N}_1,\\
        z_k&=r_ke^{i(\omega t+\th_k^*)}, \quad k \in \mathcal{N}_2.
    \end{align*}
\end{maintheorem}
\begin{proof}
    The existence and stability of the rotational state for the first $N_1$ agents has already been shown. As previously, for conciseness of notation, we denote $m_k=j$. Now as $a_{jk}>a^*$, there exists $r_k$ and $\Th_{jk}$ such that $\cos(\Th_{jk})>0$, $\sin(\Th_{jk})>0$ satisfying \eqref{fpm:lf1}-\eqref{fpm:lf2}.

To see the stability, we must return to the linearization techniques we used previously. However, the explicit values of $\Th_{jk}$ and $r_k$ will not be necessary to make conclusions on the  stability.

Once more we compute the Jacobian in the variables $(r_j,\th_j)$ for $j \in \mathcal{N}_1$ and for the remaining $N_2$ agents in the variables $(r_k, \Th_{jk})$, for each $k \in \mathcal{N}_2$ where $\Th_{jk}=\th_j-\th_k=\theta_{m_k}-\theta_k$ is the phase difference between the follower and the unique leader.

The Jacobian then takes the form
\[
J=\left(\begin{array}{c|c}
       J_1  & J_2  \\
         \hline
       J_3  & J_4
    \end{array}\right),
\]
where $J_1$ is the original Jacobian of the circulant structure. The submatrix $J_2$ is computed via the partial derivatives of $r_k$ and $\Theta_{jk}$ acting on the fixed-point maps of the leader agents. As none of the followers influence the leaders, $J_2=0$. Again, we need only compute the eigenvalues of $J_4$. To start, we compute the elements of $J_4$. They are given by taking the partial derivatives of the following fixed-point maps
\begin{align*}
    F(r_k)&=\frac{1}{\tau}(1-\frac{r_k^2}{\beta^2})r_k+a_{jk}(\cos(\Th_{jk})r_j-r_k),\\
    F(\Th_{jk})&=\sum_{l\in \cN_1}A_{jl}\frac{r_l}{r_j}\sin(\th_l-\th_j)-a_{jk}\frac{r_j}{r_k}\sin(\Th_{jk}).
\end{align*}
Their partial derivatives for each $k,k' \in \mathcal{N}_2$ with $k\neq k'$ and $j=m_k\in\mathcal{N}_1$, $j'=m_{k'}\in\mathcal{N}_1$ the respective leaders of $k$ and $k'$, are given by
$$
\begin{aligned}
    \p_{r_k} F(r_k)&=\frac{1}{\tau}(1-\frac{3r_k^2}{\b^2})-a_{jk}, &&
    \p_{r_{k'}} F(r_k)=0,\\
    \p_{\Th_{jk}}F(r_k)&=-a_{jk}\sin(\Th_{jk})r_j, &&
    \p_{\Th_{j'k'}}F(r_k)=0,\\
    \p_{r_k} F(\Th_{jk})&=a_{jk}\frac{r_j}{r_k^2}\sin(\Th_{jk}), &&
    \p_{r_{k'}}F(\Th_{jk})=0,\\
    \p_{\Th_{jk}}F(\Th_{jk})&=-a_{jk}\frac{r_j}{r_k}\cos(\Th_{jk}), &&
    \p_{\Th_{j'k'}}F(\Th_{jk})=0.
\end{aligned}
$$
Therefore the eigenvalues come from each of the $2\times 2$ submatrices:
\[
(J_4)_{ij}=\begin{pmatrix}
    \frac{1}{\tau}(1-\frac{3r_k^2}{\b^2})-a_{jk} & -a_{jk}\sin(\Th_{jk})r_j\\
    a_{jk}\frac{r_j}{r_k^2}\sin(\Th_{jk}) & -a_{jk}\frac{r_j}{r_k}\cos(\Th_{jk})
\end{pmatrix}.
\]
Computing the eigenvalues of $(J_4)_{ij}$ leads to solving
\begin{align*}
\l^2+\l\left(\frac{2r_k^2}{\t\beta^2}+2a_{jk}\cos(\Th_{jk})\frac{R}{r_k}\right)+a_{jk}^2\frac{R^2}{r_k^2}+\frac{2r_kR}{\tau\beta^2}a_{jk}\cos(\Th_{jk})=0.
\end{align*}
As $\cos(\Th_{jk})>0$, the Routh-Hurwitz criterion guarantees that both of the eigenvalues have negative real part. Therefore the augmented milling state is linearly stable.

\end{proof}

We again note the special case of the stationary twisted state where $\omega=0$. With only one leader we see that the phase of the follower must settle to the same phase of the leader:
\begin{align*}
    \ddt{\th_k}=a_{m_k k}\frac{R}{r_k}\sin(\th_{m_k}-\th_k)=0,
\end{align*}
for all $a_{m_k k}>0$. The amplitude dynamics provide the fixed-point map
\begin{align*}
    F(r_k)&=\frac{1}{\tau}\left(1-\frac{r_k^2}{\b^2}\right)r_k+a_{m_k k}(\cos(\th_{m_k}-\th_k)R-r_k)=0,\\
    &=\frac{1}{\tau}\left(1-\frac{r_k^2}{\b^2}\right)r_k+a_{m_k k}(R-r_k)=0.
\end{align*}
From this, we see that $r_k^{\infty}\in(R,\b)$ with $$\lim_{a_{m_k k}\to 0} r_k^{\infty}(a_{m_k k})=\beta, \qquad \lim_{a_{m_k k}\to \infty} r_k^{\infty}(a_{m_k k})=R.$$

\subsection{Leader-Follower discussion}
These results are reminiscent of the minimal coupling strength needed for the synchronization of Kuramoto (and Stuart-Landau) oscillators with heterogeneous natural frequencies. Indeed, the rotation of the leading agents, induced by the graph, leads to the agents (viewed from the perspective of the followers) as having natural frequency $\omega$, despite actually being identical oscillators.

In this way, we see that agents with weak connectivity to the leaders relative to the rotation speed of the twisted state cannot synchronize; rather, as the leader agents rotate it creates a periodic forcing on the follower so that it fluctuates periodically. As $a_{m_k k}$ approaches the critical threshold, the follower begins to rotate as well, but at a fluctuating rotation speed slower than that of the twisted state.

Such below threshold states provide what is known as a chimera state where a fraction of the popolution of oscillators produce a common behavior, while the rest remain desynchronized. As far as the authors are aware, both the emergence of the \textit{rotational twisted state} as well as this chimera state for the Kuramoto and Stuart-Landau oscillator systems is novel. Indeed, the common behavior produced in the chimera state is not full synchronization, but a rotational twisted state. Of further interest is the fact, that all oscillators have natural frequency zero, implying the desire for all of the oscillators \textit{not} to rotate. Despite having a connected adjacency matrix, there exists a stable asymptotic state which is given by a fixed rotation speed in one portion of the population and periodically forced desynchronization in the other portion of the population.

We highlight the fact that this chimera state cannot exist for symmetric interleader connectivity matrices since they cannot produce the subgroup rotational dynamics. The graph-induced rotations are necessary for seeing this particular phenomenon.

\section{Numerical Simulations}\label{s:7}

The final section is devoted to a numerical investigation of each of the regimes we have analytically studied.

\subsection{Basins of attraction for rotational and synchronized states.}\label{Sec:sim1}

In this first subsection, we illustrate the co-existence of rotational and synchronized states for the same parameters. 
In the following simulations, the parameters are chosen to be $\beta=1$, $\tau=0.1$, $N=15$, and $A\in M_N(\mathbb{R})$ is the circulant matrix generated by the vector $a = (0,5,1,2,1,0,\cdots,0)$. The corresponding interaction network is represented in Figure \ref{Fig:Network}.
This choice of parameters allows us to compute the phase transition parameter $\tau^* = 0.3077$, as well as the rotation radius $R = 0.8216$. 
We denote $x_j := \mathrm{Re}(z_j)$ and $y_j := \mathrm{Im}(z_j)$.
For conciseness of notation, we sometimes identify $z_j\in\C$ with its vector form in $z_j\in\R^2$ given by $z_j=(x_j,y_j)$.
We also study the evolution of the following three order parameters: 
\begin{itemize}
\item the polarization functional 
\[
P(t) = \left|\frac{1}{N}\sum_{j=1}^N \frac{\dot z_j(t)}{|\dot z_j(t)|}\right|;
\]
\item the milling functional
\[
M(t) = \left|\frac{1}{N}\sum_{j=1}^N \frac{(Z_j(t)-Z_c(t))\wedge \dot Z_j(t)}{|Z_j(t)-Z_c(t)| |\dot Z_j(t)|}\right|,
 \]
 in which for each $j$, $Z_j(t):= (x_j(t),y_j(t),0)\in\R^3$ denotes the extensions of $z_j(t)$ to $\R^3$;
\item the radial swarming functional
\[
S(t) = \frac{1}{N}\sum_{j=1}^N\frac{ \left|(z_j(t)-z_c(t))\cdot \dot z_j(t)\right|}{|z_j(t)-z_c(t)| |\dot z_j(t)|}.
 \]
 \end{itemize}

Note that by definition $0\leq P,M,S\leq 1$ for all time $t$. 
Moreover, $P = 1$ if and only if there exists $u\in\R^2$ such that for all $j,k\in\{1,\cdots,N\}$, $\frac{\dot z_j}{|\dot z_j|}=u$, i.e. all agents have aligned velocities.  
The milling functional $M$ satisfies $M = 1$ if and only if for all $j\in\{1,\cdots,N\}$, $\ddt z_j$ is perpendicular to $z_j-z_c$, and there exists $\alpha\in\{-1,1\}$ (common to all $j$) such that the angle between the two vectors satisfies $(z_j-z_c, \ddt z_j)=\alpha \frac{\pi}{2}$.
The swarming functional satisfies $S=1$ if and only if for all $j\in\indices$, $\ddt z_j$ and $z_j-z_c$ are co-linear, that is if the velocities are all radial with respect to the center of mass.

\begin{figure}[h!]
\centering
\includegraphics[width=0.5\textwidth]{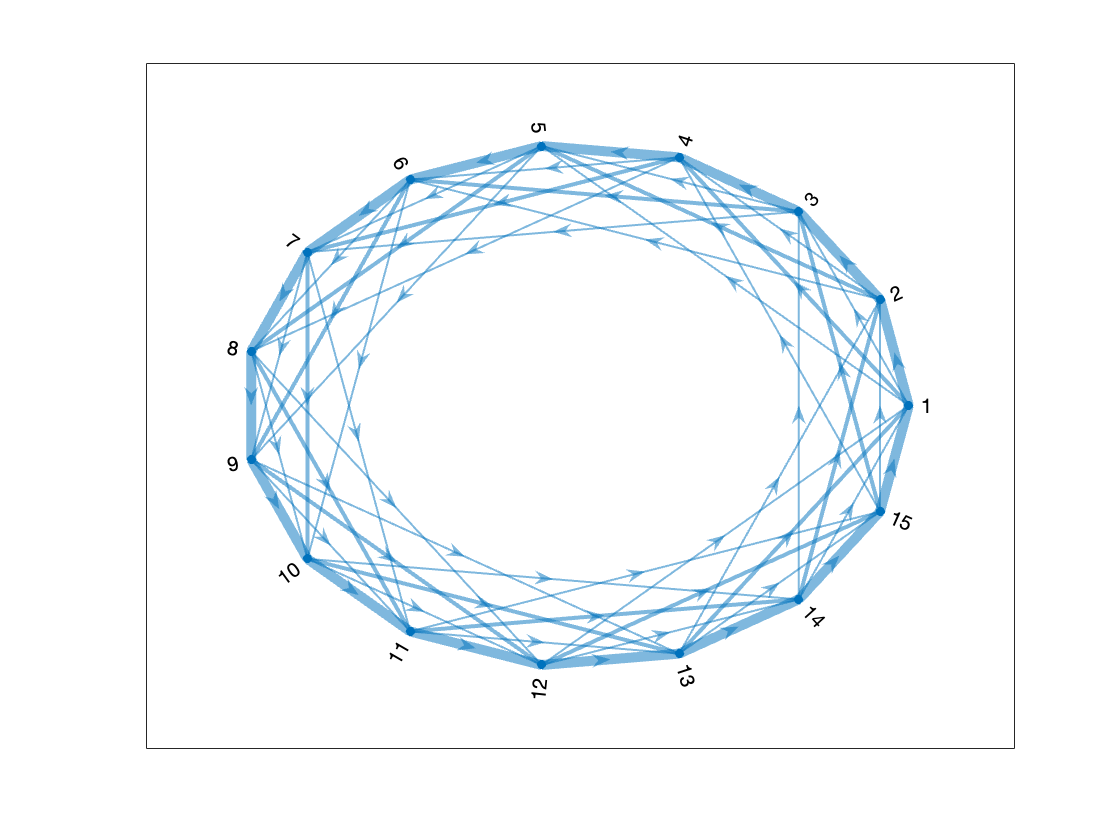}
\caption{Interaction network corresponding to the circulant matrix $A\in M_{15}(\mathbb{R})$ generated by the vector $a=(0,5,1,2,1,0,\cdots,0)$}. Line thickness is proportional to the edge weight.\label{Fig:Network}
\end{figure}

Figure \ref{Fig:StableMilling} shows the evolution of the system for a choice of initial positions in the basin of attraction of the rotational state. The radii are shown to converge to the equilibrium rotation radius $R$, the angle differences $\theta_{j+1}-\theta_j$ converge to the equilibrium value of $\frac{2\pi}{N}$, and the milling order parameter $M$ is shown to converge to 1.
\begin{figure}[h!]
\centering
\includegraphics[width=0.32\textwidth, trim = 2cm 0.5cm 2cm 0.5cm, clip=true]{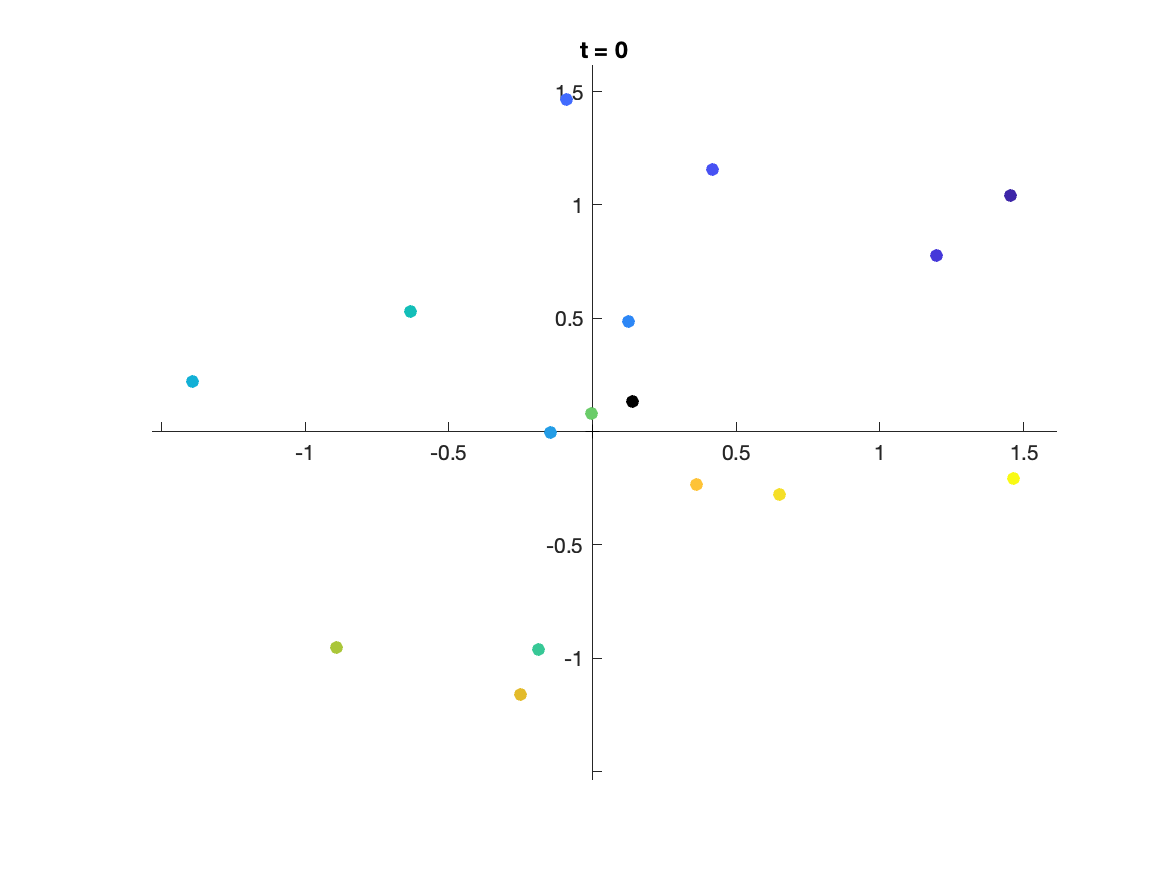}
\includegraphics[width=0.32\textwidth, trim = 2cm 0.5cm 2cm 0.5cm, clip=true]{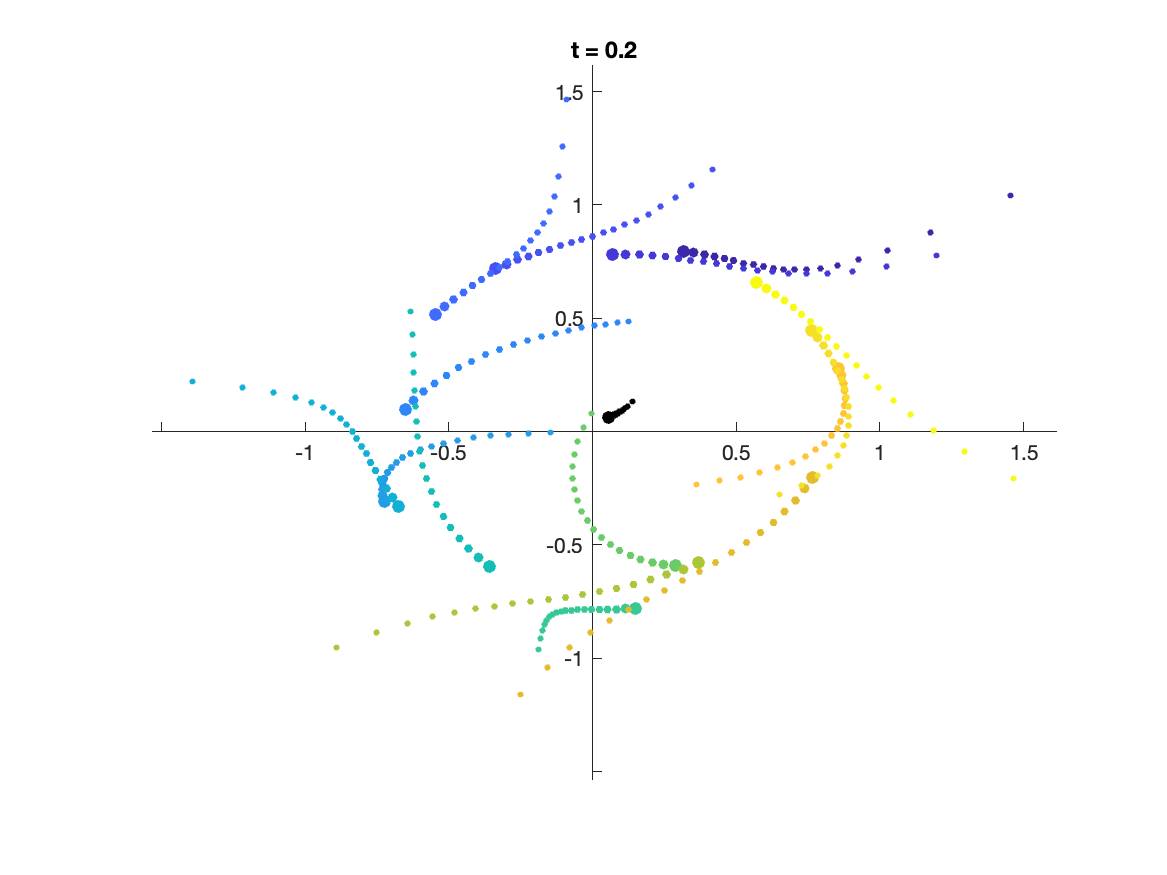}
\includegraphics[width=0.32\textwidth, trim = 2cm 0.5cm 2cm 0.5cm, clip=true]{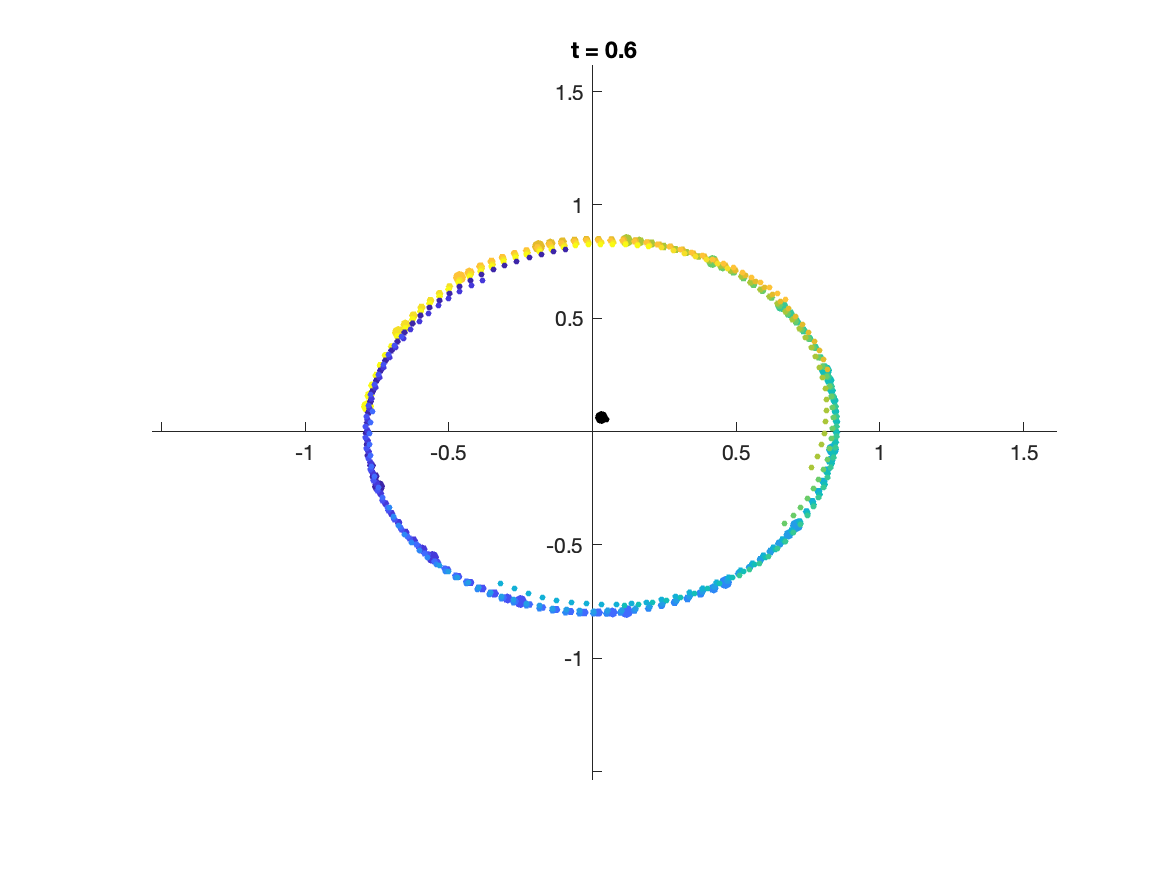}\\
\includegraphics[width=0.32\textwidth]{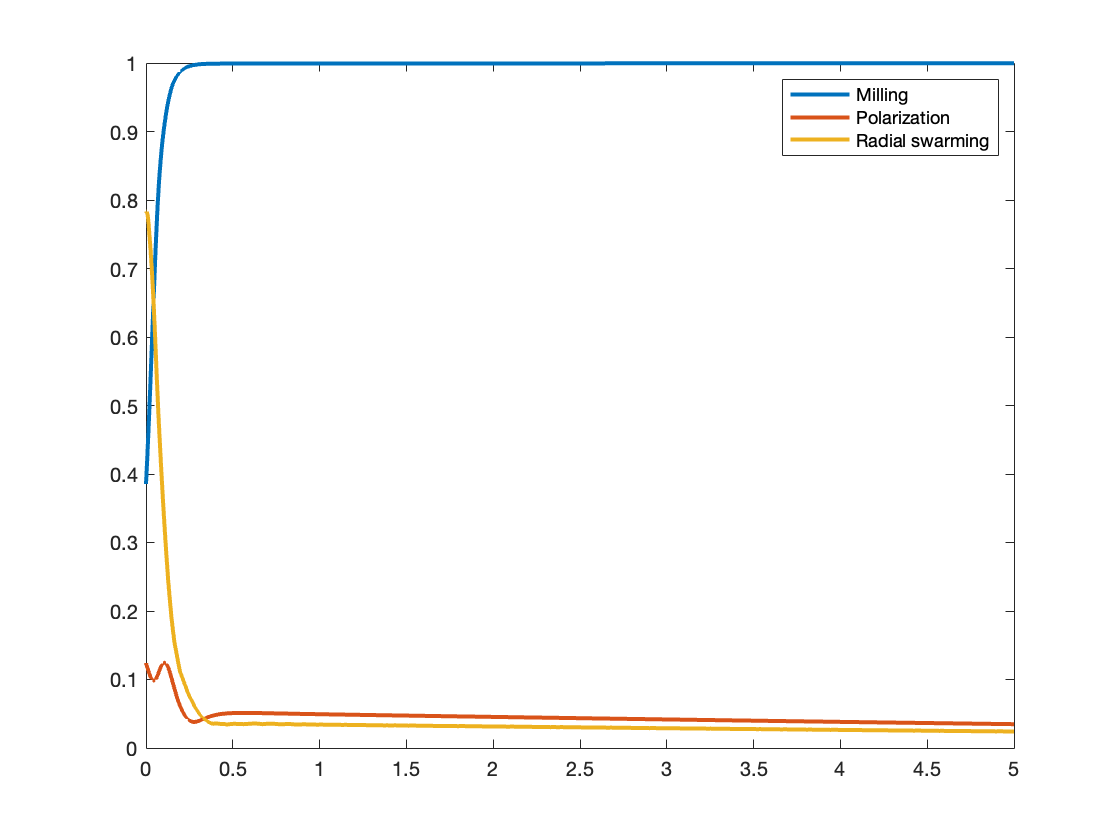}
\includegraphics[width=0.32\textwidth]{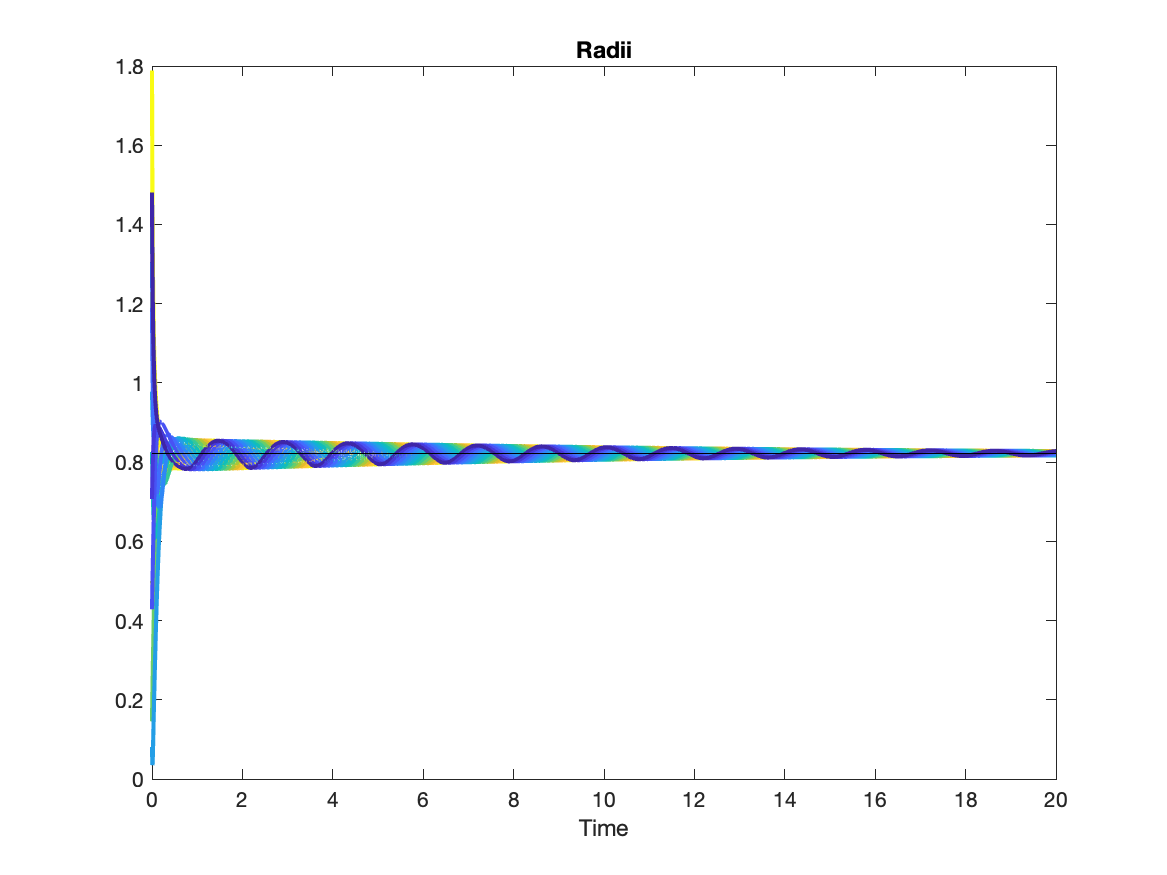}
\includegraphics[width=0.32\textwidth]{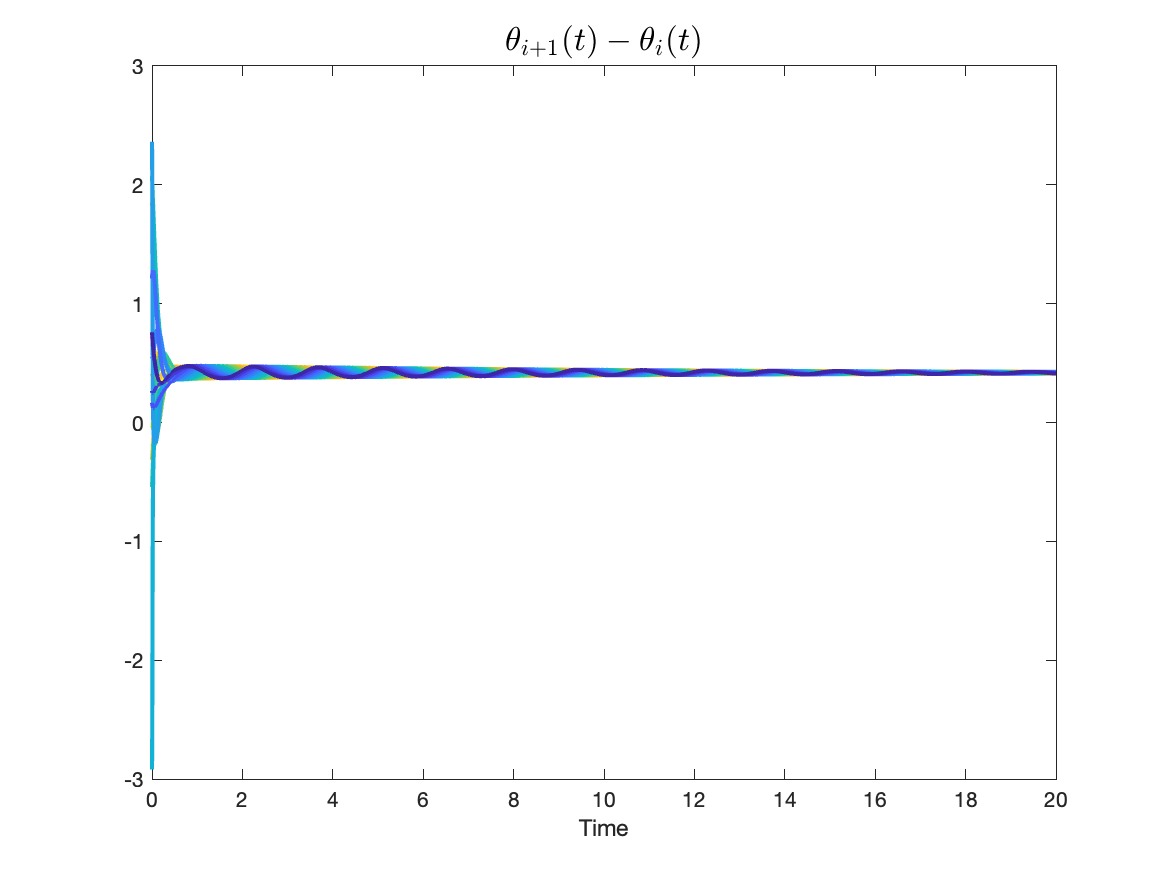}
\caption{Top row. Particles' trajectories at times $t=0$, $t=0.2$ and $t=0.8$. Each agent is represented in a different color, its last position is represented by the largest dot, and its previous positions by smaller dots. The black dot represents the trajectory of the system's center of mass.\\
Bottom row. Left: Evolution of the three order parameters $M$, $P$ and $S$. Center: Evolution of each agent's radius $r_j$. Right: Evolution of the $N$ angle differences $\theta_{j+1}-\theta_j$.   
}\label{Fig:StableMilling}
\end{figure}

Figure \ref{Fig:Sync} shows the evolution of the system for a choice of initial positions in the basin of attraction of the synchronized state, for the same parameters as in Figure \ref{Fig:StableMilling}. The radii are shown to converge to the target radius $\beta=1$ and the angle differences $\theta_{j+1}-\theta_j$ converge to $0$. The radial swarming order parameter $S$ is shwon to converge to 1, after a brief polarized transition phase during which the particles' velocities are close to being aligned.

\begin{figure}[h!]
\centering
\includegraphics[width=0.24\textwidth, trim = 2cm 0.5cm 2cm 0.5cm, clip=true]{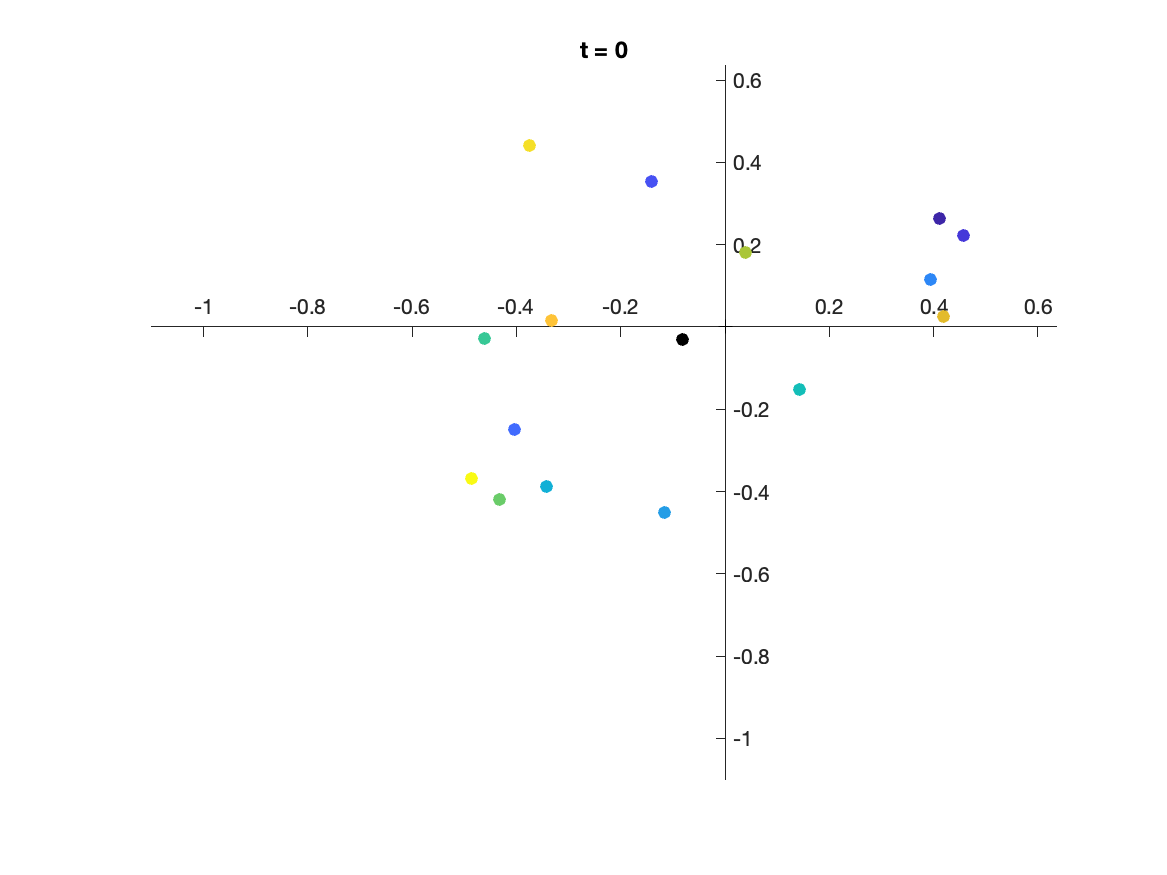}
\includegraphics[width=0.24\textwidth, trim = 2cm 0.5cm 2cm 0.5cm, clip=true]{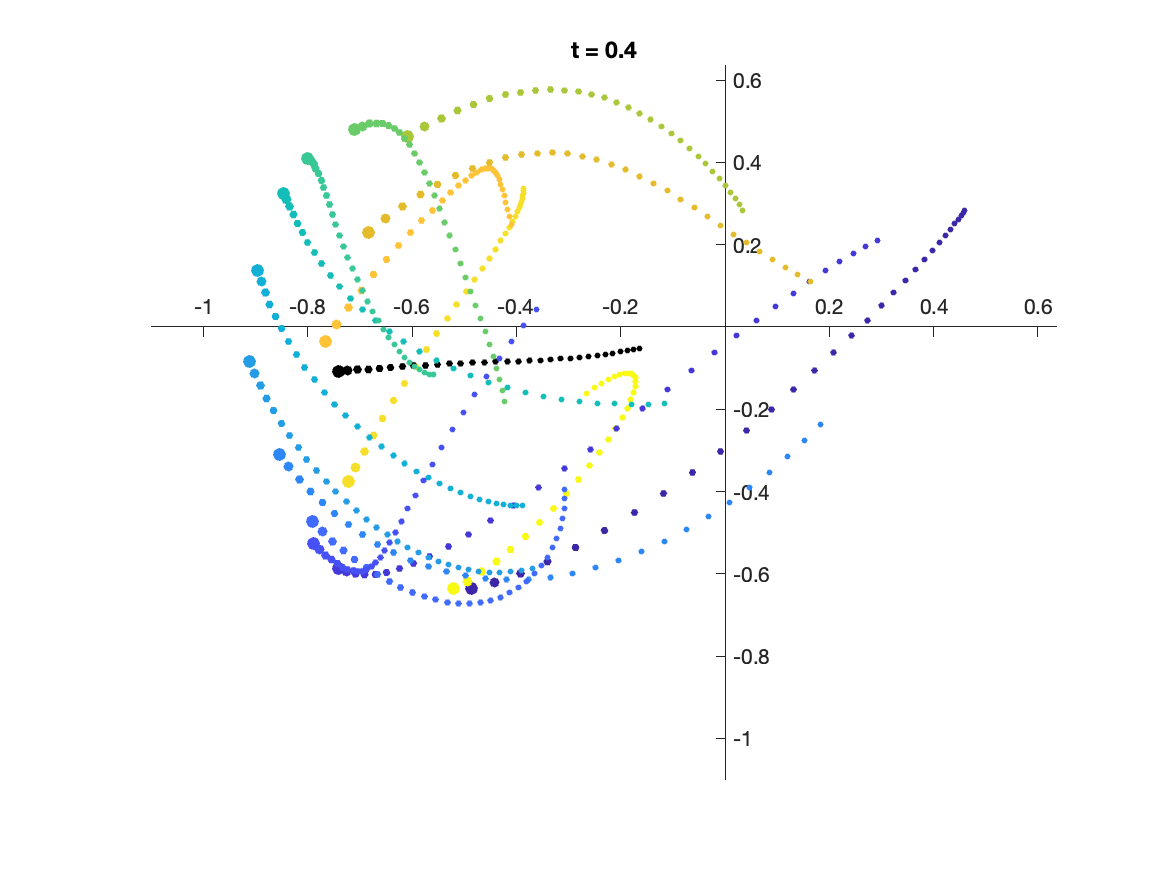}
\includegraphics[width=0.24\textwidth, trim = 2cm 0.5cm 2cm 0.5cm, clip=true]{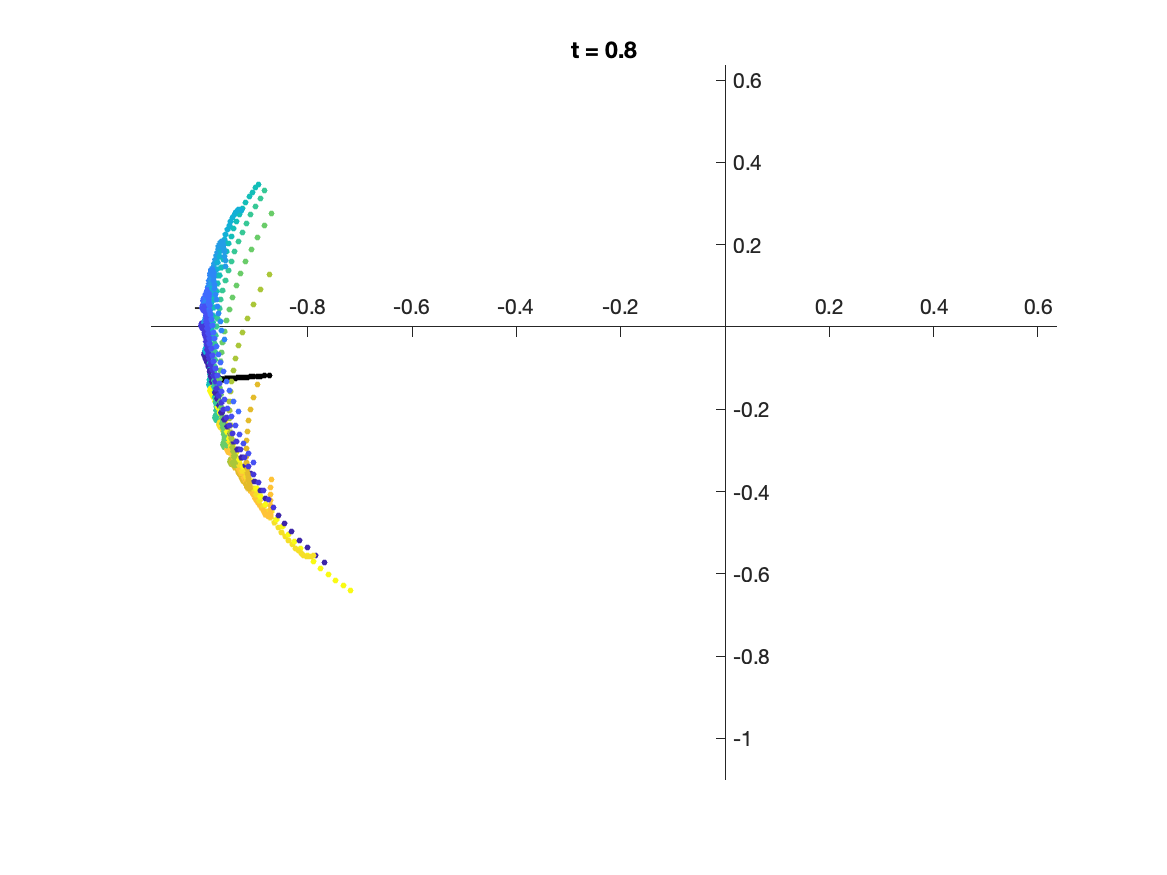}
\includegraphics[width=0.24\textwidth, trim = 2cm 0.5cm 2cm 0.5cm, clip=true]{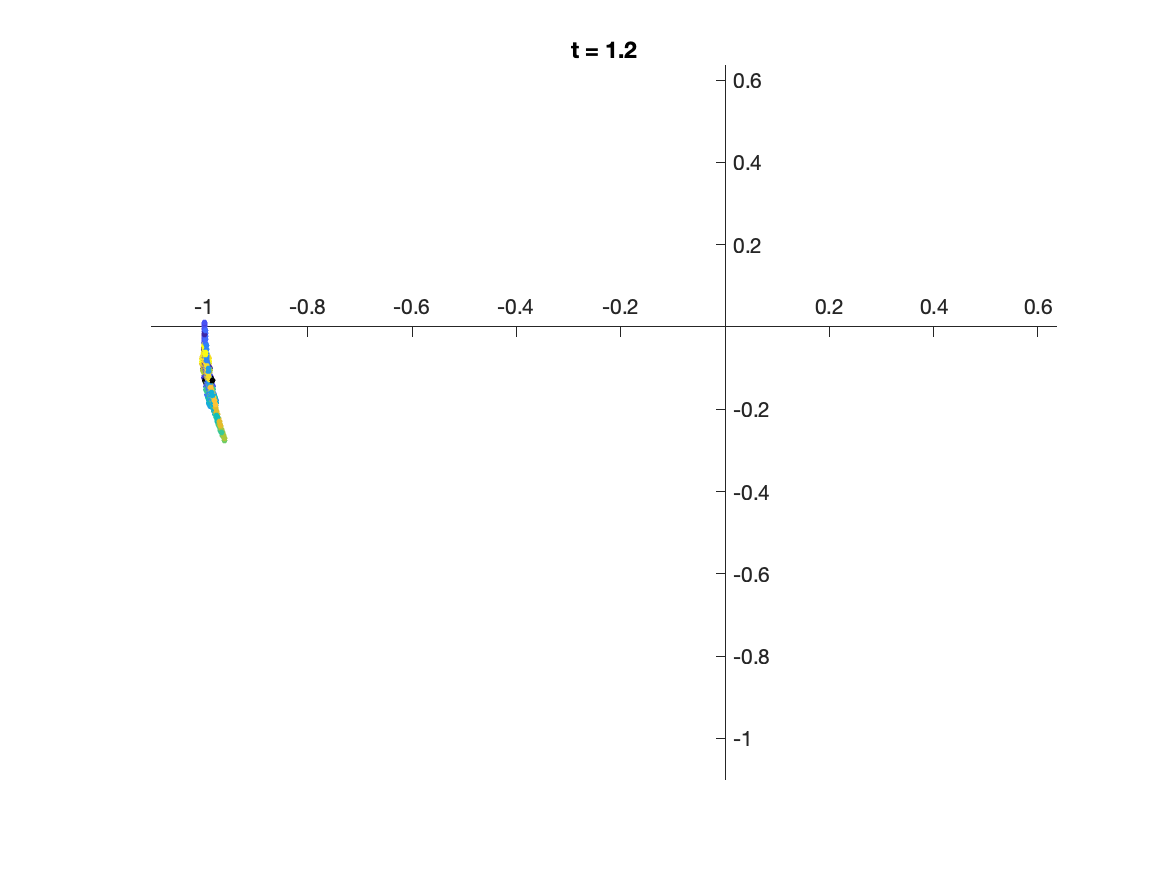}\\
\includegraphics[width=0.32\textwidth]{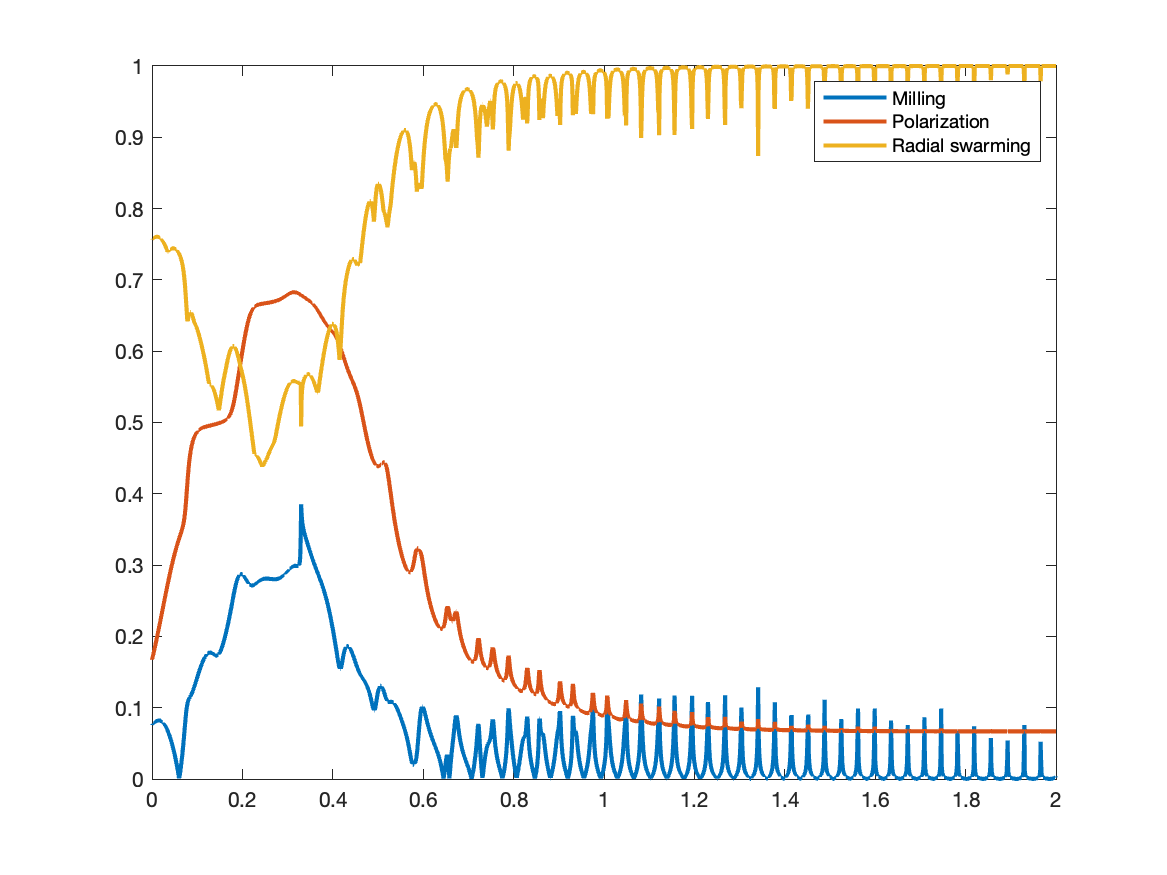}
\includegraphics[width=0.32\textwidth]{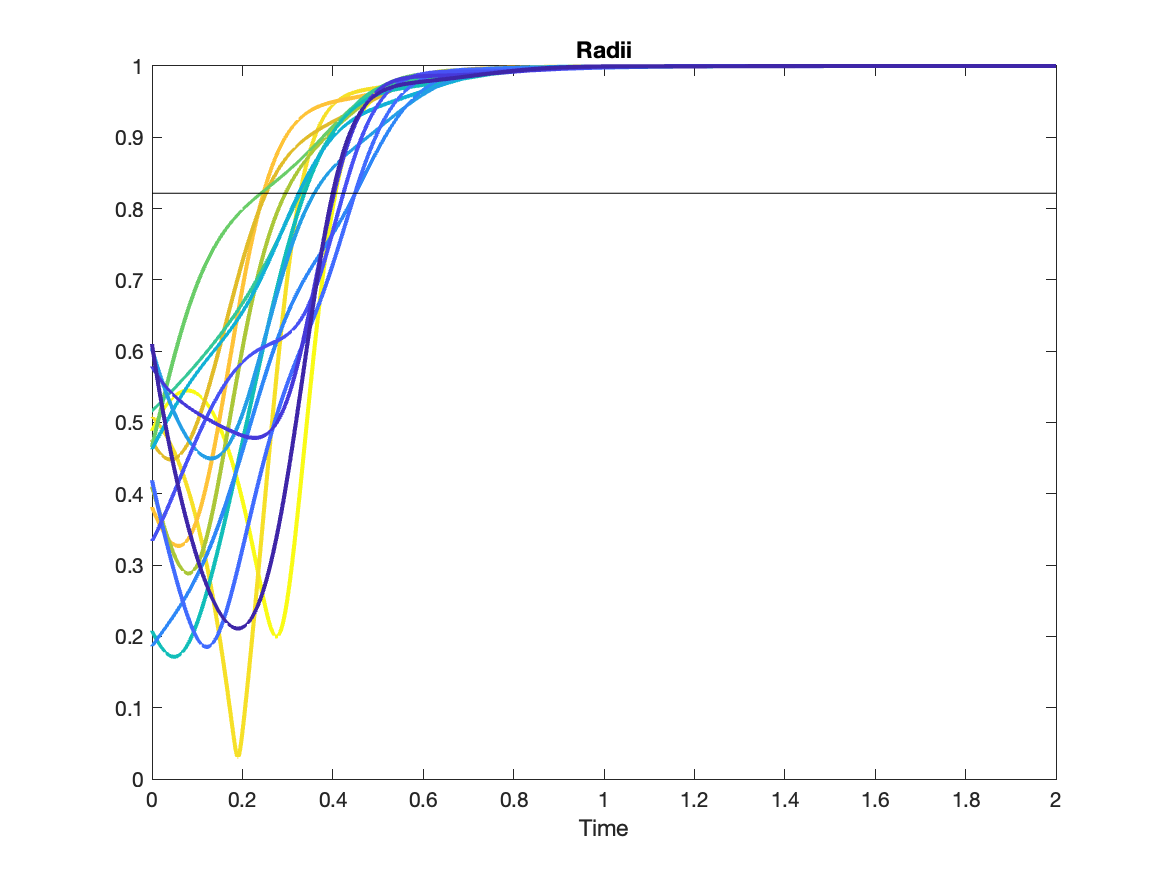}
\includegraphics[width=0.32\textwidth]{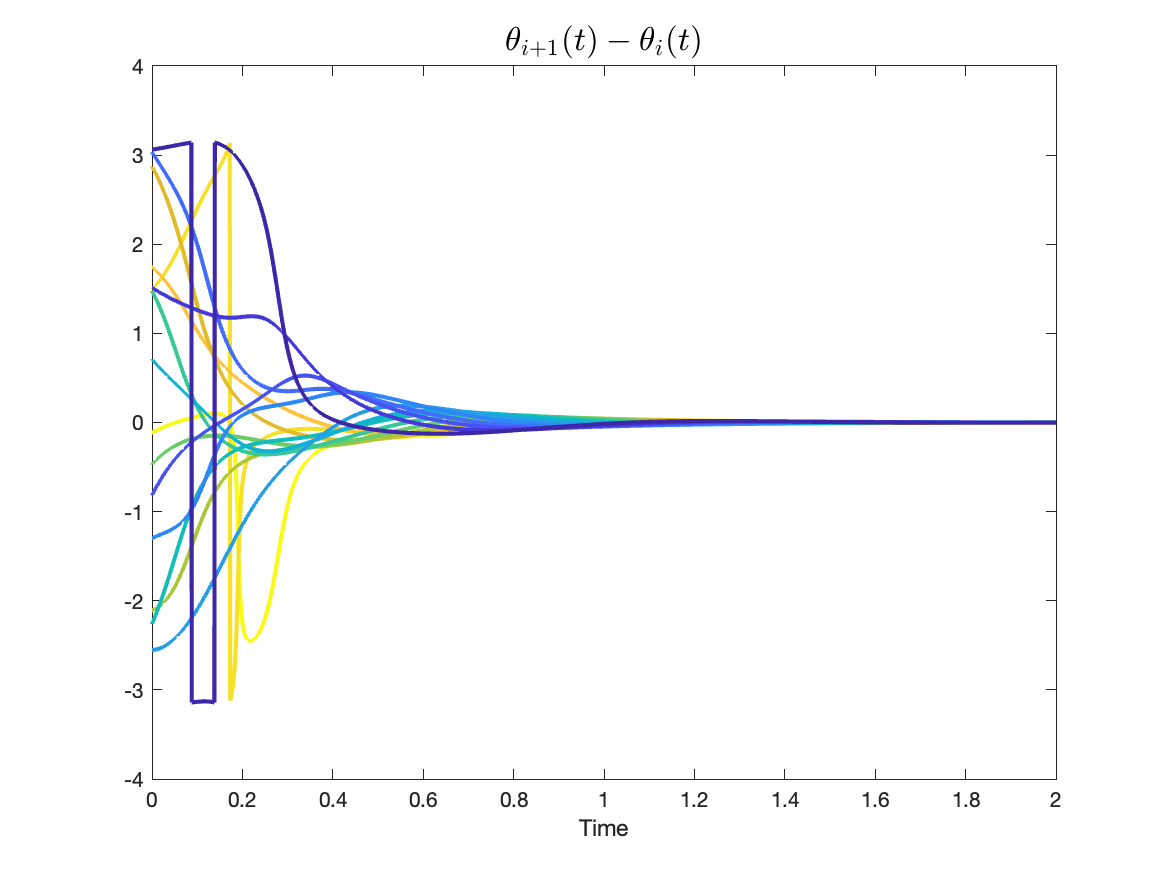}
\caption{Top row. Particles' trajectories at times $t=0$, $t=0.2$ and $t=0.8$. Each agent is represented in a different color, its last position is represented by the largest dot, and its previous positions by smaller dots. The black dot represents the trajectory of the system's center of mass.\\
Bottom row. Left: Evolution of the three order parameters $M$, $P$ and $S$. Center: Evolution of each agent's radius $r_j$. Right: Evolution of the $N$ angle differences $\theta_{j+1}-\theta_j$.   
}\label{Fig:Sync}
\end{figure}

\subsection{Inhomogeneous distribution of $\beta_j$}\label{Sec:sim2}

Figure \ref{Fig:Beta} shows the evolution of the system for a non-exchangeable system of particles in which the target radii $(\beta_j)_{j\in\{1,\cdots,N\}}$ are non identical. More specifically,
let $(B_j)_{j\in\{1,\cdots,N\}}$ be i.i.d. random variables uniformly distributed in the interval $[-0.5,0.5]$.
We construct heterogeneous target radii as follows: $\beta_j=\beta+\varepsilon B_j$.
The other parameters are chosen as in Section \ref{Sec:sim1}, i.e. $\beta=1$, $\tau=0.1$, $N=15$, and $A\in M_N(\mathbb{R})$ is the circulant matrix generated by the vector $a = (0,5,1,2,1,0,\cdots,0)$.

The radii are shown to converge to different values $R_j$, and the angle differences $\theta_{j+1}-\theta_j$ also converge to different values via a phase of damped oscillations.
The milling order parameter $M$ is shown to converge to 1.
\begin{figure}[h!]
\centering
\includegraphics[width=0.24\textwidth, trim = 2cm 0.5cm 2cm 0.5cm, clip=true]{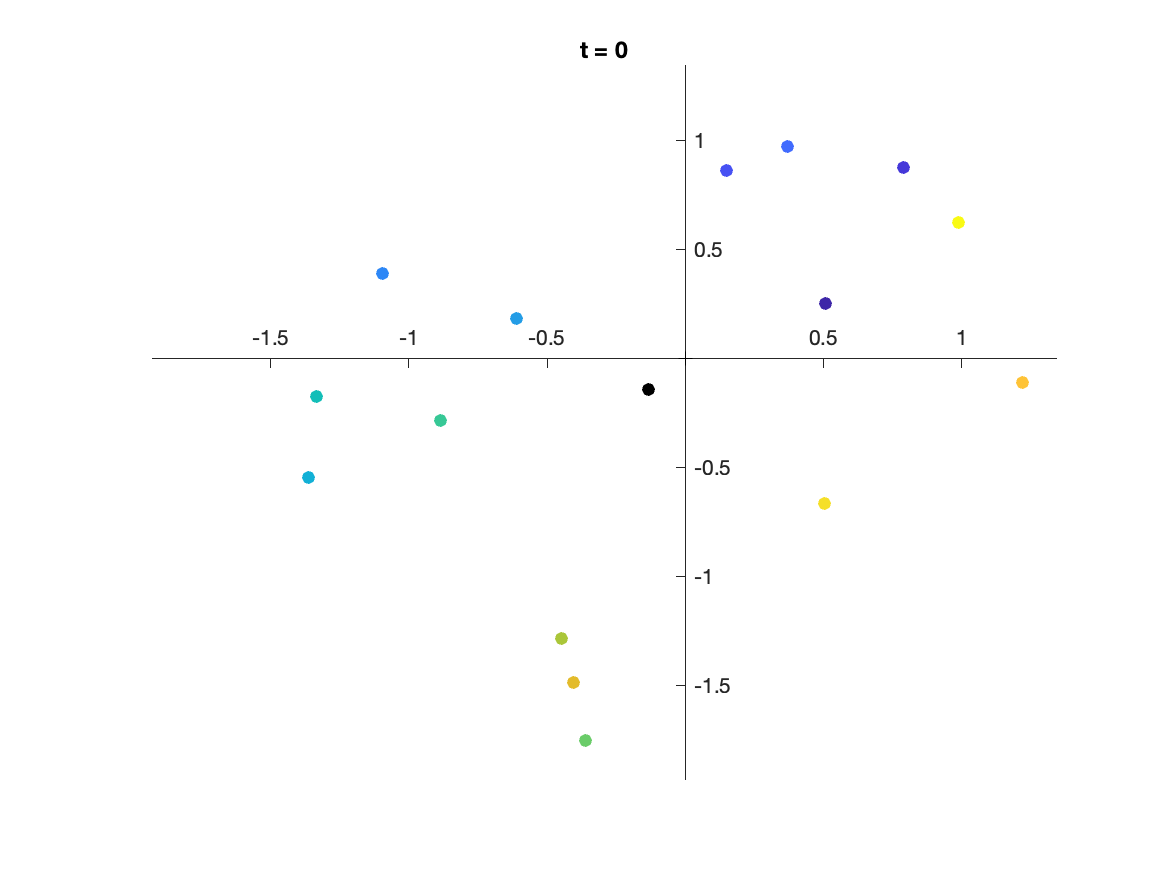}
\includegraphics[width=0.24\textwidth, trim = 2cm 0.5cm 2cm 0.5cm, clip=true]{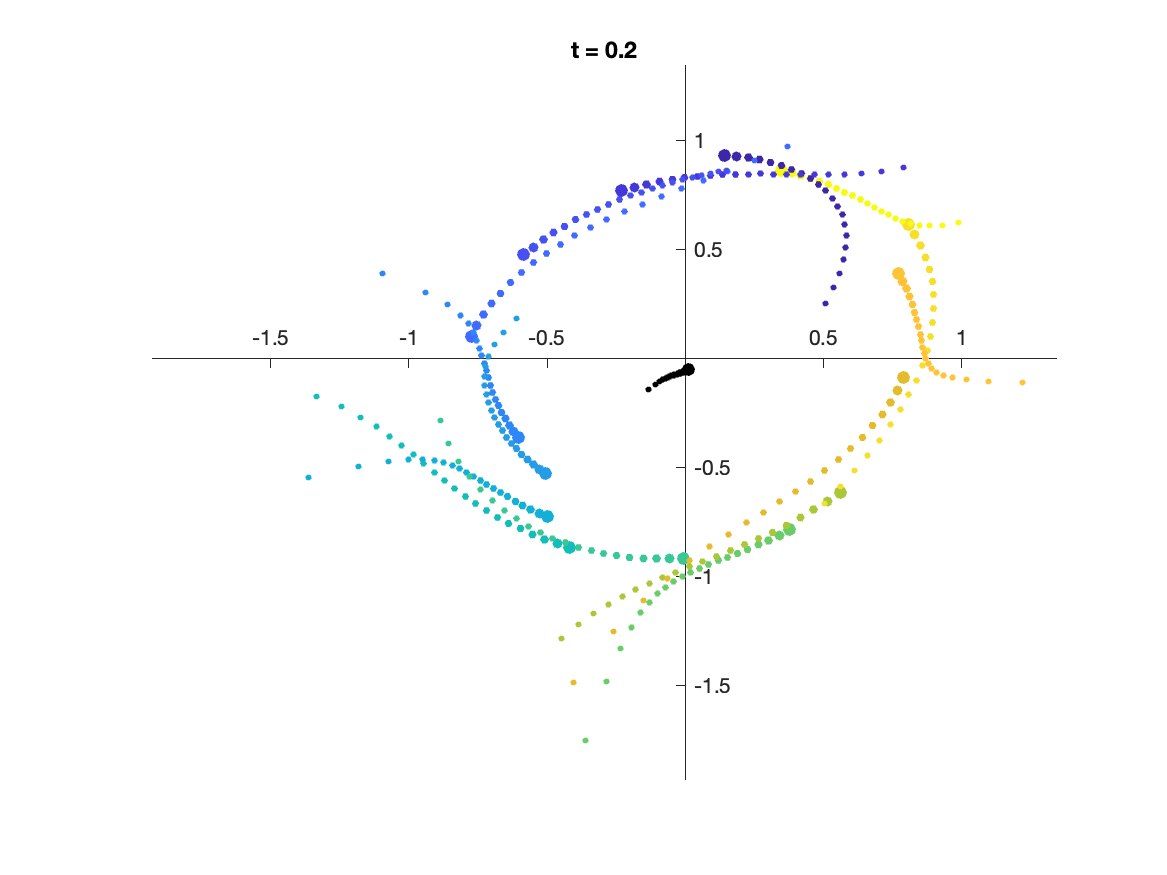}
\includegraphics[width=0.24\textwidth, trim = 2cm 0.5cm 2cm 0.5cm, clip=true]{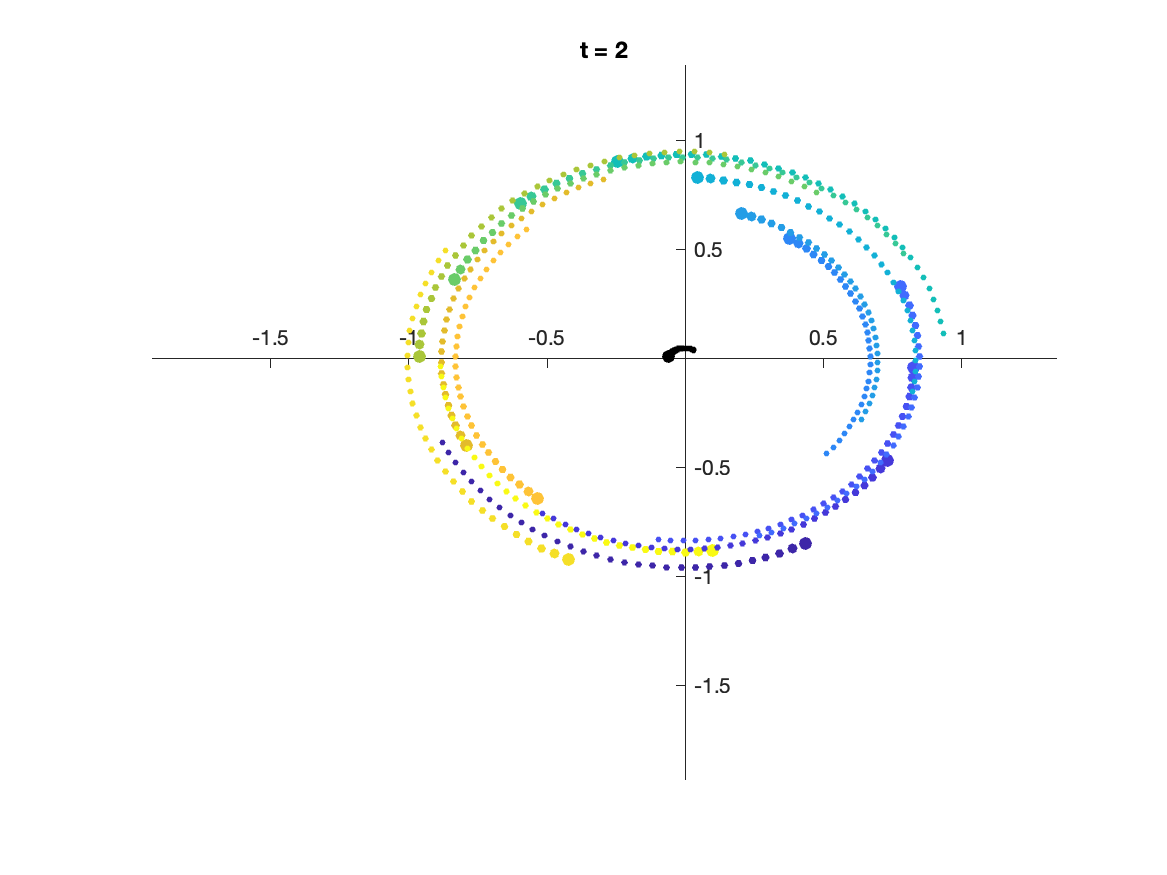}
\includegraphics[width=0.24\textwidth, trim = 2cm 0.5cm 2cm 0.5cm, clip=true]{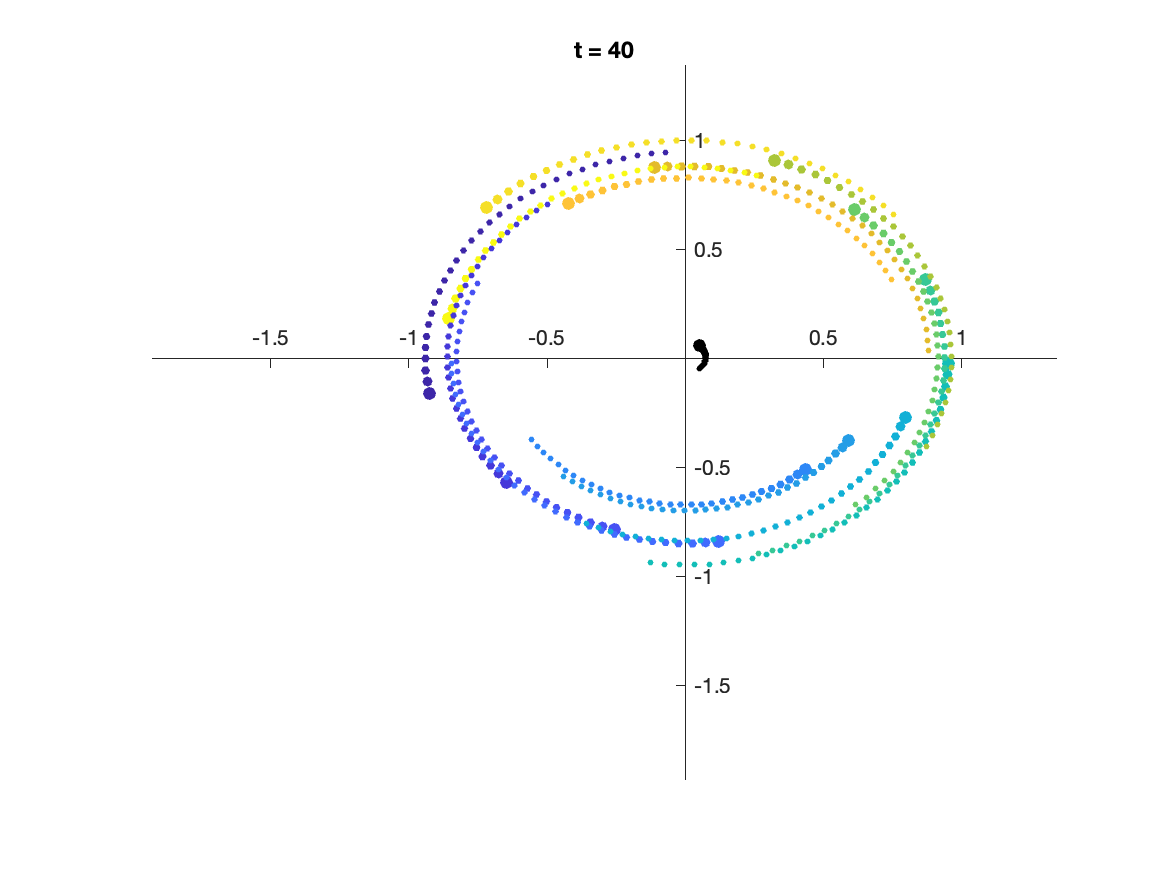}\\
\includegraphics[width=0.32\textwidth]{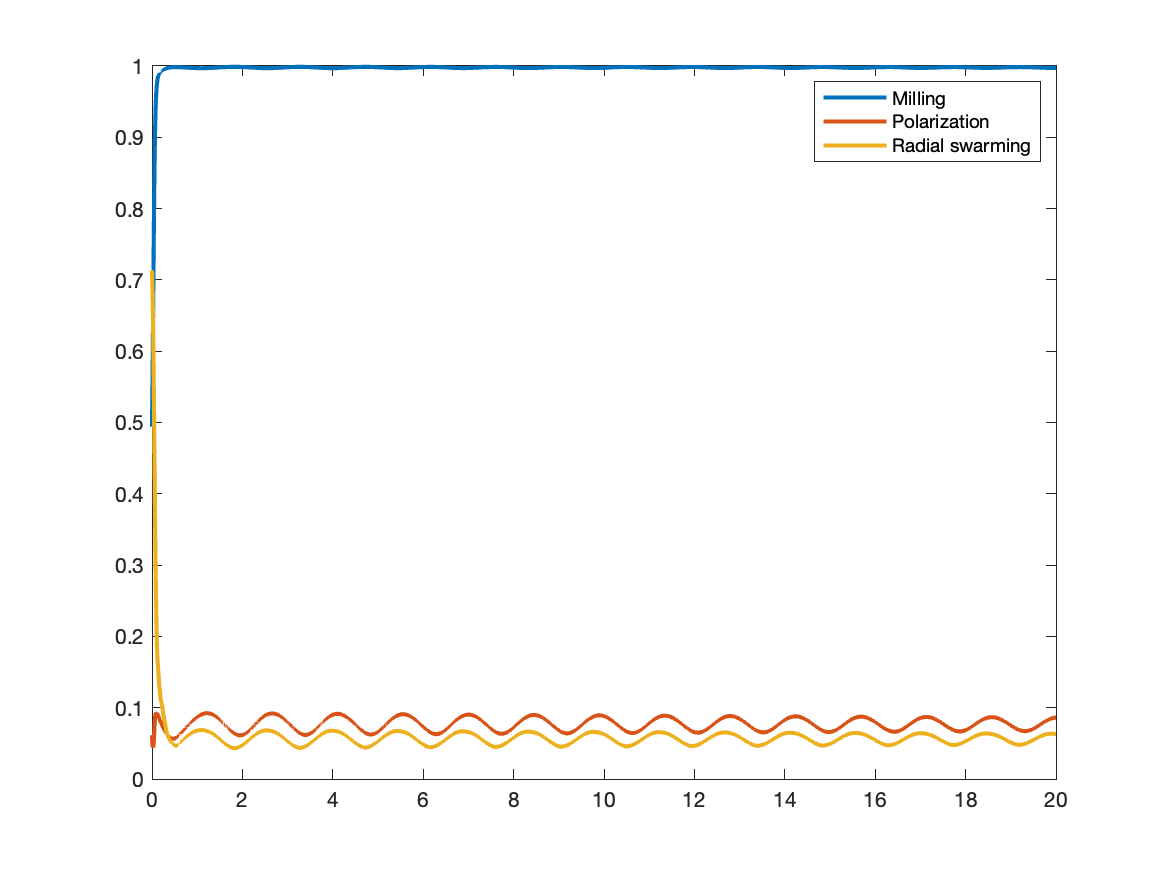}
\includegraphics[width=0.32\textwidth]{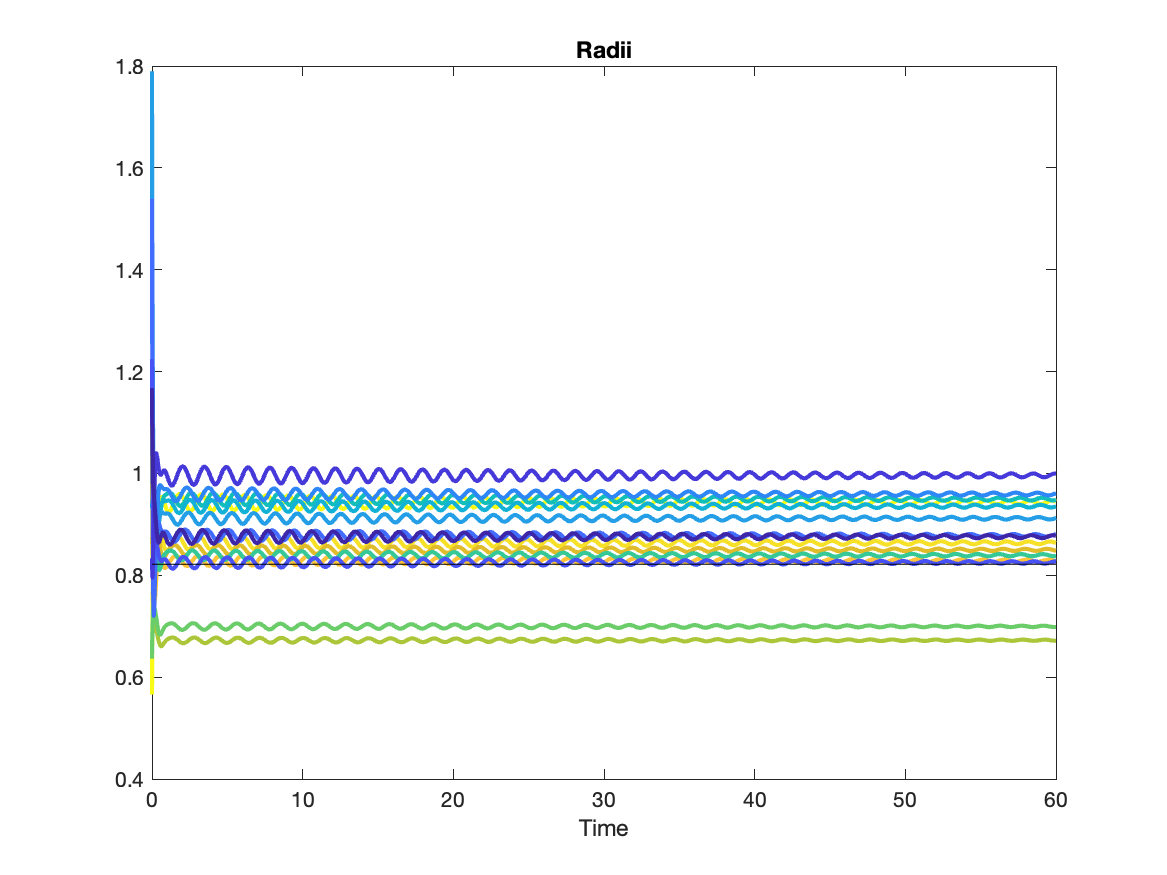}
\includegraphics[width=0.32\textwidth]{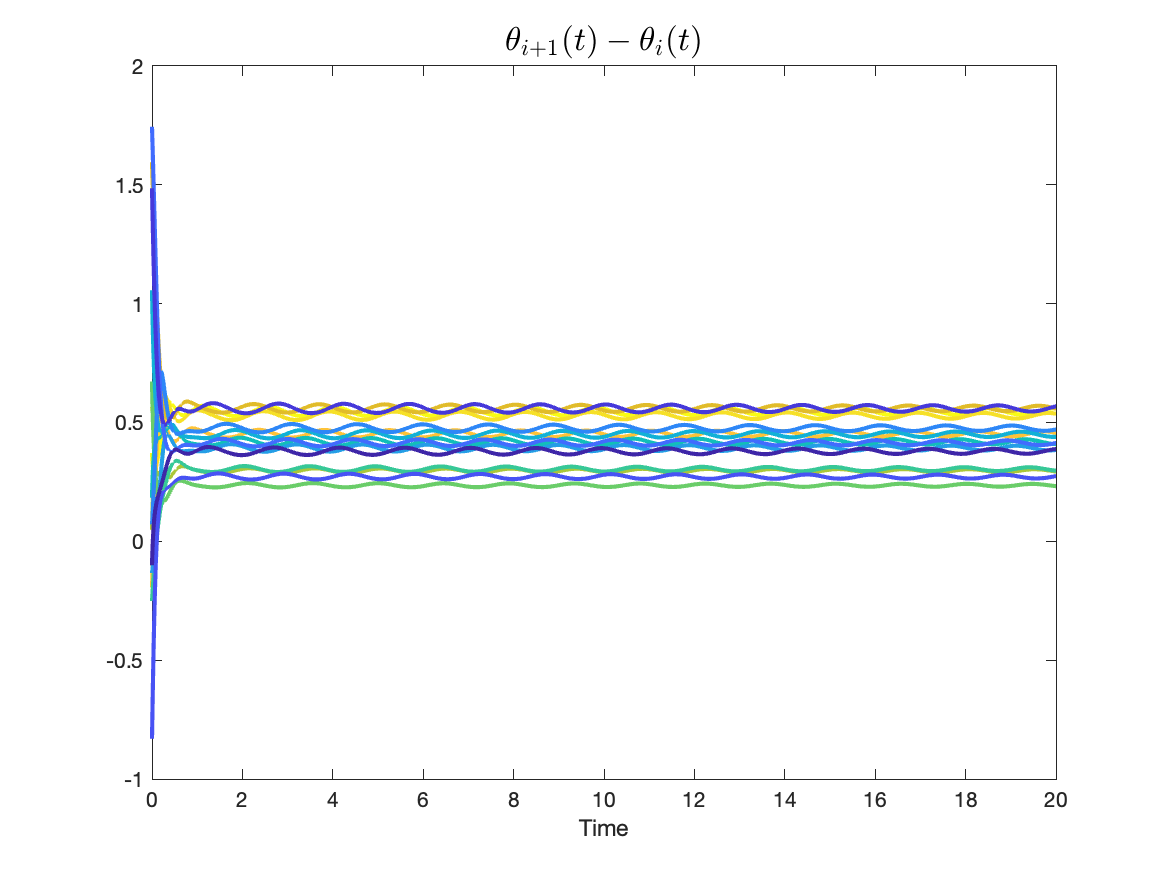}
\caption{Top row. Particles' trajectories at times $t=0$, $t=0.2$, $t=2$ and $t=40$. Each agent is represented in a different color, its last position is represented by the largest dot, and its previous positions by smaller dots. The black dots represents the trajectory of the system's center of mass.\\
Bottom row. Left: Evolution of the three order parameters $M$, $P$ and $S$. Center: Evolution of each agent's radius $r_i$. Right: Evolution of the $N$ angle differences $\theta_{i+1}-\theta_i$.   
}\label{Fig:Beta}
\end{figure}

\subsection{Follower dynamics}\label{Sec:sim3}

In this section, we illustrate the results proven in Section \ref{s:6}.
Let $N=15$, and $(A_{ij})_{i,j\in\{1,\cdots,14\}}\in M_{14}(\mathbb{R})$ be a circulant matrix determined by the vector $(0,4,2,1,0,\cdots,0)$. Let the $N$-th agent represent a follower, connected to the agent $j=5$ by the edge $A_{N5} = a>0$.
The corresponding network is illustrated in Figure  \ref{Fig:NetworkFollower}.

\begin{figure}[h!]
\centering
\includegraphics[width=0.5\textwidth]{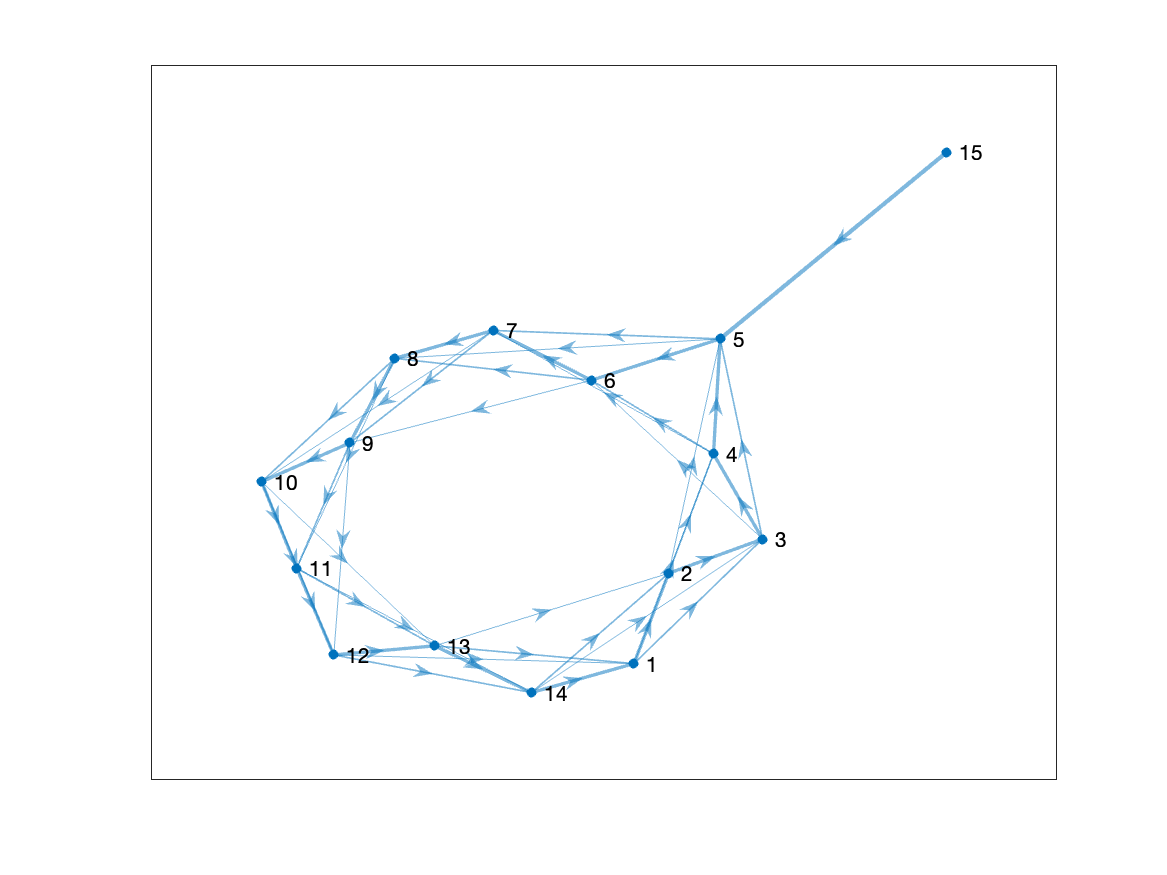}
\caption{Interaction network corresponding to the circulant matrix $A$ (Section \ref{Sec:sim3}). Line thickness is proportional to the edge weight.}\label{Fig:NetworkFollower}
\end{figure}

As in the previous sections, the other parameters are taken to be $\beta=1$ and $\tau=0.1$.
With this choice, one can compute the critical coupling strength $a^*=3.7582$ required for the follower to be able to follow the group's rotational state.
We present two simulations. In the first one, the coupling strength is taken to be $a=3<a^*$, below the critical coupling strength, and the follower is unable to follow the rotational state. In the second simulation, the coupling strength is taken to be $a=5>a^*$ is above the threshold, and the follower successfully joins the group rotation. Figure \ref{Fig:Follower} illustrates these two situations by presenting the trajectories of the leader $j=5$ and the follower $j=N$ (left column) as well as the evolution of all agents' radii and of the order parameters.

Interestingly, in the case of a successful follower ($a=5$), the follower's radius converges to a value slightly below the radius of the group.
In the case of the failed follower ($a=3$), periodic observations are observed, as the follower periodically manages to almost follow the group, before falling behind too much and swithching directions.

\begin{figure}[h!]
\centering
\includegraphics[width=0.32\textwidth]{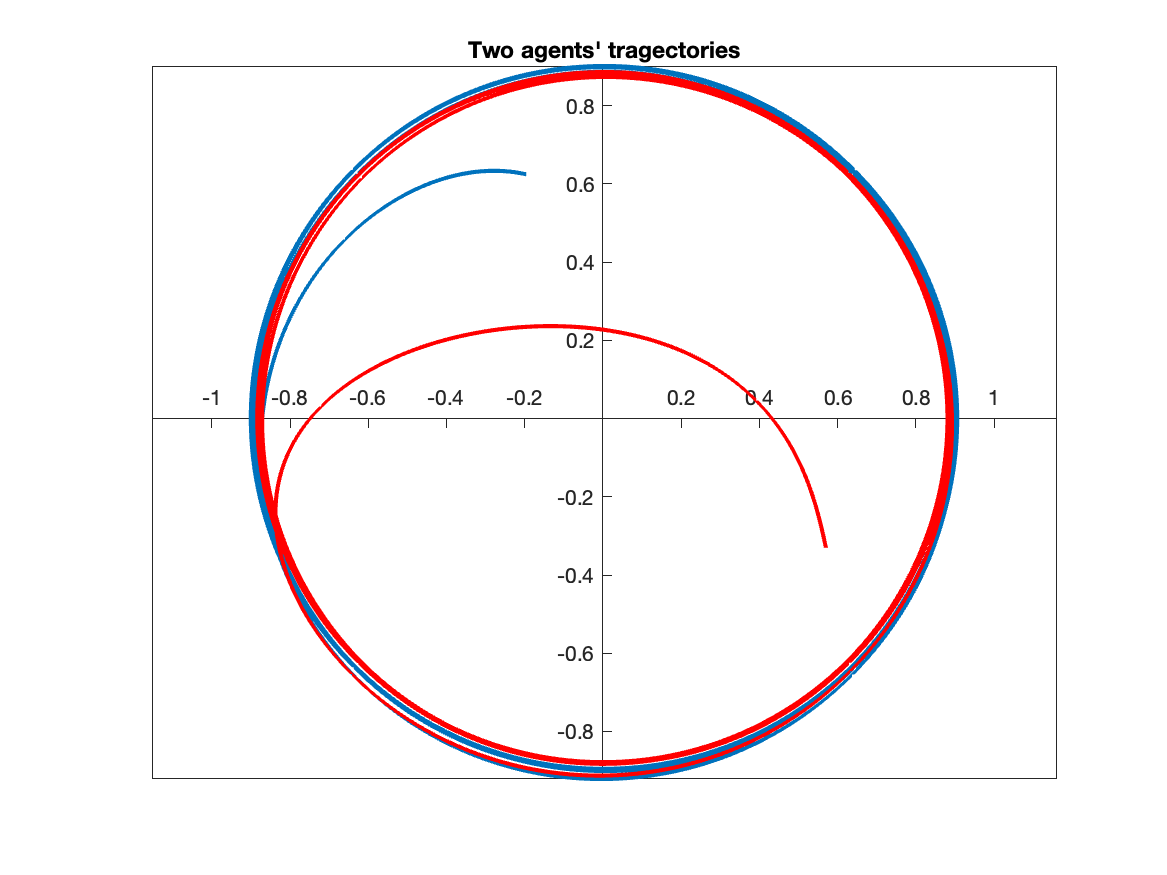}
\includegraphics[width=0.32\textwidth]{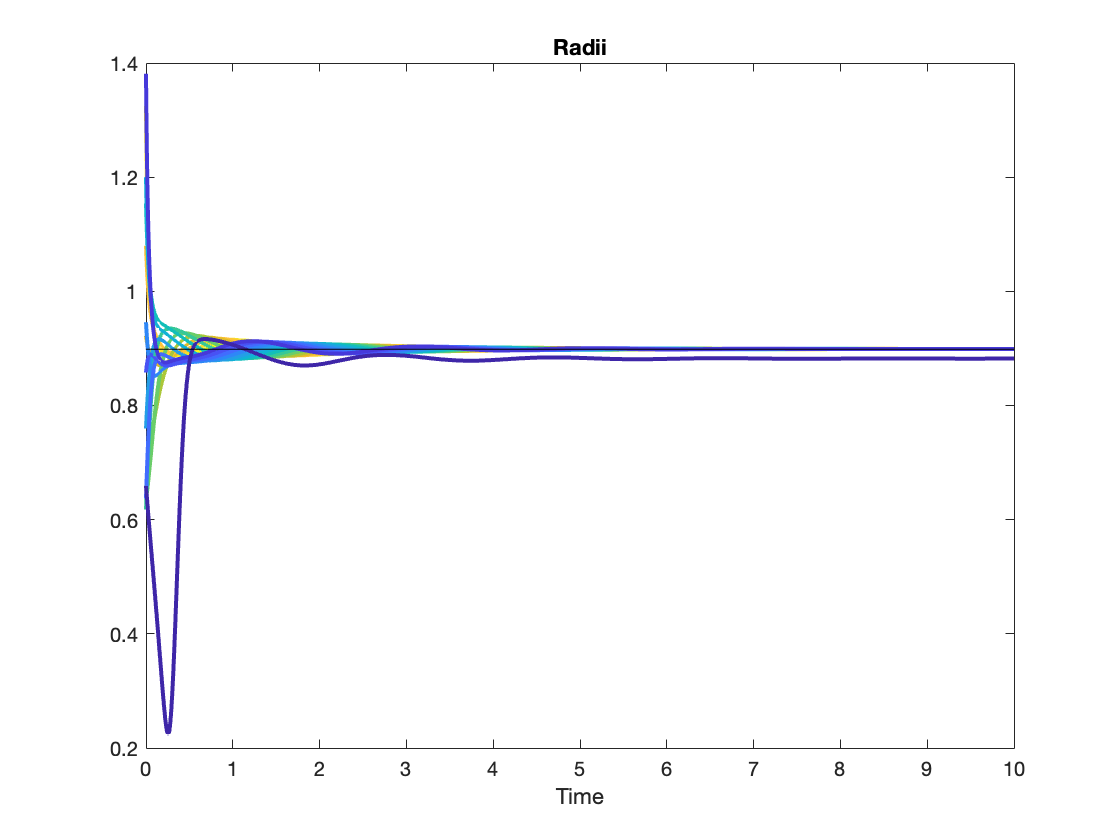}
\includegraphics[width=0.32\textwidth]{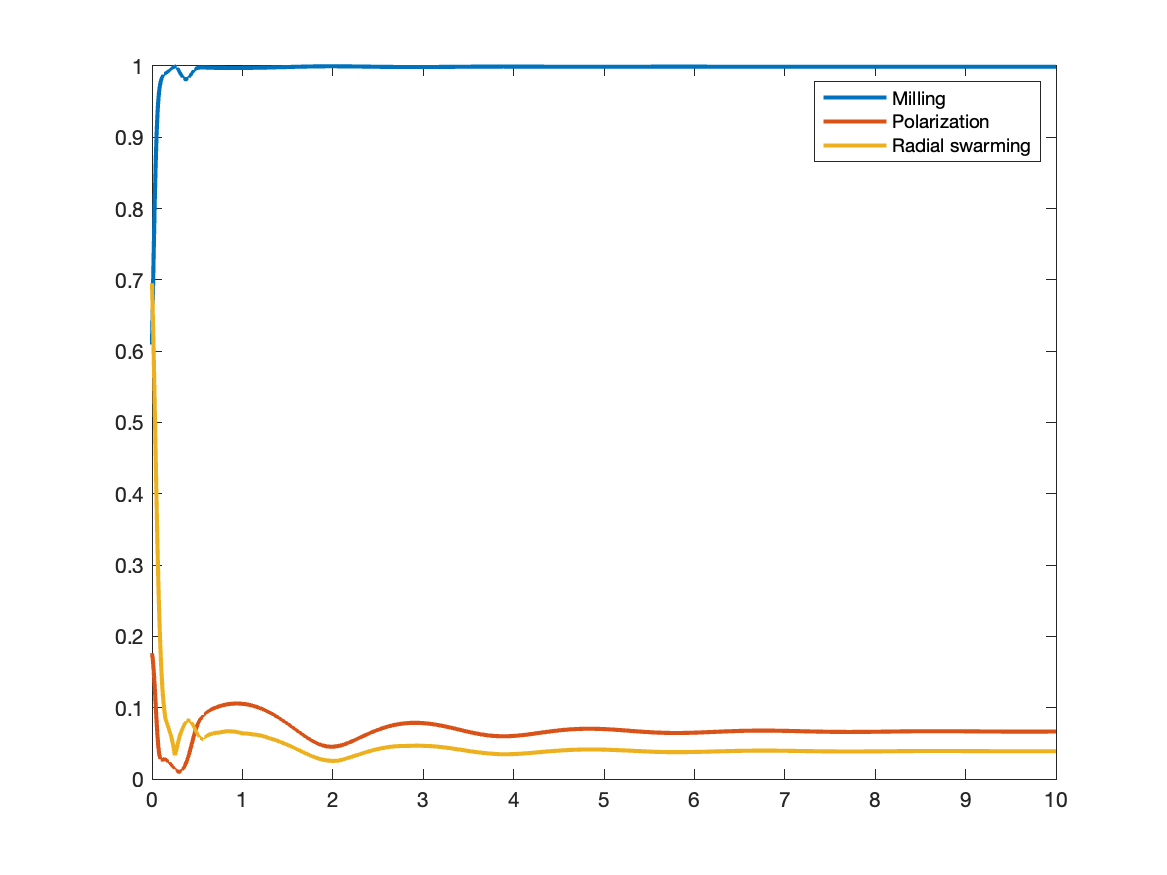}\\
\includegraphics[width=0.32\textwidth]{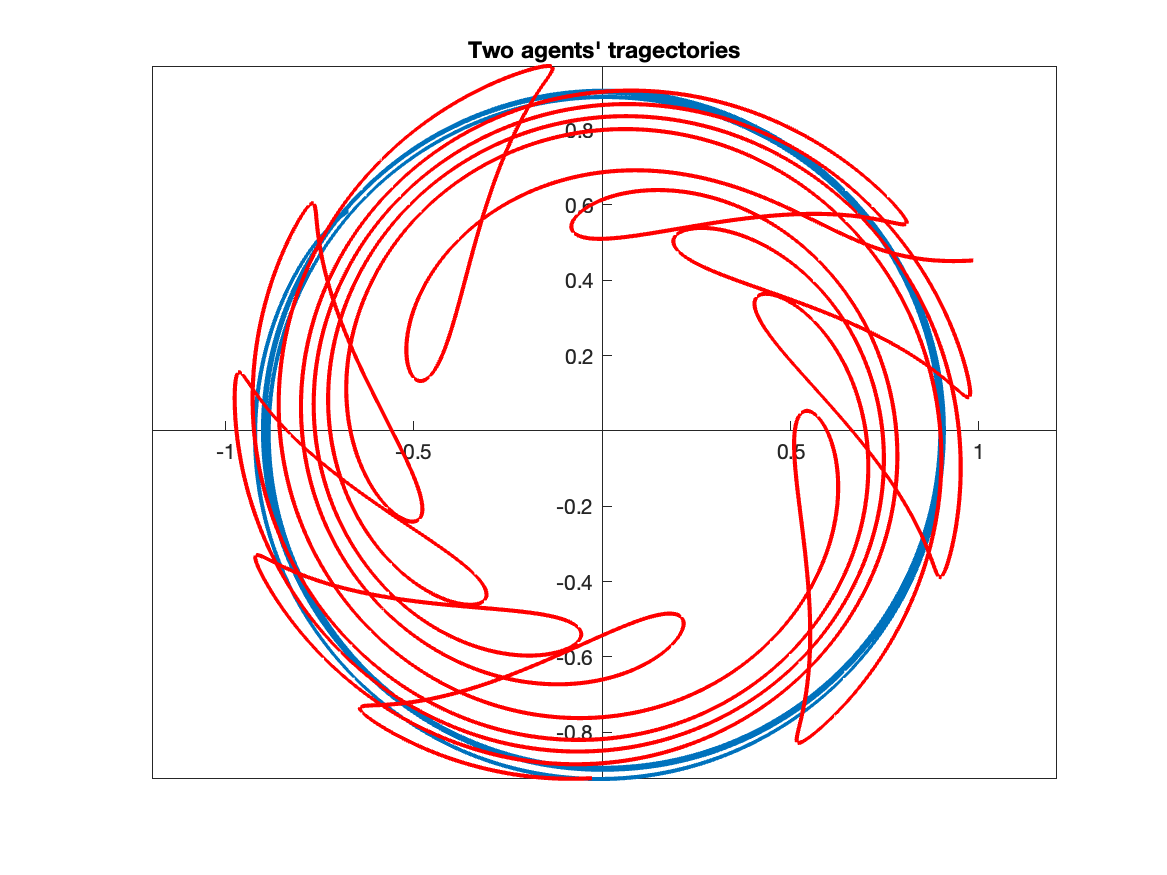}
\includegraphics[width=0.32\textwidth]{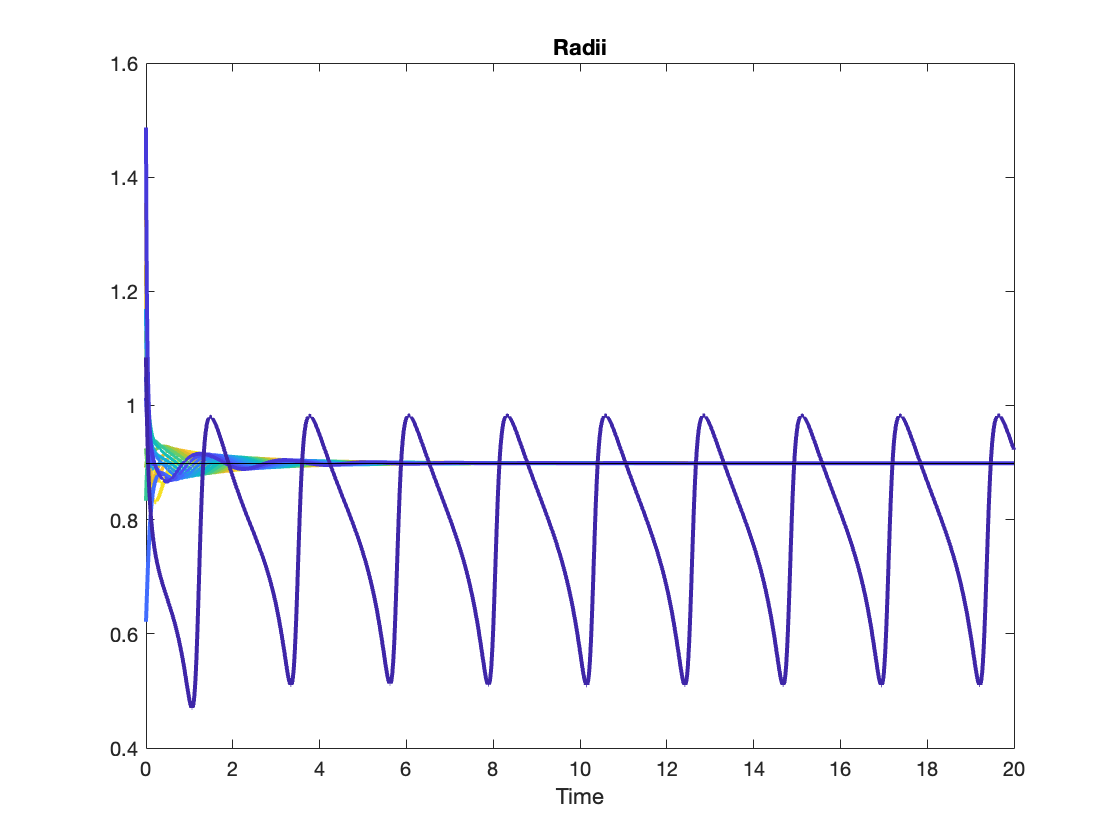}
\includegraphics[width=0.32\textwidth]{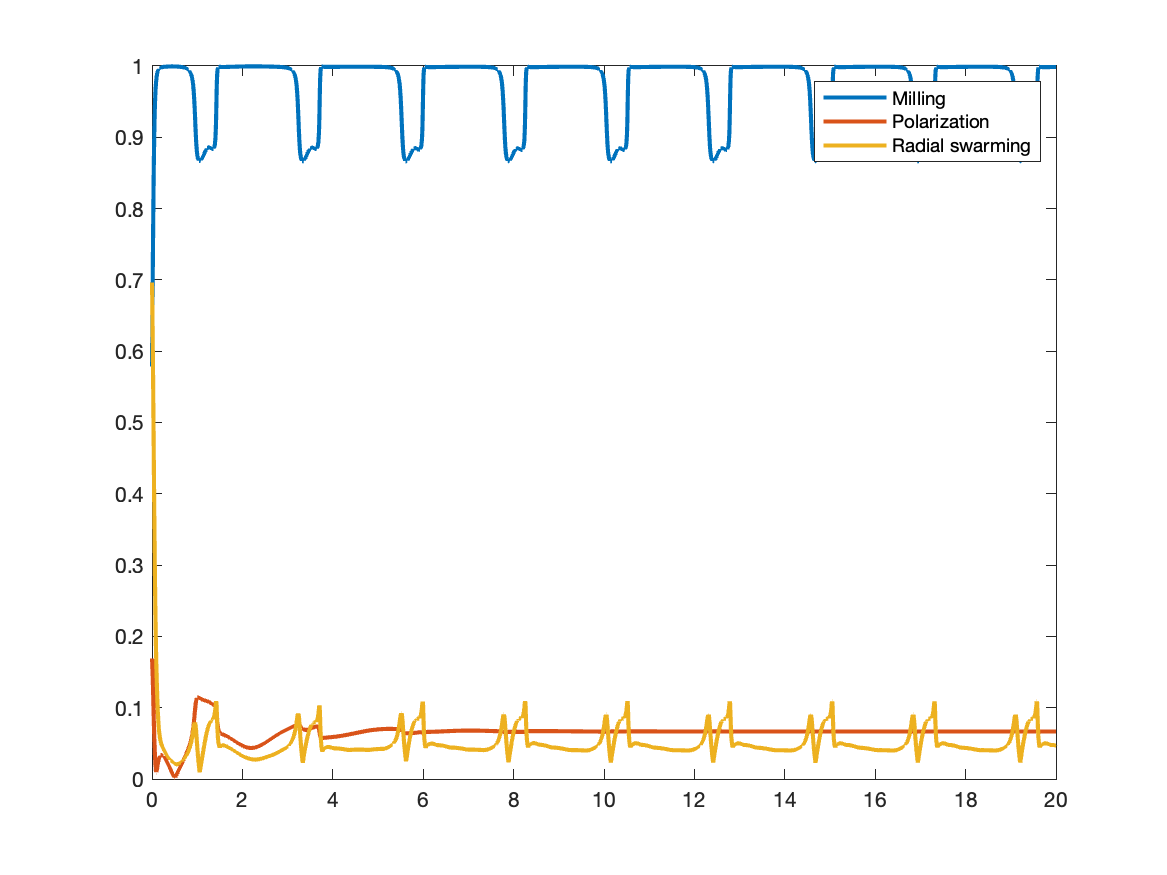}\\
\caption{Left: Trajectories of the leader $j=5$ (blue) and the follower $j=N$ (red) for $a=5>a^*$ (top row) and $a=3<a^*$ (bottom row). Center: Evolution of the radii for $a=5$ (top row) and $a=3$ (bottom row). Right: Evolution of the order parameters for $a=5$ (top row) and $a=3$ (bottom row).
}\label{Fig:Follower}
\end{figure}

\subsection{Metastability}
In this section, we illustrate the different ways that the system may converge to the synchronized state, by going through various transitory phases. 
As in Section \ref{Sec:sim1}, the parameters are chosen to be $\beta=1$, $N=15$, $a = (0,5,1,2,1,0,\cdots,0)$, which gives: $\tau^* = 0.3077$, $R = 0.8216$. We provide a heurestic justification for the qualitatively different metastable dynamics, noting that a more rigorous treatment of the metastable dynamics will be left to a future work.

In the first case (cf. Figure \ref{Fig:Meta1}), we present a situation in which $\tau=0.11<\tau^*$, but for which the stability of the rotational twisted state has already broken down. While in the second case (Figure \ref{Fig:Meta2}), we have $\tau=1>\tau^*$ so that the rotational twisted state no longer exists.

In Section \ref{s:4} it was shown that for $0<\tau<\tau^{**}\leq \tau^*$, the rotational twisted state is guaranteed to be locally stable, where $\tau^{**}$ is the specific value for which stability breaks down from one of the eigenvalues (calculated in Section \ref{s:4}) crossing the imaginary axis. This implies that the invariant manifold for radially symmetric initial data discussed in Subsection \ref{ss:Invariant} is an attracting normally hyperbolic invariant manifold for all $\tau<\tau^{**}$, where the manifold loses hyperbolicity exactly at the value $\tau=\tau^{**}$,  see \cite{Eldering2013,Kuehn2015}.

Thus for $\tau \in (\tau^{**},\tau^*)$, for initial data close to the invariant manifold, we initially see the prolonged rotational dynamics, near the rotational twisted state, but the instability creates growth in one direction until the the twisted state breaks down and the system converges to the synchronized state, as seen in Figure \ref{Fig:Meta1}. As the circulant structure continues to act on the oscillators as they approach the fixed point, the final state represents swarming behavior, $S=1$.

Finally for $\tau>\tau^*$, we have two competing mechanisms. Although the twisted state does not exist, the invarainat manifold of subsection \ref{ss:Invariant} persists, however it won't be an attracting manifold as the unstable directions for $\tau<\tau^*$ still remain unstable, and the zero fixed point for which all trajectories on this manifold must converge to, was shown to be unstable in Section \ref{s:5}. Thus, the fact that the invariant manifold produces exponential decay of the radius according to the equation
$$\ddt R(t)= \frac{1}{\tau}\left(1-\frac{\tau}{\tau^*}-\frac{R^2}{\beta^2}\right)R,$$
implies that starting close enough to the invariant manifold, the oscillators undergo rapid contraction towards zero at rate $\frac{1}{\tau^*}-\frac{1}{\tau}$. However the instability of the zero fixed point means that once the dynamics are close enough, the system escapes the neighborhood of the unstable fixed point and again converges to the synchronized state. However, the final state represents polarization, $P=1$, as the rapid contraction towards zero essentially eliminates the rotational dynamics and the nonlinear forcing dominates in the final timescale. This regime can be seen in Figure \ref{Fig:Meta2}. Initial data which start sufficiently far from the invariant manifold would not experience such rapid constriction and instead will converge to the fixed point with swarming dynamics as in Figure \ref{Fig:Meta1}.

\begin{figure}[h!]
\centering
\includegraphics[width=0.24\textwidth, trim = 2cm 0.5cm 2cm 0.5cm, clip=true]{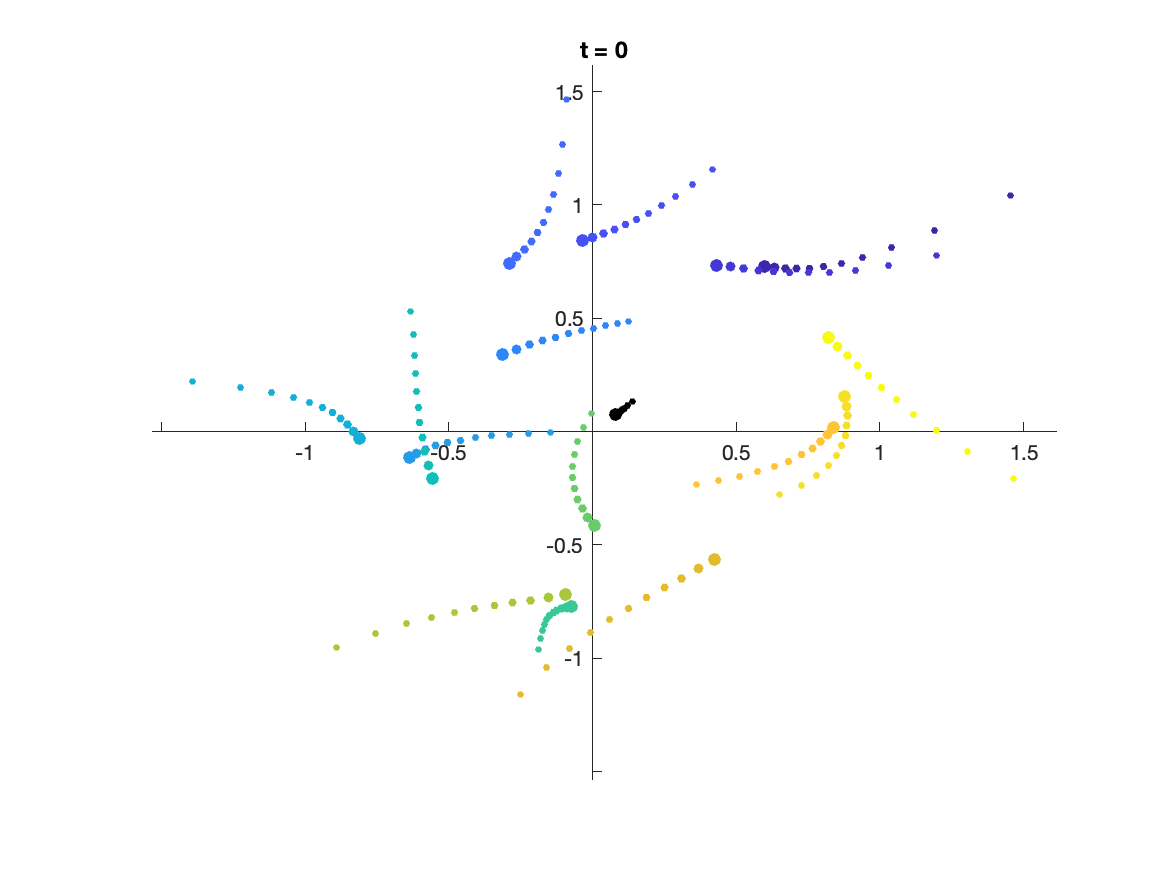}
\includegraphics[width=0.24\textwidth, trim = 2cm 0.5cm 2cm 0.5cm, clip=true]{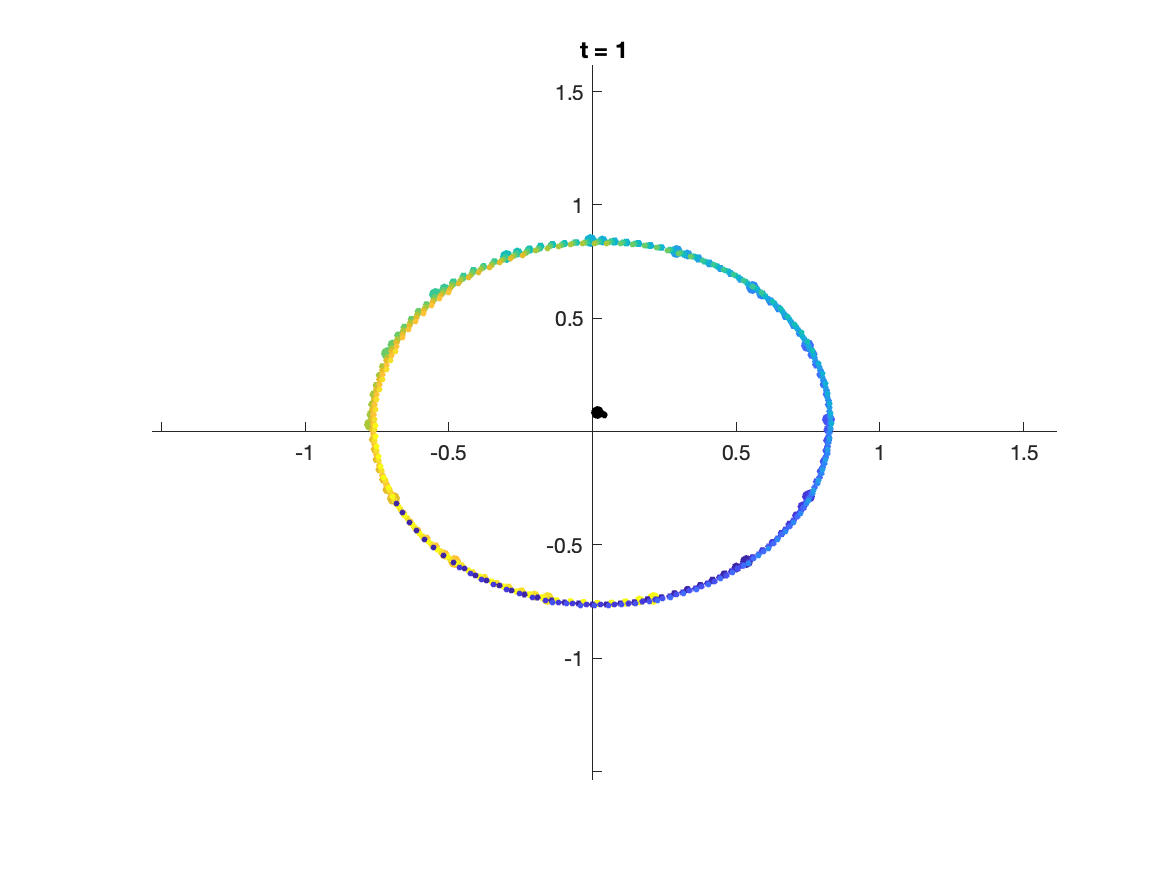}
\includegraphics[width=0.24\textwidth, trim = 2cm 0.5cm 2cm 0.5cm, clip=true]{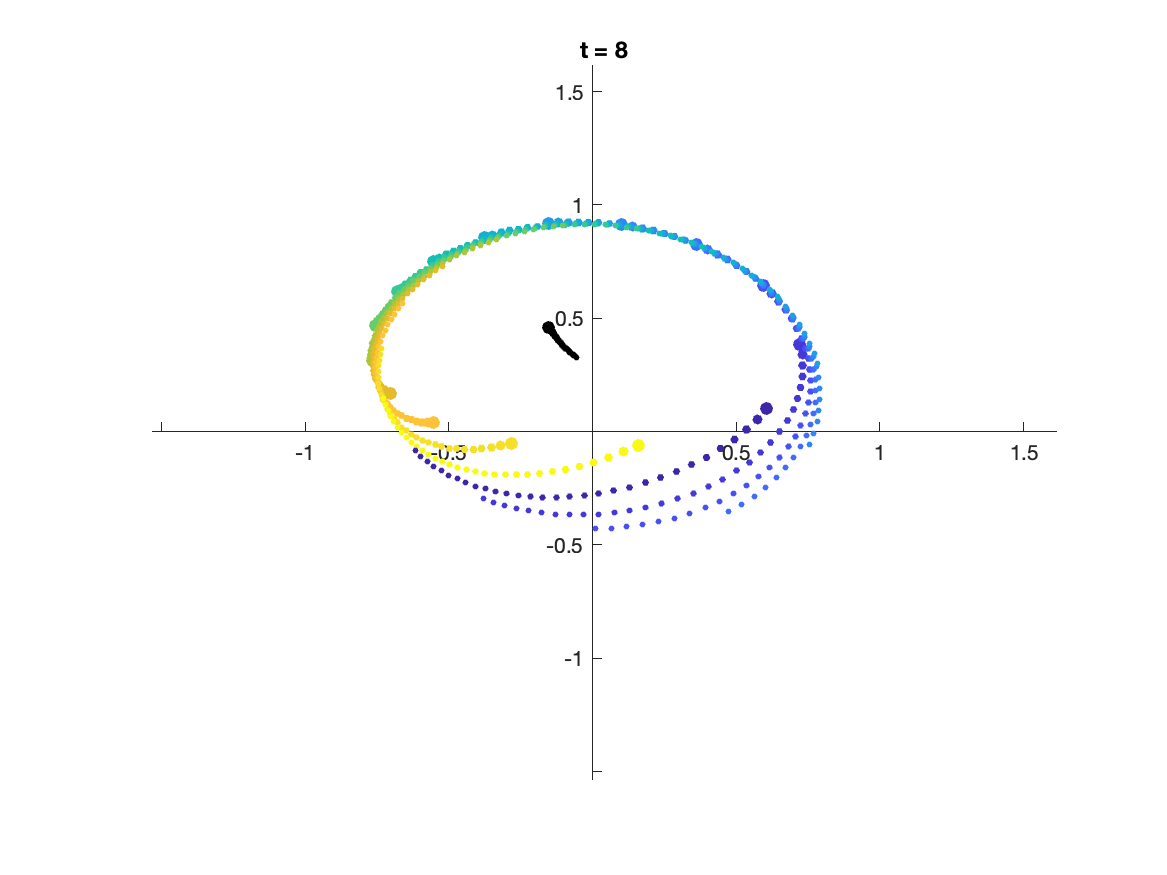}
\includegraphics[width=0.24\textwidth, trim = 2cm 0.5cm 2cm 0.5cm, clip=true]{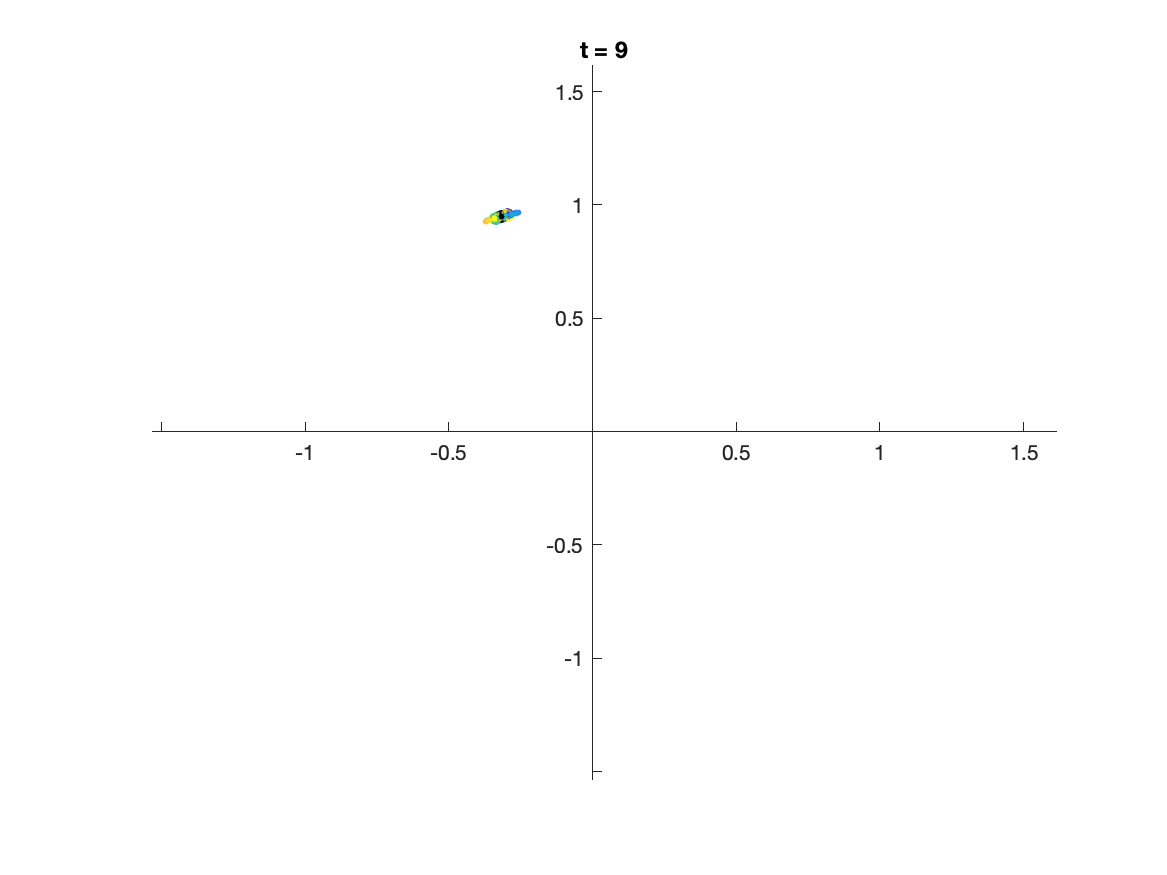}\\
\includegraphics[width=0.32\textwidth]{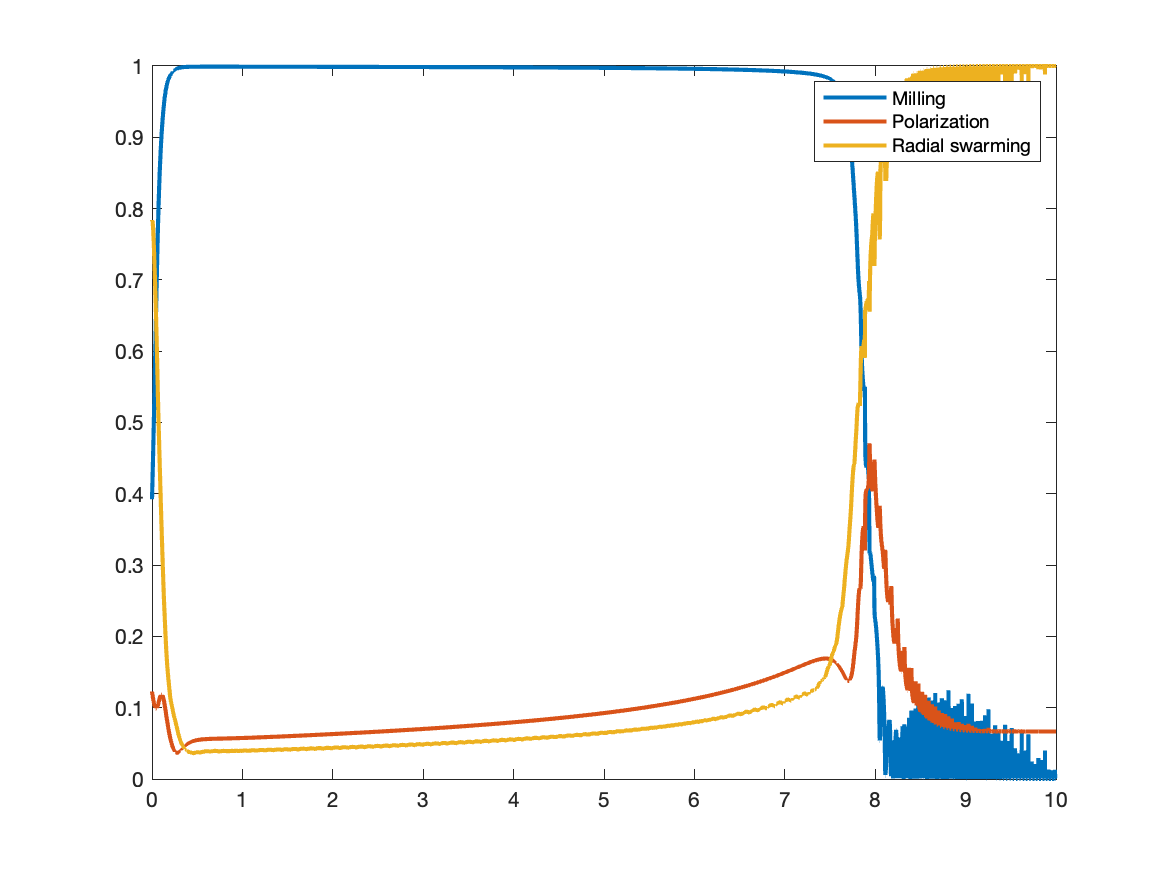}
\includegraphics[width=0.32\textwidth]{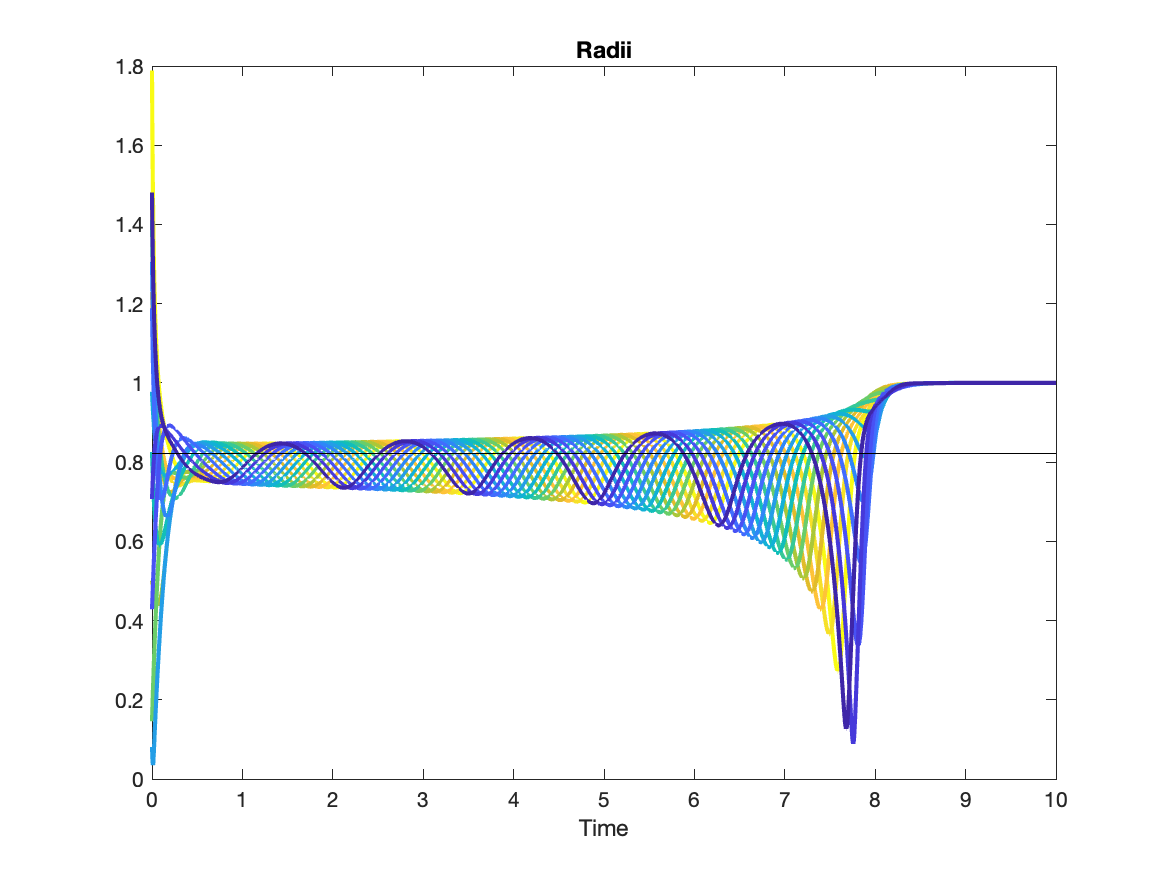}
\includegraphics[width=0.32\textwidth]{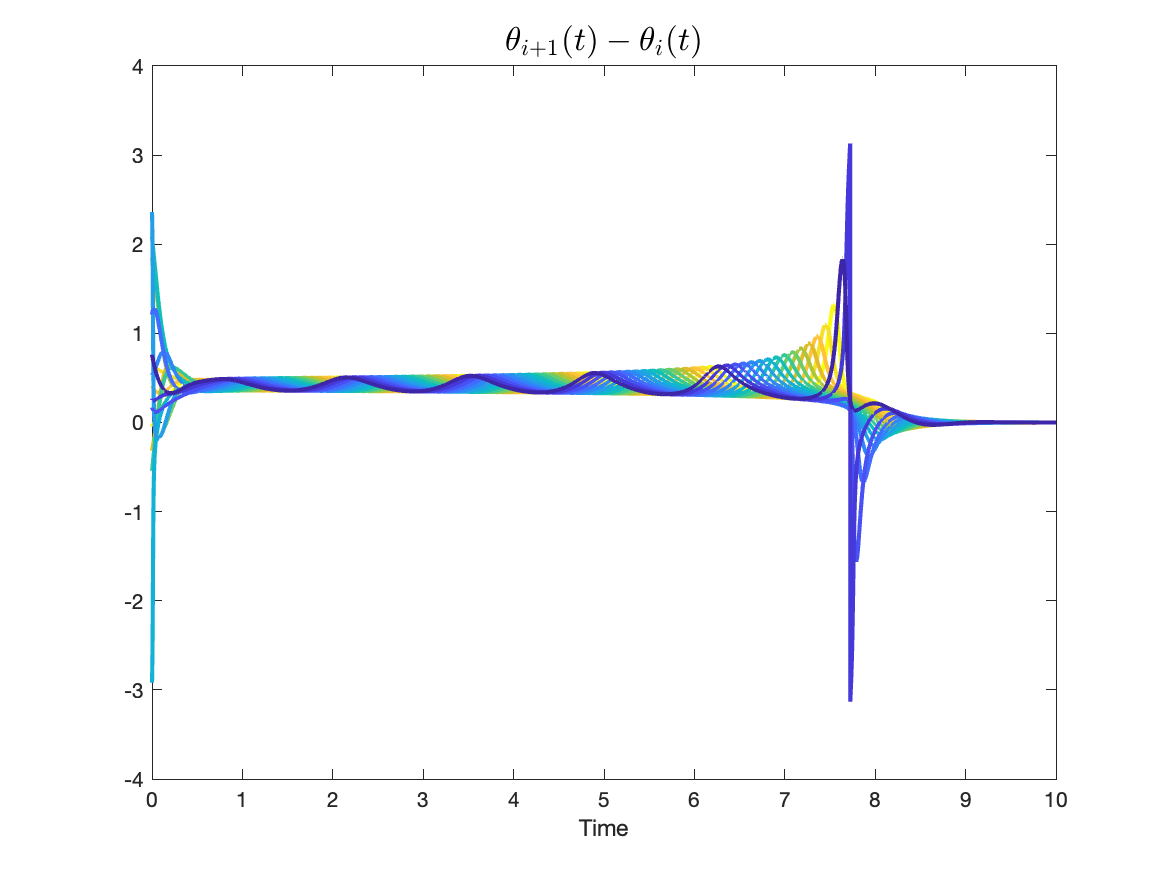}
\caption{Top row. Particles' trajectories at times $t=0.1$, $t=1$, $t=8$ and $t=9$. Each agent is represented in a different color, its last position is represented by the largest dot, and its previous positions by smaller dots. The black dot represents the trajectory of the system's center of mass.\\
Bottom row. Left: Evolution of the three order parameters $M$, $P$ and $S$. Center: Evolution of each agent's radius $r_j$. Right: Evolution of the $N$ angle differences $\theta_{j+1}-\theta_j$.   
}\label{Fig:Meta1}
\end{figure}

\begin{figure}[h!]
\centering
\includegraphics[width=0.24\textwidth, trim = 2cm 0.5cm 2cm 0.5cm, clip=true]{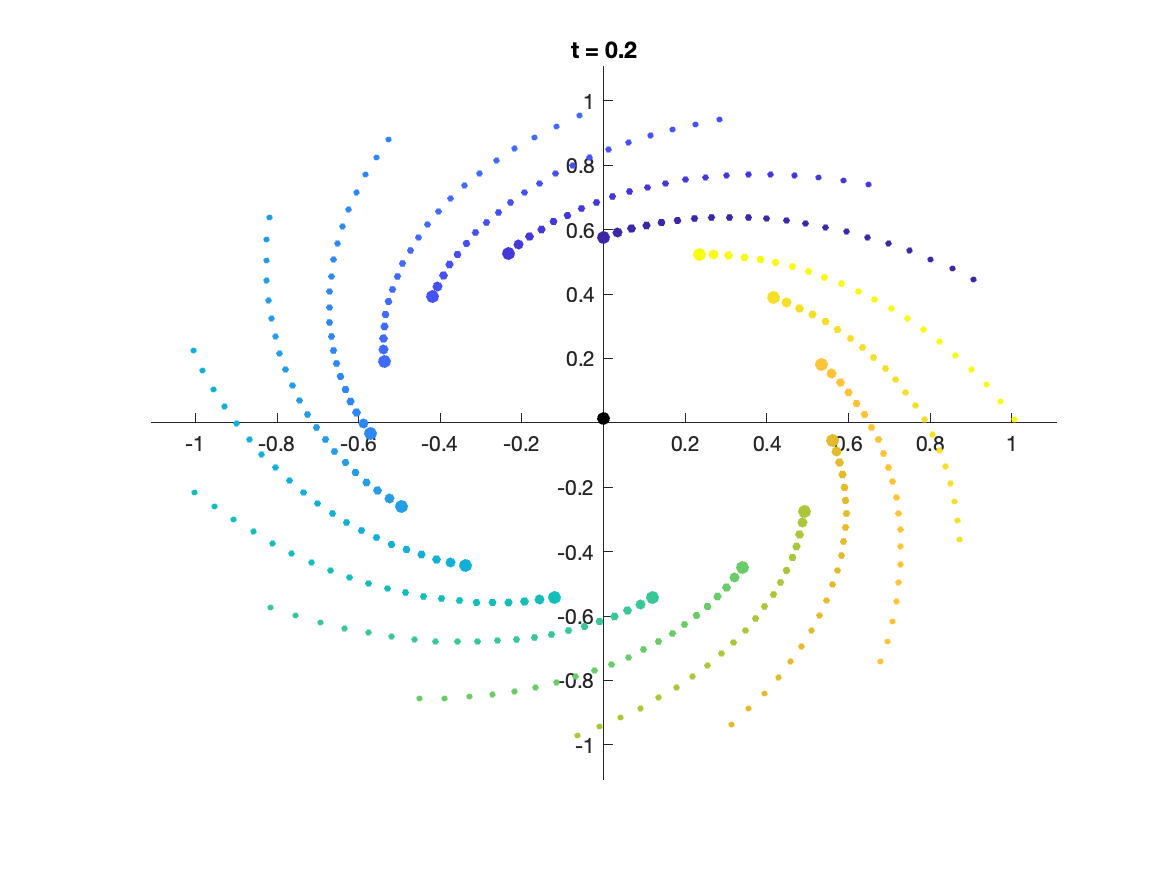}
\includegraphics[width=0.24\textwidth, trim = 2cm 0.5cm 2cm 0.5cm, clip=true]{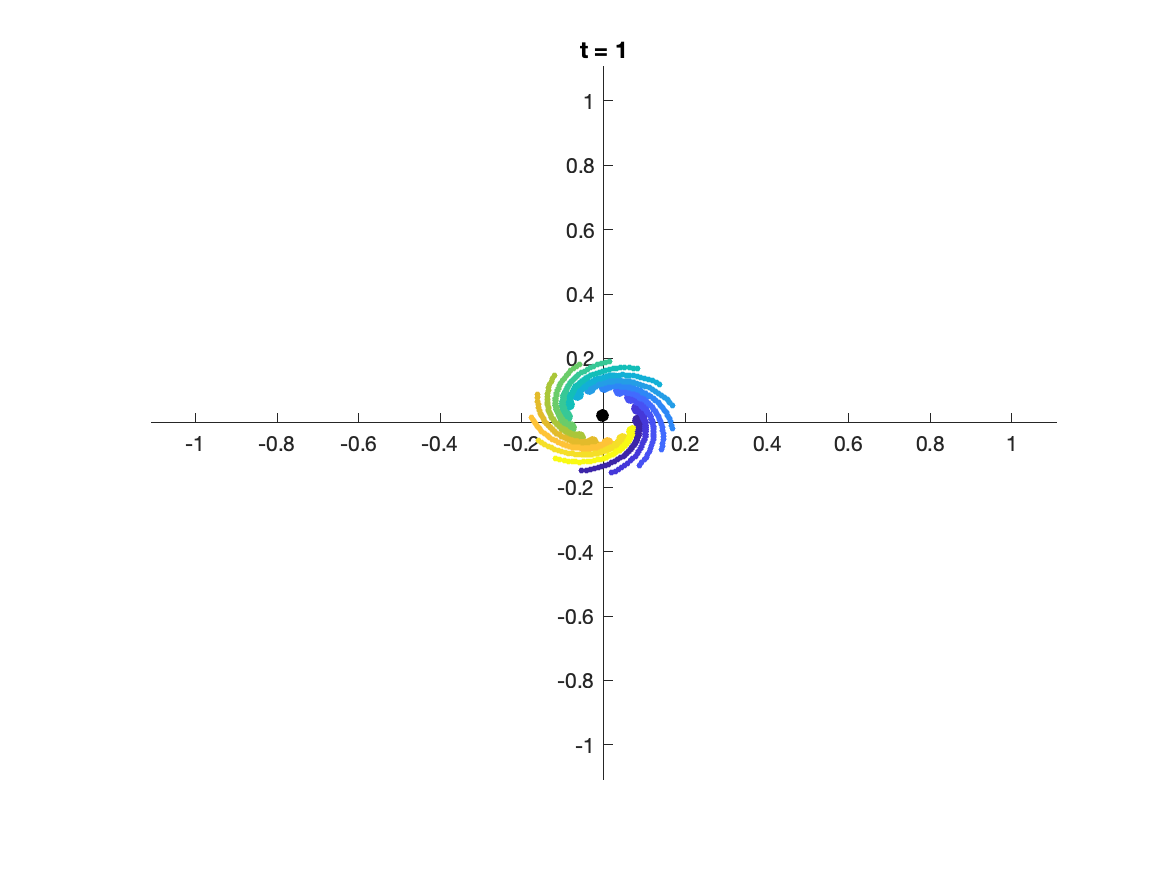}
\includegraphics[width=0.24\textwidth, trim = 2cm 0.5cm 2cm 0.5cm, clip=true]{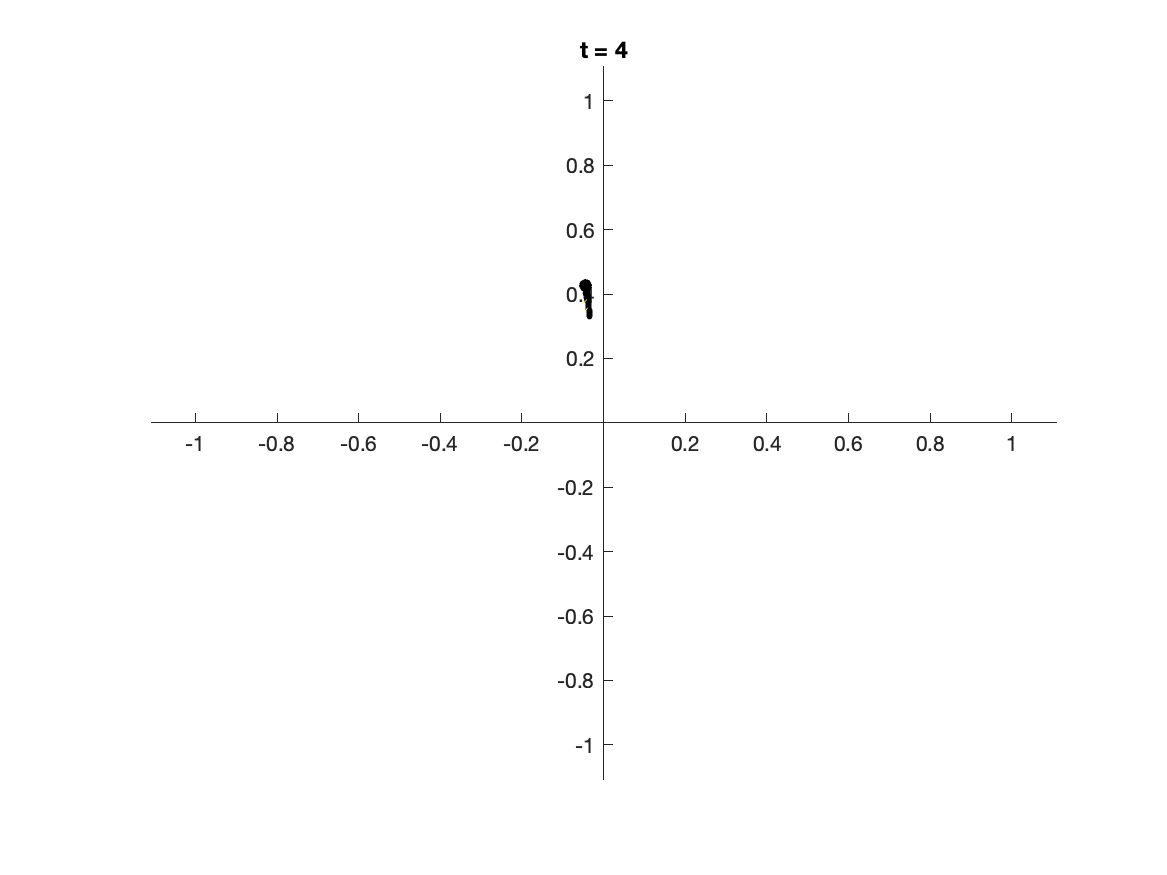}
\includegraphics[width=0.24\textwidth, trim = 2cm 0.5cm 2cm 0.5cm, clip=true]{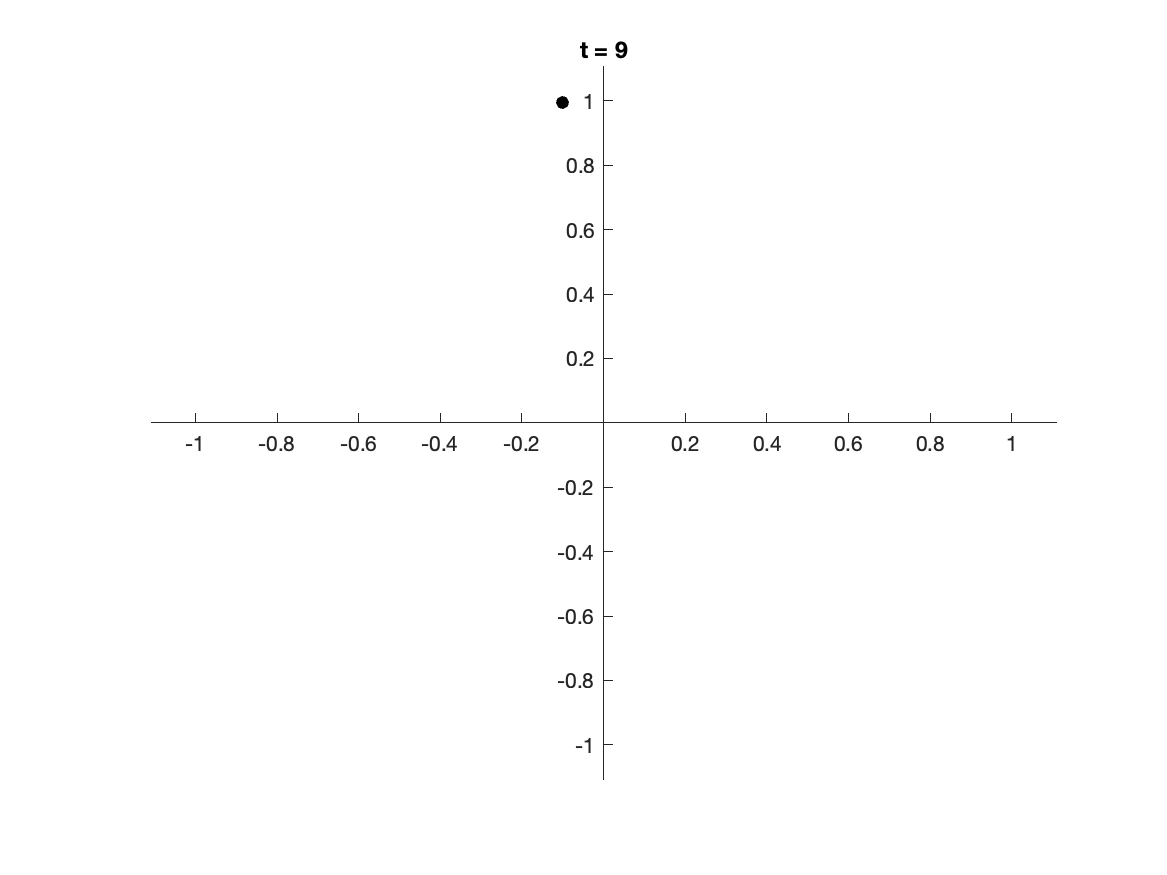}\\
\includegraphics[width=0.32\textwidth]{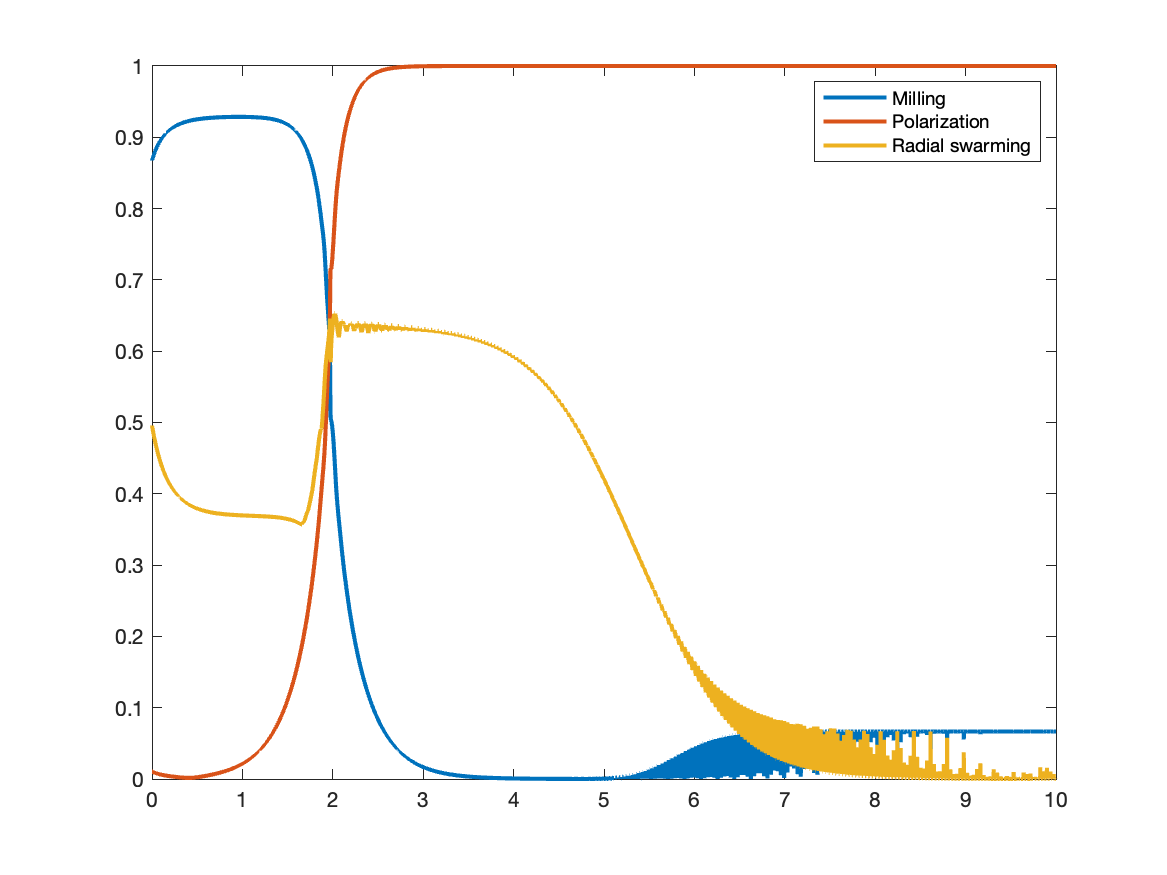}
\includegraphics[width=0.32\textwidth]{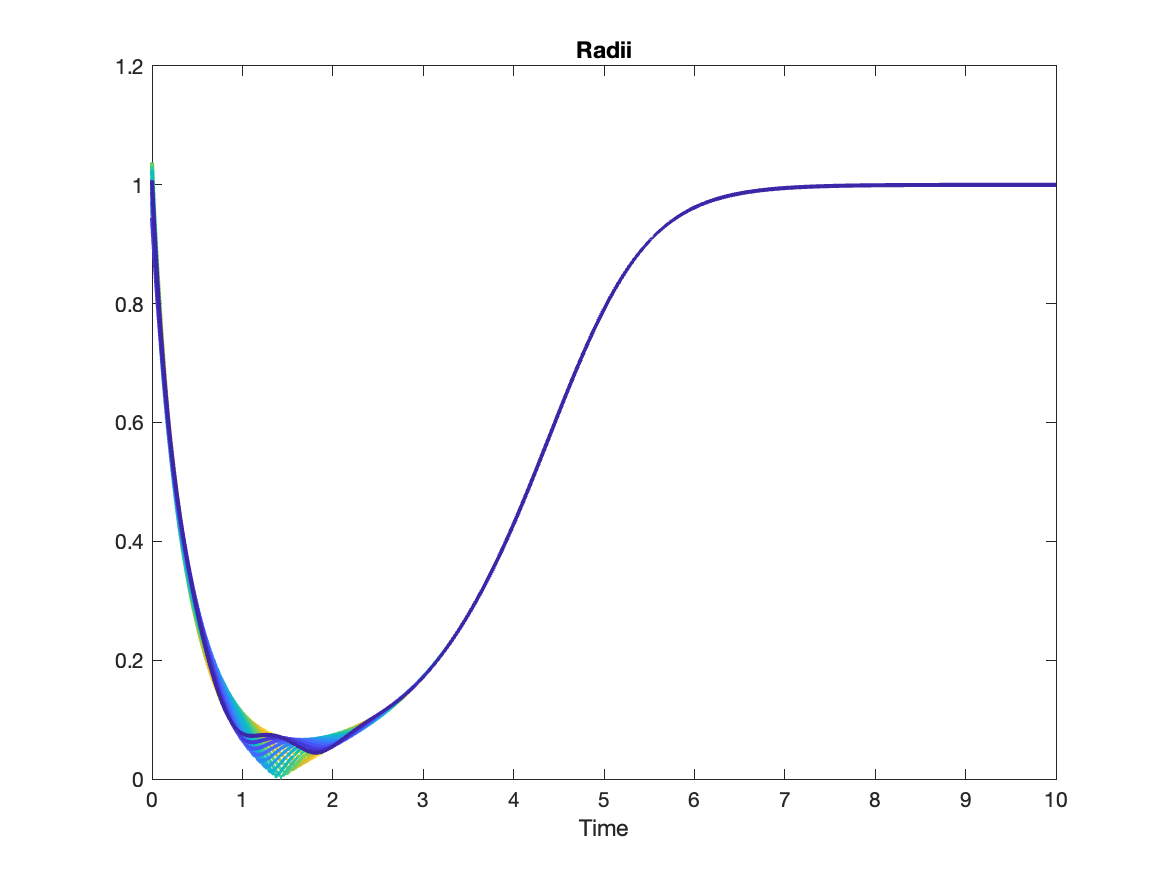}
\includegraphics[width=0.32\textwidth]{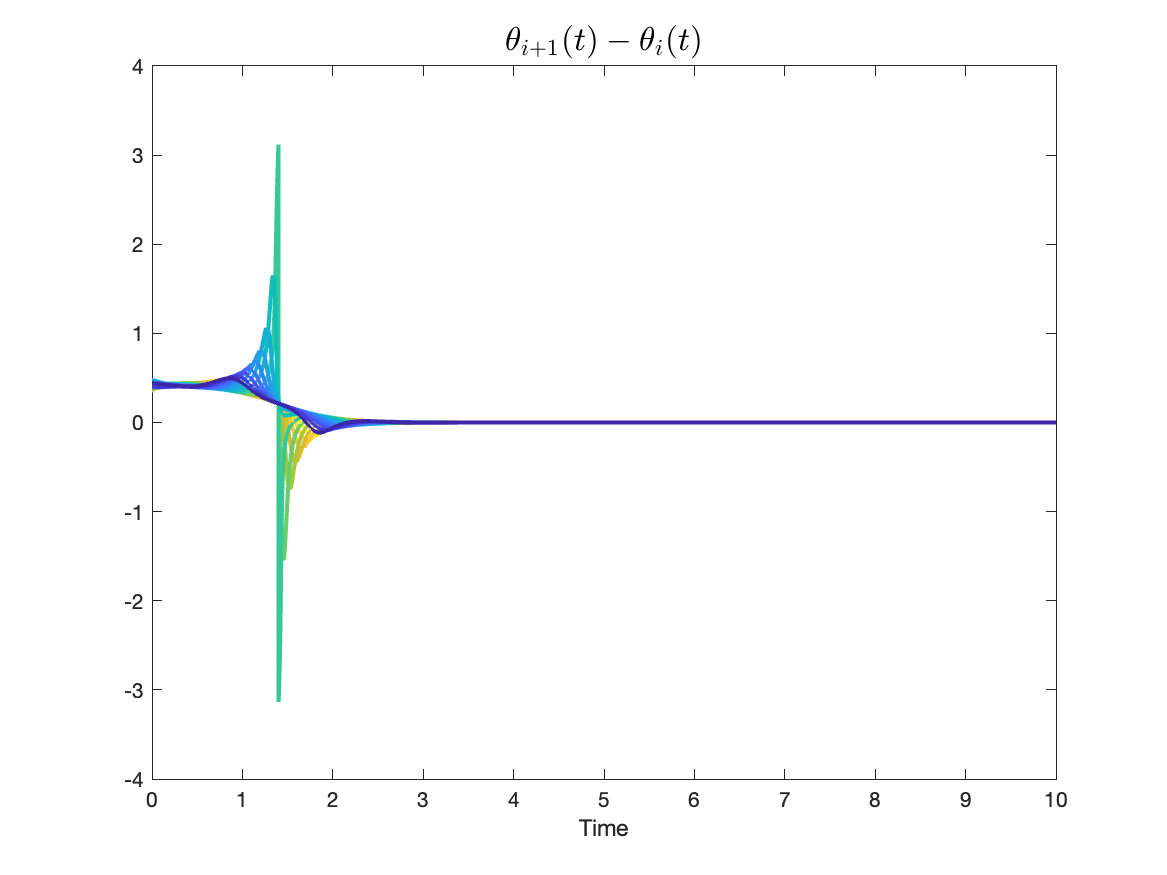}
\caption{Top row. Particles' trajectories at times $t=0.2$, $t=1$, $t=4$ and $t=9$. Each agent is represented in a different color, its last position is represented by the largest dot, and its previous positions by smaller dots. The black dot represents the trajectory of the system's center of mass.\\
Bottom row. Left: Evolution of the three order parameters $M$, $P$ and $S$. Center: Evolution of each agent's radius $r_j$. Right: Evolution of the $N$ angle differences $\theta_{j+1}-\theta_j$.   
}\label{Fig:Meta2}
\end{figure}

\section*{Acknowledgments}
DP's work has been supported by the State Research Agency of the Spanish Ministry of Science and FEDER-EU, grant  PID2022-137228OB-I00 (MICIU/AEI/10.13039/501100011033); by Consejer\'ia de Universidad, Investigaci\'on e Innovaci\'on (Junta de Andaluc\'ia), Grant QUAL21-011 of Modeling Nature Research Unit, and by Consejer\'ia de Universidad, Investigaci\'on e Innovaci\'on \& ERDF/EU Andalusia Program, grant C-EXP-265-UGR23. DNR has been partially supported by the Modeling Nature (MNat) Research Unit, project QUAL21-011. DNR has received funding from the National Science Centre, Poland, grant number 2023/50/A/ST1/00447. NPD has received support by the grant ANR-24-CE40-4644-01 funded by the Agence Nationale de la Recherche for the project ``FISH''.
This work was initiated during the sabbatical year of NPD at the University of Granada and Instituto Universitario de Matem\'aticas of Granada. She thanks both institutions for their hospitality. Project supported by a 2026 Leonardo Grant for Scientific Research and Cultural Creation from the BBVA Foundation. The BBVA Foundation accepts no responsibility for the opinions, statements and contents included in the project and/or the results thereof, which are entirely the responsibility of the authors.

\bibliographystyle{plain}
\bibliography{BibMilling}

\begin{thebibliography}{10}

\bibitem{Abelson1964}
Robert~P. Abelson.
\newblock Mathematical models of the distribution of attitudes under
  controversy.
\newblock In Norman Frederiksen and Harold Gulliksen, editors, {\em
  Contributions to Mathematical Psychology}, pages 142--160. Holt, Rinehart \&
  Winston, New York, 1964.

\bibitem{Abelson1967}
Robert~P. Abelson.
\newblock Mathematical models in social psychology.
\newblock In Leonard Berkowitz, editor, {\em Advances in Experimental Social
  Psychology}, volume~3, pages 1--54. Academic Press, 1967.

\bibitem{acebron2005}
Juan~A. Acebr{\'o}n, Luis~L. Bonilla, Conrad J.~P. Vicente, Felix Ritort, and
  Renato Spigler.
\newblock The {K}uramoto model: A simple paradigm for synchronization
  phenomena.
\newblock {\em Reviews of Modern Physics}, 77(1):137--185, 2005.

\bibitem{Aoki82}
Ichiro Aoki.
\newblock A simulation study on the schooling mechanism in fish.
\newblock {\em Nippon Suisan Gakkaishi}, 48(8), 1982.
\newblock \href{https://doi.org/10.2331/suisan.48.1081}{DOI}.

\bibitem{Ashraf16}
I.~Ashraf, R.~{{\smash{Godoy-Diana}}}, J.~Halloy, B.~Collignon, and
  B.~{Thiria}.
\newblock Synchronization and collective swimming patterns in fish
  ({\textit{hemigrammus bleheri}}).
\newblock {\em Journal of The Royal Society Interface}, 13(123):20160734, 2016.
\newblock \href{https://doi.org/10.1098/rsif.2016.0734}{DOI}.

\bibitem{Ballerini08}
Michele Ballerini, Nicola Cabibbo, Raphael Candelier, Andrea Cavagna, Evaristo
  Cisbani, Irene Giardina, Vivien Lecomte, Alberto Orlandi, Giorgio Parisi,
  Andrea Procaccini, et~al.
\newblock Interaction ruling animal collective behavior depends on topological
  rather than metric distance: Evidence from a field study.
\newblock {\em Proceedings of the national academy of sciences},
  105(4):1232--1237, 2008.

\bibitem{Berglund_2025}
Nils Berglund, Georgi~S Medvedev, and Gideon Simpson.
\newblock Metastability in the stochastic nearest-neighbour {K}uramoto model of
  coupled phase oscillators.
\newblock {\em Nonlinearity}, 38(9):095031, sep 2025.

\bibitem{B22}
Beno\^{\i}t {Bonnet-Weill}, Nastassia {{Pouradier Duteil}}, and Mario
  {Sigalotti}.
\newblock Consensus formation in first-order graphon models with time-varying
  topologies.
\newblock {\em Mathematical Models and Methods in Applied Sciences},
  32(11):2121--2188, 2022.
\newblock \href{https://doi.org/10.1142/S0218202522500518}{DOI}.

\bibitem{Calovi14}
Daniel~S Calovi, Ugo Lopez, Sandrine Ngo, Cl\'ement Sire, Hugues Chat\'e, and
  Guy Theraulaz.
\newblock Swarming, schooling, milling: phase diagram of a data-driven fish
  school model.
\newblock {\em New Journal of Physics}, 16(1):015026, jan 2014.
\newblock \href{https://dx.doi.org/10.1088/1367-2630/16/1/015026}{DOI}.

\bibitem{CarrilloChoiHauray2018}
Jos{\'e}~A Carrillo, Young-Pil Choi, Maxime Hauray, and Samir Salem.
\newblock Mean-field limit for collective behavior models with sharp
  sensitivity regions.
\newblock {\em Journal of the European Mathematical Society}, 21(1):121--161,
  2018.

\bibitem{CarrilloDOrsognaPanferov09}
Jos\'e~A. Carrillo, M.~R. D'Orsogna, and V.~Panferov.
\newblock Double milling in self-propelled swarms from kinetic theory.
\newblock {\em Kinetic and Related Models}, 2(2):363--378, 2009.
\newblock
  \href{https://www.aimsciences.org/article/id/51183b23-1022-4990-8f78-422b88c0aeb9}{DOI}.

\bibitem{Couzin02}
Iain~D. Couzin, Jens Krause, Richard James, Graeme~D. Ruxton, and Nigel~R.
  Franks.
\newblock Collective memory and spatial sorting in animal groups.
\newblock {\em Journal of Theoretical Biology}, 218(1):1--11, 2002.
\newblock
  \href{https://www.sciencedirect.com/science/article/pii/S0022519302930651}{DOI}.

\bibitem{CristianiFrascaPiccoli2011}
Emiliano Cristiani, Paolo Frasca, and Benedetto Piccoli.
\newblock Effects of anisotropic interactions on the structure of animal
  groups.
\newblock {\em Journal of mathematical biology}, 62(4):569--588, 2011.

\bibitem{CuckerSmale07}
F.~Cucker and S.~Smale.
\newblock Emergent behavior in flocks.
\newblock {\em IEEE Transactions on Automatic Control}, 52:852--862, 2007.

\bibitem{PhysRevLett.93.104101}
Hiroaki Daido and Kenji Nakanishi.
\newblock Aging transition and universal scaling in oscillator networks.
\newblock {\em Phys. Rev. Lett.}, 93:104101, Aug 2004.

\bibitem{PhysRevE.75.056206}
Hiroaki Daido and Kenji Nakanishi.
\newblock Aging and clustering in globally coupled oscillators.
\newblock {\em Phys. Rev. E}, 75:056206, May 2007.

\bibitem{Eldering2013}
Jaap Eldering.
\newblock Introduction.
\newblock {\em Atlantis Press}, pages 1--33, 2013.

\bibitem{Fiedler1973}
Miroslav Fiedler.
\newblock Algebraic connectivity of graphs.
\newblock {\em Czechoslovak Math. J.}, 23(2):298--305, 1973.

\bibitem{geshkovski2024dynamic}
Borjan Geshkovski, Hugo Koubbi, Yury Polyanskiy, and Philippe Rigollet.
\newblock Dynamic metastability in the self-attention model.
\newblock {\em arXiv preprint arXiv:2410.06833}, 2024.

\bibitem{geshkovski2025mathematical}
Borjan Geshkovski, Cyril Letrouit, Yury Polyanskiy, and Philippe Rigollet.
\newblock A mathematical perspective on {T}ransformers.
\newblock {\em Bulletin of the American Mathematical Society}, 62(3), 2025.

\bibitem{GMS2021}
Monica Goebel, Matthew~S. Mizuhara, and Sofia Stepanoff.
\newblock Stability of twisted states on lattices of kuramoto oscillators.
\newblock {\em Chaos: An Interdisciplinary Journal of Nonlinear Science},
  31(10):103106, 10 2021.

\bibitem{GuRohling19}
Changgui Gu, Ping Wang, Tongfeng Weng, Huijie Yang, and Jos Rohling.
\newblock Heterogeneity of neuronal properties determines the collective
  behavior of the neurons in the suprachiasmatic nucleus.
\newblock {\em Mathematical Biosciences and Engineering}, 16(4):1893--1913,
  2019.
\newblock
  \href{https://www.aimspress.com/article/doi/10.3934/mbe.2019092}{DOI}.

\bibitem{Ha_2010}
Seung-Yeal Ha, Taeyoung Ha, and Jong-Ho Kim.
\newblock Asymptotic dynamics for the {C}ucker--{S}male-type model with the
  {R}ayleigh friction.
\newblock {\em Journal of Physics A: Mathematical and Theoretical},
  43(31):315201, jul 2010.

\bibitem{HaLiu09}
Seung-Yeal Ha and Jian-Guo Liu.
\newblock {A simple proof of the {C}ucker-{S}male flocking dynamics and
  mean-field limit}.
\newblock {\em Communications in Mathematical Sciences}, 7(2):297 -- 325, 2009.

\bibitem{Hajnal1958}
John Hajnal.
\newblock Weak ergodicity in non-homogeneous {M}arkov chains.
\newblock {\em Proc. Cambridge Philos. Soc.}, 54(2):233--246, 1958.

\bibitem{Hoppensteadt1997}
Frank~C. Hoppensteadt and Eugene~M. Izhikevich.
\newblock Weakly connected oscillators.
\newblock {\em Springer New York}, pages 247--293, 1997.

\bibitem{kassabov2019}
Martin Kassabov, Steven~H. Strogatz, and Alex Townsend.
\newblock Universal bounds on the critical coupling strength for the {K}uramoto
  model.
\newblock {\em Physical Review E}, 99(2):022205, 2019.

\bibitem{KrauseGodinBrown1996}
Jens Krause, Jean-Guy~J Godin, and David Brown.
\newblock Phenotypic variability within and between fish shoals.
\newblock {\em Ecology}, 77(5):1586--1591, 1996.

\bibitem{Kuehn2015}
Christian Kuehn.
\newblock General {F}enichel {T}heory.
\newblock {\em Springer International Publishing}, pages 19--51, 2015.

\bibitem{Kuramoto75}
Y.~Kuramoto.
\newblock Self-entrainment of a population of coupled nonlinear oscillators.
\newblock In H.~Araki, editor, {\em International Symposium on Mathematical
  Problems in Theoretical Physics, Lecture Notes in Physics, Vol. 39,}, pages
  420--422, New York, NY, USA, 1975. Springer.

\bibitem{10.1007/BFb0013365}
Yoshiki Kuramoto.
\newblock Self-entrainment of a population of coupled non-linear oscillators.
\newblock In Huzihiro Araki, editor, {\em International Symposium on
  Mathematical Problems in Theoretical Physics}, pages 420--422, Berlin,
  Heidelberg, 1975. Springer Berlin Heidelberg.

\bibitem{Lafoux23}
Baptiste Lafoux, Jeanne Moscatelli, Ramiro {{\smash{Godoy-Diana}}}, and
  Benjamin {{Thiria}}.
\newblock Illuminance-tuned collective motion in fish.
\newblock {\em Communications Biology}, 6(1):585, 2023.

\bibitem{LRS}
Daniel Lear, David~N Reynolds, and Roman Shvydkoy.
\newblock Grassmannian reduction of {C}ucker-{S}male systems and dynamical
  opinion games.
\newblock {\em Discrete \& Continuous Dynamical Systems}, 41(12), 2021.

\bibitem{ling2019}
Shuyang Ling, Ruozhou Xu, and Afonso~S. Bandeira.
\newblock On the landscape of synchronization networks: A perspective from
  nonconvex optimization.
\newblock {\em SIAM Journal on Optimization}, 29(3):1879--1907, 2019.

\bibitem{lu2020}
Jianfeng Lu and Stefan Steinerberger.
\newblock Synchronization of {K}uramoto oscillators in dense networks.
\newblock {\em Nonlinearity}, 33(12):6677--6696, 2020.

\bibitem{McQuadePiccoliPouradierDuteil19}
Sean McQuade, Benedetto Piccoli, and Nastassia {Pouradier Duteil}.
\newblock {Social dynamics models with time-varying influence}.
\newblock {\em Math. Models Methods Appl. Sci.}, 29(4):681--716, 2019.

\bibitem{Med2}
Georgi~S. Medvedev and Matthew~S. Mizuhara.
\newblock The {K}uramoto model on the {S}ierpinski {G}asket ii: {T}wisted
  {S}tates.
\newblock {\em Journal of Nonlinear Science}, 36(4):70, 2026.

\bibitem{MedTang2015}
Georgi~S. Medvedev and Xuezhi Tang.
\newblock Stability of twisted states in the {K}uramoto model on {C}ayley and
  random graphs.
\newblock {\em Journal of Nonlinear Science}, 25(6):1169--1208, 2015.

\bibitem{millan2025synchronizationcoupledstuartlandauoscillators}
Ana~P Mill\'an, David Poyato, David~N Reynolds, and Francesco Tudisco.
\newblock Synchronization of coupled {Stuart-Landau} oscillators: How
  heterogeneity can facilitate synchronization.
\newblock {\em arXiv.2510.05243}, 2025.

\bibitem{Moreau2005}
Luc Moreau.
\newblock Stability of multiagent systems with time-dependent communication
  links.
\newblock {\em IEEE Transactions on automatic control}, 50(2):169--182, 2005.

\bibitem{MT14}
S.~Motsch and E.~Tadmor.
\newblock Heterophilious dynamics enhances consensus.
\newblock {\em SIAM Review}, 56(4):577--621, 2014.

\bibitem{RevModPhys.90.031001}
Miguel~A. Mu\~noz.
\newblock Colloquium: Criticality and dynamical scaling in living systems.
\newblock {\em Rev. Mod. Phys.}, 90:031001, Jul 2018.

\bibitem{OlfatiSaberFaxMurray2007}
Reza Olfati-Saber, J.~Alex Fax, and Richard~M. Murray.
\newblock Consensus and cooperation in networked multi-agent systems.
\newblock {\em Proc. IEEE}, 95(1):215--233, 2007.

\bibitem{OlfatiSaberMurray2004}
Reza Olfati-Saber and Richard~M. Murray.
\newblock Consensus problems in networks of agents with switching topology and
  time-delays.
\newblock {\em IEEE Trans. Automat. Control}, 49(9):1520--1533, 2004.

\bibitem{PANTELEY2015645}
Elena Panteley, Antonio Loria, and Ali~El Ati.
\newblock On the stability and robustness of {S}tuart-{L}andau oscillators.
\newblock {\em IFAC-PapersOnLine}, 48(11):645--650, 2015.

\bibitem{Panteley01022020}
Elena Panteley, Antonio Lor{\'\i}a, and Ali El-Ati.
\newblock Practical dynamic consensus of {S}tuart--{L}andau oscillators over
  heterogeneous networks.
\newblock {\em International Journal of Control}, 93(2):261--273, 2020.

\bibitem{PAD2024}
Adri{\'a}n Ponce-Alvarez and Gustavo Deco.
\newblock The hopf whole-brain model and its linear approximation.
\newblock {\em Scientific Reports}, 14(1):2615, 2024.

\bibitem{PROSKURNIKOV2018166}
Anton~V. Proskurnikov and Roberto Tempo.
\newblock A tutorial on modeling and analysis of dynamic social networks. part
  ii.
\newblock {\em Annual Reviews in Control}, 45:166--190, 2018.

\bibitem{HegselmannKrause02}
Hegselmann Rainer and Ulrich Krause.
\newblock Opinion dynamics and bounded confidence: Models, analysis and
  simulation.
\newblock {\em Journal of Artificial Societies and Social Simulation}, 5(3),
  2002.

\bibitem{reynolds2025consensuspolarizationoptimizationmean}
David~N. Reynolds and Pedro~J. Torres.
\newblock Consensus, polarization, and optimization of the mean value in a
  nonlinear model of opinion dynamics.
\newblock {\em arXiv.2509.14918}, 2025.

\bibitem{Reynolds_2025}
David~N Reynolds and Francesco Tudisco.
\newblock Unique nash equilibrium of a nonlinear model of opinion dynamics on
  networks with friction-inspired stubbornness.
\newblock {\em Nonlinearity}, 38(5):055027, may 2025.

\bibitem{rodrigues2016}
Francisco~A. Rodrigues, Thomas K. D.~M. Peron, Peng Ji, and J{\"u}rgen Kurths.
\newblock The {K}uramoto model in complex networks.
\newblock {\em Physics Reports}, 610:1--98, 2016.

\bibitem{Sahoo_2019}
Samir Sahoo, Vaibhav Varshney, Awadhesh Prasad, and Ram Ramaswamy.
\newblock Ageing in mixed populations of {S}tuart--{L}andau oscillators: the
  role of diversity.
\newblock {\em Journal of Physics A: Mathematical and Theoretical},
  52(46):464001, oct 2019.

\bibitem{10.1371/journal.pcbi.1010662}
Yonatan Sanz~Perl, Anira Escrichs, Enzo Tagliazucchi, Morten~L. Kringelbach,
  and Gustavo Deco.
\newblock Strength-dependent perturbation of whole-brain model working in
  different regimes reveals the role of fluctuations in brain dynamics.
\newblock {\em PLOS Computational Biology}, 18(11):1--32, 11 2022.

\bibitem{Sclosa_2024}
Davide Sclosa.
\newblock Completely degenerate equilibria of the {K}uramoto model on networks.
\newblock {\em Nonlinearity}, 37(9):095021, aug 2024.

\bibitem{Seneta2006}
Eugene Seneta.
\newblock {\em Non-negative Matrices and {M}arkov Chains}.
\newblock Springer Series in Statistics. Springer, New York, 2nd edition, 2006.
\newblock Revised printing of the 1981 edition.

\bibitem{strogatz2000}
Steven~H. Strogatz.
\newblock From {K}uramoto to {C}rawford: exploring the onset of synchronization
  in populations of coupled oscillators.
\newblock {\em Physica D}, 143(1-4):1--20, 2000.

\bibitem{taylor2012}
Richard Taylor.
\newblock There is no non-zero stable fixed point for dense networks in the
  homogeneous {K}uramoto model.
\newblock {\em Journal of Physics A: Mathematical and Theoretical},
  45(5):055102, 2012.

\bibitem{Thome_2026}
Nicolas Thom\'e, Yukiteru Murakami, and Katharina Krischer.
\newblock Synchronized bimodal amplitude patterns in heterogeneous oscillatory
  media-experiment and theory.
\newblock {\em New Journal of Physics}, 28(2):023902, feb 2026.

\bibitem{townsend2017}
Alex Townsend, Mike Stillman, and Steven~H. Strogatz.
\newblock Dense networks that do not synchronize and sparse ones that do.
\newblock {\em Chaos}, 27(7):073119, 2017.

\bibitem{NMDDSSCKLHG2023}
Anne~M. van Nifterick, Danique Mulder, Denise~J. Duineveld, Marina Diachenko,
  Philip Scheltens, Cornelis~J. Stam, Ronald~E. van Kesteren, Klaus
  Linkenkaer-Hansen, Arjan Hillebrand, and Alida~A. Gouw.
\newblock Resting-state oscillations reveal disturbed excitation--inhibition
  ratio in {A}lzheimer's disease patients.
\newblock {\em Scientific Reports}, 13(1):7419, 2023.

\bibitem{Vicsek95}
Tam{\'a}s Vicsek, Andr{\'a}s Czir{\'o}k, Eshel Ben-Jacob, Inon Cohen, and Ofer
  Shochet.
\newblock Novel type of phase transition in a system of self-driven particles.
\newblock {\em Physical review letters}, 75(6):1226, 1995.

\bibitem{WileyStrogatzGirvan06}
Daniel~A. Wiley, Steven~H. Strogatz, and Michelle Girvan.
\newblock {The size of the sync basin}.
\newblock {\em Chaos: An Interdisciplinary Journal of Nonlinear Science},
  16(1), 03 2006.
\newblock 015103.

\bibitem{yoneda2021}
Takashi Yoneda, Yuki Tatsukawa, and Hiromichi Ohta.
\newblock Improved lower bounds for synchronization thresholds in the
  {K}uramoto model.
\newblock {\em Nonlinearity}, 34(6):4107--4125, 2021.

\bibitem{10.1145/3774904.3792675}
Kaicheng Zhang, David~N. Reynolds, Piero Deidda, and Francesco Tudisco.
\newblock {S}tuart-{L}andau oscillatory graph neural network.
\newblock In {\em Proceedings of the ACM Web Conference 2026}, WWW '26, pages
  1469--1480, New York, NY, USA, 2026. Association for Computing Machinery.

\end{thebibliography}

\end{document}